\documentclass[twoside,11pt,preprint]{article}
\usepackage{blindtext}

\usepackage{jmlr2e}
\usepackage{graphicx} 
\usepackage{amsmath}
\usepackage{graphicx}
\usepackage{listings}
\usepackage{caption}
\usepackage{lineno}
\usepackage{bm}
\usepackage{bbm}
\usepackage{subeqnarray}
\usepackage{amsfonts}
\usepackage{amssymb}
\usepackage{amsbsy}
\usepackage{amsmath}
\usepackage{algorithmic}
\usepackage{algorithm}
\usepackage{enumitem}
\usepackage{xcolor}
\usepackage{dsfont}
\usepackage{multirow}
\usepackage{makecell}
\usepackage{tabularx}
\usepackage{pgfplots}
\usepackage{subcaption}
\usepackage{wrapfig}
\usepackage{booktabs}
\usepackage{pgfplots}
\usepackage{pgfplotstable}
\usepackage{graphicx}
\usepackage{threeparttable}

\allowdisplaybreaks

\def\blue#1{\textcolor{blue}{#1}}

\newtheorem{assumption}{Assumption}
\newcommand{\E}{\mathbb{E}}
\newcommand{\be}{\begin{equation}}
\newcommand{\ee}{\end{equation}}
\newcommand{\ie}{{\em i.e.\xspace}}

\usepackage{lastpage}

\ShortHeadings{Optimal Complexity in Byzantine-Robust Distributed Stochastic Optimization}{Shi, Peng, Yuan, Wang, and Ling}
\firstpageno{1}
\begin{document}

\title{Optimal Complexity in Byzantine-Robust Distributed Stochastic Optimization with Data Heterogeneity}

\author{\name Qiankun Shi \email shiqk@mail2.sysu.edu.cn \\
       \addr School of Computer Science and Engineering\\
       Sun Yat-Sen University\\
       Guangzhou, China\\
       Pengcheng Laboratory\\
       Shenzhen, China
       \AND
       \name Jie Peng \email pengj95@mail2.sysu.edu.cn \\
       \addr School of Computer Science and Engineering\\
       Sun Yat-Sen University\\
       Guangzhou, China
       \AND
       \name Kun Yuan \email kunyuan@pku.edu.cn \\
       \addr Center for Machine Learning Research\\
       Peking University\\
       Beijing, China
       \AND
       \name Xiao Wang \email wxucas@outlook.com \\
       \addr School of Computer Science and Engineering\\
       Sun Yat-Sen University\\
       Guangzhou, China
       \AND
       \name Qing Ling \email lingqing556@mail.sysu.edu.cn \\
       \addr School of Computer Science and Engineering\\
       Sun Yat-Sen University\\
       Guangzhou, China
       }

\editor{My editor}

\maketitle

\begin{abstract}
     In this paper, we establish tight lower bounds for Byzantine-robust distributed first-order stochastic optimization methods in both strongly convex and non-convex stochastic optimization. We reveal that when the distributed nodes have heterogeneous data, the convergence error comprises two components: a non-vanishing Byzantine error and a vanishing optimization error. We establish the lower bounds on the Byzantine error and on the minimum number of queries to a stochastic gradient oracle required to achieve an arbitrarily small optimization error. Nevertheless, we identify significant discrepancies between our established lower bounds and the existing upper bounds. To fill this gap, we leverage the techniques of Nesterov's acceleration and variance reduction to develop novel Byzantine-robust distributed stochastic optimization methods that provably match these lower bounds, up to logarithmic factors, implying that our established lower bounds are tight.
\end{abstract}


\section{Introduction}

Large-scale stochastic optimization has emerged as an indispensable tool in machine learning, particularly in the training of large foundation models \citep{brown2020language, openai2024gpt4technicalreport}. Solving such large and intricate problems poses formidable challenges, often requiring days or months to complete. Consequently, it is imperative to expedite large-scale stochastic optimization through distributed methods. The appeal of distributed stochastic optimization lies in its potential to harness the combined power of distributed computing nodes to handle the size of modern datasets. In this paper, we explore a server-based distributed architecture, where the nodes communicate with a server that coordinates their activities and manages the distribution and aggregation of computational tasks.

Nevertheless, the promise of distributed stochastic optimization is underpinned by the assumption of a trustworthy system, in which all nodes adhere to the prescribed computational protocol. The introduction of Byzantine faults/attacks to a fraction of the nodes, i.e., arbitrary deviations from expected behaviors, potentially due to node malfunctions \citep{zhang2020blockchain, xiao2023bce}, malicious manipulations \citep{attias2022improved,liu2024badsampler}, or poisoned data \citep{mahloujifar2019universal, lewis2023attacks}, poses a significant challenge to distributed stochastic optimization methods and leads to incorrect solutions or even total failures.
The complexity of defending against Byzantine attacks is further compounded in scenarios involving heterogeneous data, where the nodes may possess non-identically distri- buted data samples such that differentiating Byzantine attacks and honest behaviors be- comes highly nontrivial \citep{li2019rsa,karimireddy2020byzantine}. In this paper, we devote to investigating the optimal complexity in Byzantine-robust distributed stochastic optimiza- tion with data heterogeneity.



The basic concept of Byzantine robustness originates from the seminal work of \citep{lamport1982byzantine}, aiming at achieving consensus in a distributed system where some nodes may act maliciously or fail arbitrarily. It is then extended to the area of distributed deterministic optimization \citep{su2016fault, chen2017distributed}. In recent years, Byzantine robustness in distributed stochastic optimization has attracted immense research interest due to the popularity of large-scale machine learning \citep{guerraoui2023byzantine,ye2025generalization}.
The pursuit of Byzantine-robust methods has led to the development of diverse strategies aimed at fortifying distributed stochastic optimization against the attacks from Byzantine nodes. The majority of these strategies rely on robust aggregators, with which the server either removes suspicious stochastic gradients prior to averaging \citep{chen2018draco,alistarh2018byzantine,xie2019zeno} or uses statistically robust estimators, such as trimmed mean \citep{yin2018byzantine}, median \citep{yin2018byzantine}, geometric median \citep{wu2020federated}, to name a few.


In scenarios with heterogeneous data, the performance of the aforementioned methods shall be significantly degraded. When data heterogeneity appears among the nodes, their local stochastic gradients exhibit varying statistical properties, diminishing the effectiveness of the robust aggregators that utilize statistical similarity to distinguish the Byzantine nodes from the rest honest nodes and leading to unavoidable convergence errors \citep{li2019rsa,wu2020federated,el2021collaborative,karimireddy2020byzantine,guerraoui2023byzantine,peng2025mean}. In light of this issue, advanced robust aggregators that are relatively insensitive to data heterogeneity, such as bucketing \citep{karimireddy2020byzantine} and nearest neighbor mixing \citep{allouah2023fixing}, has been proposed.
%


While the existing methods enjoy theoretical guarantees and/or empirical successes, the performance limits of Byzantine-robust distributed stochastic optimization methods have not been fully clarified. This paper aims to reveal the performance limits by establishing the optimal complexity in Byzantine-robust distributed stochastic optimization. We focus on first-order, synchronous methods; extensions to zeroth-order \citep{egger2025byzantine}, second-order \citep{cao2020distributed,ghosh2020distributed,koushkbaghi2024byzantine} and asynchronous methods \citep{el2021collaborative,yang2023buffered} will be our future works.


\subsection{Problem setup}


We consider a distributed system comprising a server and $n$ nodes. The sets of the honest and Byzantine nodes are denoted by $\mathcal{H}$ and $\mathcal{B}$, respectively. Note that the identities of the honest and Byzantine nodes are unknown to the server. We assume $|\mathcal{B}| < |\mathcal{H}|$ throughout this paper. The purpose of Byzantine-robust distributed stochastic optimization is to find a minimizer to
\begin{align}
\label{prob-general}
\min_{x \in {\mathbb{R}}^d}\ f(x) = \frac{1}{|\mathcal{H}|}\sum_{i \in \mathcal{H}} f_i(x), \quad \text{where} \quad f_i(x) = \mathbb{E}_{\xi \sim \mathcal{D}_i} [F(x,\xi)].
\end{align}
Here, \(\xi\) represents a random variable following the local data distribution \(\mathcal{D}_i\) of node $i$, and \(F: \mathbb{R}^d \times \mathbb{R}^q \rightarrow \mathbb{R}\) is a Borel measurable function. Each  function \(f_i\) is accessible locally by node \(i\) and is assumed to be smooth. It is important to note that data heterogeneity typically exists; that is, the local data distributions \(\{\mathcal{D}_i\}_{i=1}^n\) differ across the nodes.

\subsection{Fundamental open questions}\label{ssec:questions}


The convergence error of a Byzantine-robust distributed stochastic optimization method, after taking $k$ oracle queries of the stochastic gradients, typically comprises two components: Byzantine error and optimization error. Specifically:
\begin{align}
\textbf{Convergence error} = \textbf{Byzantine error + \textbf{Optimization error ${\boldsymbol{\epsilon}}$}}.
\end{align}
The Byzantine error is non-vanishing; it persists throughout the entire optimization process, even as the number of oracle queries $k$ approaches infinity.
Conversely, the optimization error $\epsilon$ typically decreases with the number of oracle queries $k$.
The interplay between these two error components characterizes the overall performance of Byzantine-robust distributed stochastic optimization methods, with the goal of simultaneously minimizing both errors to achieve certified Byzantine robustness and fast convergence rate. This dual objective presents a fundamental challenge in the design and analysis of Byzantine-robust distributed stochastic optimization methods. Tackling this challenge requires answering the following two fundamental questions:


\begin{itemize}
\item[Q1.] \textit{What is the smallest Byzantine error that any Byzantine-robust distributed stochastic optimization methods can achieve?}

\item[Q2.] \textit{What is the optimal convergence rate at which the optimization error $\epsilon$ decreases to zero for any Byzantine-robust distributed stochastic optimization methods, or equivalently, what is the minimum number of queries to a stochastic gradient oracle required to attain an arbitrarily small $\epsilon$?}




\end{itemize}


In this paper, we answer these two open questions via establishing tight lower bounds of the Byzantine error and the oracle query complexity. Note that several pioneering works have already shed light on these two open questions. The work of \citep{alistarh2018byzantine} provides valuable insights into the tight lower bound of the oracle query complexity in strongly convex optimization. However, the analysis is confined to homogeneous data distribution and does not account for the lower bound of the Byzantine error. On the other hand, the work of \citep{karimireddy2020byzantine} addresses the tight lower bound of the Byzantine error in non-convex optimization and heterogeneous data distribution, but does not explore the oracle query complexity.




\subsection{Main results and contributions}


In this paper, we provide a comprehensive analysis that establishes the lower bounds of the Byzantine error and the oracle query complexity in Byzantine-robust distributed stochastic optimization. We also validate the tightness of these lower bounds through developing methods that can attain optimal Byzantine robustness and optimal convergence rate simultaneously. In particular, our contributions are:
\begin{itemize}
    \item We establish the lower bounds on the Byzantine error for Byzantine-robust distributed methods in both strongly convex and non-convex stochastic optimization.


    \item We establish the lower bounds on the convergence rate at which the optimization error $\epsilon$ approaches zero for Byzantine-robust distributed methods in both strongly convex and non-convex stochastic optimization. Leveraging these results, we reveal the lower bounds on the minimum number of queries to a stochastic gradient oracle required to achieve an arbitrarily small optimization error $\epsilon$.


    \item We identify significant discrepancies between our established lower bounds and the Byzantine robustness and convergence rates reported in the existing works. To fill this gap, we propose novel Byzantine-robust distributed stochastic optimization methods that provably match these lower bounds, up to at most logarithmic factors. This fact implies that our established lower bounds are tight, and our proposed methods attain the optimal Byzantine robustness and the optimal convergence rates simultaneously.

\end{itemize}

The bounds established in this paper, along with those from the existing state-of-the-art Byzantine-robust distributed stochastic optimization methods, are summarized in Tables \ref{table:sc} and \ref{table:nc}. From the lower bound perspective, our work simultaneously explores the Byzantine error and the oracle query complexity, whereas the existing studies address only one of these two aspects \citep{alistarh2018byzantine,karimireddy2021learning}.
From the upper bound perspective, our proposed methods match the lower bounds, achieving superior Byzantine robustness while demonstrating theoretically faster convergence rates.

\begin{table}[t]
\caption{\small Lower and upper bounds of finding $x$ such that $\E[\|\nabla f(x)\|]$ is no larger than the Byzantine error plus the optimization error $\epsilon$ in strongly convex stochastic optimization. Notations: $n$ is the number of nodes; $\delta \in [0,1/2)$ is the estimated fraction of Byzantine nodes that is no smaller than the true fraction of Byzantine nodes; \( \sigma^2 \) bounds the variance of the stochastic gradient estimates (see Assumption \ref{ass:u}), with \( \sigma^2 = 0 \) corresponding to deterministic optimization; \(\zeta^2\) bounds the stochastic gradient dissimilarity between the nodes (see Assumption \ref{ass:i}), with \(\zeta^2 = 0\) corresponding to homogeneous data distribution;
$\rho \ge 0$ is the coefficient to measure the robustness of an aggregator (see Definition \ref{d:agg});
$R = \|x^0 - \arg\min_x f(x)\|$; $L$ is the Lipschitz smoothness constant; $\mu$ is the strongly convex constant; $\kappa:={L}/{\mu}$ is the condition number; $\tilde \Omega(\cdot)$ and $\tilde O(\cdot)$ hide constants and logarithmic factors.
}
    \centering
    \begin{threeparttable}  
    \renewcommand{\arraystretch}{1}
    \begin{tabularx}{\textwidth}{Xccc}
    \toprule
      & \textbf{Byzantine error} & \textbf{Oracle query complexity}   & \textbf{Reference}  \\
     \midrule
     \multirow{4}{*}{\rotatebox{90}{\thead{\bf Lower bound}}} & \multirow{2}{*}{/} & \multirow{2}{*}{$\Omega\left(\frac{\delta^2 \sigma^2}{\epsilon^2} + \frac{\sigma^2}{n\epsilon^2}\right)$} & \multirow{2}{*}{\cite{alistarh2018byzantine}} \\
    \multirow{4}{*}{}   &\multirow{2}{*}{}  &\multirow{2}{*}{}  & \multirow{2}{*}{}\vspace{0.15cm}\\
     \multirow{4}{*}{} & \multirow{2}{*}{\blue{$\Omega\left(\rho^{1/2}\delta^{1/2}\zeta\right)$}} & \multirow{2}{*}{\blue{$\tilde \Omega\left(\frac{\rho\delta \sigma^2}{\epsilon^2} + \frac{\sigma^2}{(1-\delta)n\epsilon^2}+ \kappa^{1/2}\right)$}} & \multirow{2}{*}{\blue{Theorem \ref{thm:flb}}} \\
    \multirow{4}{*}{}   &\multirow{2}{*}{}  &\multirow{2}{*}{}  & \multirow{2}{*}{}\vspace{0.15cm}\\
    \midrule
     \multirow{10}{*}{\rotatebox{90}{\thead{\bf Upper bound}}} & \multirow{2}{*}{/} & \multirow{2}{*}{$\tilde O \left(\kappa + \frac{\kappa\delta^2 \sigma^2}{\epsilon^2} + \frac{\kappa\sigma^2}{n\epsilon^2}\right)$} & \multirow{2}{*}{\cite{alistarh2018byzantine}$^\dagger$} \\
    \multirow{10}{*}{} &\multirow{2}{*}{}   &\multirow{2}{*}{}  & \multirow{2}{*}{}\\
    \multirow{10}{*}{} & \multirow{2}{*}{$O\left(L\kappa \delta n \zeta\right)$} & \multirow{2}{*}{$O \left(\frac{LR^2}{\epsilon^2}+\frac{\kappa}{\mu}\frac{(\delta^2+(1-\delta)^2)n^2\zeta^2+(1-\delta)n\sigma^2}{\epsilon^2} \right)$} & \multirow{2}{*}{\cite{li2019rsa}$^\dagger$} \\
    \multirow{10}{*}{} &\multirow{2}{*}{}   &\multirow{2}{*}{}  & \multirow{2}{*}{}\\
    \multirow{10}{*}{} & \multirow{2}{*}{$O\left(\kappa^2 \zeta\right)$} & \multirow{2}{*}{$\tilde O \left(\kappa+\frac{\kappa^5(1+\delta)\sigma^2}{(1-\delta)n\epsilon^2} + \frac{\kappa^5\delta\sigma^2}{\epsilon^2}\right)$} & \multirow{2}{*}{\cite{data2021byzantine}$^\dagger$} \\
    \multirow{10}{*}{} &\multirow{2}{*}{}   &\multirow{2}{*}{}  & \multirow{2}{*}{}\\
    \multirow{10}{*}{} & \multirow{2}{*}{$O\left(\frac{\zeta}{1-2\delta}\right)$} & \multirow{2}{*}{$\tilde O \left(\kappa + \frac{\kappa L \sigma^2}{(1-2\delta)^2\epsilon^2}  \right)$} & \multirow{2}{*}{\cite{pillutla2022robust}$^\dagger$} \\
    \multirow{10}{*}{} &\multirow{2}{*}{}   &\multirow{2}{*}{}  & \multirow{2}{*}{}\\
    \multirow{10}{*}{} & \multirow{2}{*}{\blue{$O\left(\kappa^{1/2}\rho^{1/2}\delta^{1/2}\zeta\right)$}} & \multirow{2}{*}{\blue{$\tilde O \left(\frac{\kappa^{3/2}\rho\delta \sigma^2}{\epsilon^2} + \frac{\kappa^{3/2}\sigma^2}{(1-\delta)n\epsilon^2}+ \kappa^{1/2}\right)$}} & \multirow{2}{*}{\blue{Algorithm \ref{algo:restart} (Thm. \ref{thm:sc-restart})}} \\
    \multirow{10}{*}{} &\multirow{2}{*}{}   &\multirow{2}{*}{}  & \multirow{2}{*}{}\\
    \bottomrule
\end{tabularx}
    \begin{tablenotes}
        \scriptsize
        \item $\dagger$ The bounds are established for specific robust aggregators.
    \end{tablenotes}
    \end{threeparttable}  
    \label{table:sc}
\end{table}

\begin{table}[t]
    \centering
    \begin{threeparttable}
    \renewcommand{\arraystretch}{1}
    \caption{\small Lower and upper bounds of finding $x$ such that $\E[\|\nabla f(x)\|]$ is no larger than the Byzantine error plus the optimization error $\epsilon$ in non-convex stochastic optimization. Notations not appeared in Table \ref{table:sc}: $m$ is the batch size; $\Delta := f(x^0) - f(x^\star)$;  $c_1=\frac{\rho\delta\sigma^4}{(1-\delta)n\epsilon^4}+\frac{\sigma^2}{(1-\delta)n\epsilon^2}$; $c_2 = \frac{(1-\delta)^{1/3}L^{1/3}\Delta^{1/3}\sigma^{2/3}}{(1+(1-\delta)\rho\delta n)^{1/3}n^{1/3}\epsilon^{4/3}}$.}
    \begin{tabularx}{\textwidth}{X c c c}
    \toprule
      & \centering\textbf{Byzantine error} & \textbf{Oracle query complexity}   & \textbf{References}  \\
     \midrule
     \multirow{4}{*}{\rotatebox{90}{\thead{\bf Lower bound}}} & \centering\multirow{2}{*}{$\Omega\left(\delta^{1/2}\zeta\right)$} & \multirow{2}{*}{/} & \multirow{2}{*}{\cite{karimireddy2020byzantine}} \\
    \multirow{4}{*}{}   &\centering\multirow{2}{*}{}  &\multirow{2}{*}{}  & \multirow{2}{*}{}\vspace{0.15cm}\\
     \multirow{4}{*}{} & \centering\multirow{2}{*}{\blue{$\Omega\left(\rho^{1/2}\delta^{1/2}\zeta\right)$}} & \multirow{2}{*}{\blue{$\Omega \left(\frac{L\Delta\rho\delta \sigma^2}{\epsilon^4} + \frac{L\Delta\sigma^2}{(1-\delta)n\epsilon^4} + \frac{L\Delta}{\epsilon^2}\right)$}} & \multirow{2}{*}{\blue{Theorem \ref{thm:flb}}} \\
    \multirow{4}{*}{}   &\multirow{2}{*}{}  &\multirow{2}{*}{}  & \multirow{2}{*}{}\vspace{0.15cm}\\
    \midrule
    \multirow{10}{*}{\rotatebox{90}{\thead{\bf Upper bound}}} &
    \centering\multirow{2}{*}{$O\left((1+\delta)^{1/2} \zeta\right)$} & \multirow{2}{*}{$ O \left(\frac{L^2R^2}{\epsilon^2}\left(1 + \frac{(1+\delta)\sigma^2}{(1-\delta)n\epsilon^2}+ \frac{\delta\sigma^2}{\epsilon^2}\right)  \right)$} & \multirow{2}{*}{\cite{data2021byzantine}$^\dagger$} \\
    \multirow{10}{*}{} &\multirow{2}{*}{}   &\multirow{2}{*}{}  & \multirow{2}{*}{}\\
    \multirow{10}{*}{} & \centering\multirow{2}{*}{$O\left(\rho^{1/2}\delta^{1/2}\zeta\right)$} & \multirow{2}{*}{$O \left(\frac{L\Delta\rho\delta\sigma^2}{\epsilon^4}+\frac{L\Delta\sigma^2}{(1-\delta)n\epsilon^4}+\frac{L\Delta}{\epsilon^2}+c_1\right)$}   & \multirow{2}{*}{\cite{karimireddy2020byzantine}} \\
    \multirow{10}{*}{} &\multirow{2}{*}{}   &\multirow{2}{*}{}  & \multirow{2}{*}{}\\
    \multirow{10}{*}{} & \centering\multirow{2}{*}{\blue{$O\left(\rho^{1/2}\delta^{1/2}\zeta\right)$}} & \multirow{2}{*}{\blue{$O \left(\frac{L\Delta\rho\delta\sigma^2}{\epsilon^4}+\frac{L\Delta\sigma^2}{(1-\delta)n\epsilon^4}+\frac{L\Delta}{\epsilon^2}+c_1\right)$}}   & \multirow{2}{*}{\blue{Algorithm \ref{subalgo} (Cor. \ref{cor:nc1})}} \\
    \multirow{10}{*}{} &\multirow{2}{*}{}   &\multirow{2}{*}{}  & \multirow{2}{*}{}\\
    \multirow{10}{*}{} & \centering\multirow{2}{*}{$O\left(\rho^{1/2}\delta^{1/2}\zeta\right)$} & \multirow{2}{*}{$O \left(\frac{L\Delta\rho\delta\sigma^2}{\epsilon^4}+\frac{L\Delta\sigma^2}{(1-\delta)n\epsilon^4}+\frac{L\Delta}{\epsilon^2}+c_2\right)$}   & \multirow{2}{*}{\cite{allouah2023fixing}} \\
    \multirow{10}{*}{} &\multirow{2}{*}{}   &\multirow{2}{*}{}  & \multirow{2}{*}{}\\
    \multirow{10}{*}{} & \centering\multirow{2}{*}{\blue{$O\left(\rho^{1/2}\delta^{1/2}\zeta\right)$}} & \multirow{2}{*}{\blue{$\tilde O \left(\frac{L\Delta\rho\delta \sigma^2}{\epsilon^4} + \frac{L\Delta\sigma^2}{(1-\delta)n\epsilon^4} + \frac{L\Delta}{\epsilon^2}\right)$}} & \multirow{2}{*}{\blue{Algorithm \ref{algo-nc} (Thm. \ref{thm:nc})}} \\
    \multirow{10}{*}{} &\multirow{2}{*}{}   &\multirow{2}{*}{}  & \multirow{2}{*}{}\\
    \bottomrule
\end{tabularx}
\label{table:nc}
    \begin{tablenotes}
        \scriptsize
        \item $\dagger$ The bound is established for a specific robust aggregator.
        \vspace{1pt}
    \end{tablenotes}
    \end{threeparttable}
\end{table}

\subsection{Related works}

{\bf Lower bounds for Byzantine-free single-node optimization.} For deterministic problems, the lower bounds on the iteration complexity of strongly convex and convex optimization are established and proved to be tight in the works of \citep{nemirovskij1983problem,nesterov2003introductory}. That of non-convex optimization are established in \citep{carmon2020lower,carmon2021lower}. For convex stochastic problems with the finite-sum and expectation-minimization structures, the tight lower bounds are derived in \citep{woodworth2016tight} and \citep{foster2019complexity}, respectively. For non-convex stochastic problems, the works of \citep{fang2018spider} and \citep{li2021page} investigate the tight lower bound in the finite-sum structure, while the work of \citep{arjevani2023lower} considers that in expectation-minimization. 

\noindent {\bf Lower bounds for Byzantine-free distributed optimization.} The lower bounds on the iteration complexity of distributed strongly convex deterministic optimization is established in \citep{scaman2017optimal}, in which the network can be both server-based and server-less. A distributed dual accelerated method is proposed to achieve these lower bounds. For distributed server-less, non-convex, stochastic optimization, the optimal oracle query and communication complexities are obtained in \citep{lu2021optimal} given that the communication graphs are linear. These findings are extended to general graphs in \citep{yuan2022revisiting}. The communication complexity of distributed server-based methods with communication compression is investigated in \citep{huang2022lower}.

\noindent {\bf Lower bounds for Byzantine-robust distributed stochastic optimization.} For the convex problems with homogeneous data distributions, the optimal oracle query complexity is established in \citep{alistarh2018byzantine}. When the data distributions are heterogeneous, the non-vanishing Byzantine error emerges \citep{karimireddy2020byzantine}. In contrast, our work simultaneously establishes the tight lower bounds of the Byzantine error and the oracle query complexity.
The lower bound of the statistical learning rate for Byzantine-robust distributed stochastic mean estimation is investigated in \citep{yin2018byzantine}. Two Byzantine-robust methods based on the trimmed mean and coordinate-wise median aggregators are proposed to achieve the order-optimal statistical learning rate. The impact of the dimensionality on the statistic learning rate is taken into account in \citep{zhu2023byzantine}.

\subsection{Organization}

The rest of this paper is organized as follows. Section \ref{sec:2} introduces Byzantine-robust distributed stochastic optimization, including the function, stochastic gradient oracle, robust aggregator, and method classes that are necessary for the ensuing analysis. Section \ref{sec:lb} states the lower bounds of the Byzantine error and the oracle query complexity for strongly convex and non-convex problems.
Section \ref{sec:ub} proposes novel methods to attain the established lower bounds, validating their tightness. Numerical experiments are conducted in Section \ref{sec:exp}. Section \ref{sec:con} summarizes this work.


\section{Byzantine-robust distributed stochastic optimization}\label{sec:2}

This section specifies the notations, assumptions, and problem setup under which we study the optimal complexity for solving the Byzantine-robust distributed stochastic optimization problem in the form of \eqref{prob-general}.

\subsection{Notations}
Throughout this paper, we use $\E_{\xi \sim \mathcal D}$ to denote the expectation over $\xi$, which is a random variable following the distribution $\mathcal D$, and we refer to it as $\E_\xi$ or $\E$ if there is no confusion. We use \( t \) and \( k \) to denote the numbers of iterations and oracle queries, respectively. Accordingly, \( T \) and \( K \) represent the overall numbers of iterations and oracle queries, respectively.
The Euclidean norm is denoted by $\|\cdot\|$.  We use the big-$O$ notations to describe complexity, with $O(\cdot)$ and $\Omega(\cdot)$ hiding constants while $\tilde O(\cdot)$ hiding both constants and logarithmic factors.

\subsection{Function class $\mathcal{F}$}\label{ssec:f}
We let the function class $\mathcal{F}_{L,\zeta^2}$, abbreviated as $\mathcal{F}$, denote the set of all functions $f$ satisfying Assumptions \ref{ass:basic} and \ref{ass:i} for any underlying dimension $d \in \mathbb{N}_+$.


\begin{assumption}\label{ass:basic}
    The function $f(x)$ is continuously differentiable. The functions $\{f_i(x)\}_{i\in \mathcal H}$ are lower bounded. In addition, the functions $\{f_i(x)\}_{i\in \mathcal H}$ are $L$-smooth, \ie, there exists a constant $L>0$ such that
    \begin{align*}
        \| \nabla f_i(x) - \nabla f_i(y) \| \le L\|x - y\|
    \end{align*}
    for all $i \in \mathcal{H}$ and $x, y \in \mathbb{R}^d$.
\end{assumption}

\begin{assumption}\label{ass:i}
    The gradients $\{\nabla f_i(x)\}_{i\in \mathcal H}$ satisfy
    \[
    \frac{1}{|\mathcal H|}\sum_{i \in \mathcal{H}} \|\nabla f_i(x) - \nabla f(x)\|^2 \leq \zeta^2
    \]
    for some $\zeta^2\geq 0$, where $\nabla f(x) = (1/|\mathcal H|)\sum_{i \in \mathcal{H}}\nabla f_i(x)$ according to \eqref{prob-general}.
\end{assumption}

Assumption \ref{ass:basic} is very common. Assumption \ref{ass:i} is widely used in distributed optimization to restrict the data heterogeneity \citep{lian2017can,reddi2020adaptive,karimireddy2020byzantine,allouah2023fixing}.

In the ensuing analysis, we shall examine the complexity bounds when $\{f_i(x)\}_{i \in \mathcal{H}}$ are either $\mu$-strongly convex or non-convex. Below we give the definition of $\mu$-strong convexity.

\begin{definition}
    The functions $\{f_i(x)\}_{i \in \mathcal{H}}$ are \(\mu\)-strongly convex, if there exists a constant $\mu$ $>0$ such that
    \[
        f_i(y) \geq f_i(x) + \nabla f_i(x)^\top (y - x) + \frac{\mu}{2} \|y - x\|^2.
    \]
    for all $i \in \mathcal{H}$ and $x, y \in \mathbb{R}^d$.
\end{definition}

If a function is both $L$-smooth and $\mu$-strongly convex, then $\mu \leq L$.


\subsection{Stochastic gradient oracle class $\mathcal{O}$}\label{ssec:oracle}
We assume that at each iteration $t$, each node $i \in \mathcal H$ can obtain its local stochastic gradient $\nabla F(x,\xi_i^t)$ through an oracle $\mathsf O$, \ie, $\nabla F(x, \xi_i^t)  = \mathsf O(F, x, \xi_i^t)$. We let $\mathcal{O}_{\sigma^2}$, abbreviated as $\mathcal O$, denote the set of all oracles that satisfy the following assumption.
\begin{assumption}\label{ass:u}
The function $F(x,\xi)$ is continuously differentiable with respect to $x$, and the stochastic gradient \( \nabla F(x, \xi_i^t)  = \mathsf O(F, x, \xi_i^t)\) obtained by node \( i\in \mathcal{H} \) through the oracle $\mathsf O \in \mathcal O$ satisfies the following conditions:
\begin{itemize}
\item The random variable \( \xi_i^t \) is independently drawn across all nodes $i \in \mathcal{H}$ and all iterations $t \in \mathbb{N}$.
\item The stochastic gradient is an unbiased estimator of the true gradient, i.e.,
   \[
   \mathbb{E}_{\xi_i^t}[\nabla F(x, \xi_i^t)] = \nabla f_i(x), \quad \forall i\in \mathcal{H}, \ t \in \mathbb{N}.
   \]
\item The variance of the stochastic gradient is bounded, i.e.,
   \[
   \mathbb{E}_{\xi_i^t}[\|\nabla F(x, \xi_i^t) - \nabla f_i(x)\|^2] \leq \sigma^2, \quad \forall i\in \mathcal{H}, \ t \in \mathbb{N}
   \]
   for some constant \( \sigma^2 \ge 0 \).
\end{itemize}
\end{assumption}

In distributed stochastic optimization, independent sampling across different nodes and iterations is common. Besides, the unbiasedness and the bounded variance of the stochastic gradient are widely used assumptions in stochastic optimization \citep{bottou2018optimization}. In the ensuing analysis, we will denote all honest nodes computing their stochastic gradients once as one oracle query.

\subsection{Robust aggregator class $\mathcal{A}$}\label{ssec:agg}
Robust aggregators are essential for mitigating the impact of the Byzantine nodes that inject malicious updates in distributed stochastic optimization. A number of effective robust aggregators, such as Krum, Median, Trimmed Mean, and others, have been proposed in the literature with theoretical guarantees and empirical successes. Nevertheless, the theoretical limits of Byzantine-robust distributed stochastic optimization methods with these robust aggregators remain unknown. Investigating each individual robust aggregator would require an impractical amount of effort. For this reason, this paper does not study the optimal complexity with a specific robust aggregator, but instead focuses on a class of $(\delta_{\rm max},\rho)$-robust aggregators $\mathcal A$ \citep{allouah2023fixing,farhadkhani2022byzantine,karimireddy2020byzantine} defined as follows.
\begin{definition}[$(\delta_{\rm max},\rho)$-robust aggregator]\label{d:agg}
    Consider $n$ inputs $\{w_i\}_{i=1}^n$ from all $n$ nodes, $|\mathcal{H}|$ of them being from the honest nodes in $\mathcal{H}$ and the number of honest nodes satisfying $|\mathcal H| \geq (1 - \delta)n$ with $0 < \delta \leq \delta_{\rm max} < 0.5$. Define $\bar w = \frac{1}{|\mathcal H|}\sum_{i \in \mathcal H} w_i$. An aggregator $\mathsf A \in \mathcal A$ is called $(\delta_{\rm max},\rho)$-robust if there exists a constant $\rho \geq 0$ such that the output 
    {$w=\mathsf A(\{w_i\}_{i=1}^n)$} satisfies
    \begin{align}
    \label{eq:agg}
    \|w - \bar w\|^2 \leq  \frac{\rho\delta}{|\mathcal H|}\sum_{i \in \mathcal H} \|w_i - \bar w\|^2.
    \end{align}
\end{definition}

%

\begin{table*}[t!]
\caption{\small Comparison between different $(\delta_{\rm max},\rho)$-robust aggregators.
}
\centering
\renewcommand{\arraystretch}{1}
\begin{tabular}{cccc}
\toprule
  & & $\rho \delta$   & References \\
\midrule
&Krum & $6+\frac{6\delta}{1-2\delta}$   & \cite{blanchard2017machine}  \vspace{1.5mm} \\
&Median (Med) & $4\left(1+\frac{\delta}{1-2\delta}\right)^2$   & \cite{yin2018byzantine} \vspace{1mm} \\
&Trimmed Mean (TM) & $\frac{6\delta}{1-2\delta}\left(1+\frac{\delta}{1-2\delta}\right)$   & \cite{yin2018byzantine}\\
&FABA & $\frac{2\delta|\mathcal H|}{1-3\delta}$   & \cite{xia2019faba}\\
&Geometric Median (GM) & $4\left(1+\frac{\delta}{1-2\delta}\right)^2$   & \cite{wu2020federated}\\
&Center Clipping (CC) & $18\sqrt{2}\delta\sqrt{|\mathcal H|}$   & \cite{karimireddy2021learning} \vspace{1.5mm} \\
\midrule
&\textbf{Lower bound} & $\frac{\delta}{1-2\delta}$   & \cite{allouah2023fixing}
  \\
\bottomrule
\end{tabular}\label{tab:rho}
\begin{tablenotes}
        \scriptsize
        \item $\dagger$ The robustness coefficients $\rho$ of Krum, Med, TM, and GM, and the lower bound are established in \citep{allouah2023fixing}. That of FABA comes from \citep{peng2024mean}. That of CC is given in Appendix \ref{appendix:A}.
    \end{tablenotes}
\end{table*}

Using a $(\delta_{\rm max},\rho)$-robust aggregator, the deviation of the robust average $w$ from the true average $\bar{w}$ is able to be bounded by the variance of the honest inputs $\{w_i\}_{i \in \mathcal{H}}$. These robust aggregators effectively mitigate the impact of the Byzantine nodes, preventing output divergence.
In particular, a robust aggregator will recover the exact average if the honest inputs $\{w_i\}_{i \in \mathcal{H}}$ are equal and in the majority. Furthermore, it is important to note that $\delta_{\rm max}$ denotes the maximum fraction of the Byzantine nodes the robust aggregator can tolerate, while $\delta$ serves as a form of prior knowledge about the problem, representing the estimated fraction of the Byzantine nodes in the distributed network and being no smaller than the true fraction of the Byzantine nodes. In the analysis, we generally assume that \(\delta > 0\); otherwise, we can simply use the mean aggregator to obtain \(\bar{w}\), resulting in a trivial outcome. We also require $\delta_{\max}<0.5$, meaning that the Byzantine nodes are not dominant. Last but not least, the robustness coefficient $\rho$ plays a key role in characterizing the effectiveness of a robust aggregator. For most robust aggregators, $\rho$ is a function of $\delta$ is dependent on the priori knowledge of $\delta$. Table \ref{tab:rho} lists the robustness coefficients $\rho$ of various $(\delta_{\rm max},\rho)$-robust aggregators.

\subsection{Method class $\mathcal{M}$}
In this paper, we investigate a class of server-based methods to solve the Byzantine-robust distributed stochastic optimization problem in the form of \eqref{prob-general}. With an initialized variable $x^0$, these methods proceed with three phases.
\begin{itemize}
    \item {\it Local computation.} Upon receiving the variable $x^t$ transmitted by the server, each honest node $i \in \mathcal{H}$ computes a batch of $m$ vectors as
    \[
    \mathbf{x}^{(t)}_i = \mathbf{x}^{(t)}:= \left( x^{(t,1)}, \cdots, x^{(t,m)} \right)\in \mathbb R^{d \times m}\ {\rm with}~ x^{(t,l)} \in {\rm span}(x^0,\cdots,x^t) \subseteq \mathbb R^d,\ \forall l \in [m].
    \]
    Note that $\{\mathbf{x}_i^{(t)}\}_{i \in \mathcal H}$ are identical across all honest nodes. Given $\mathbf{x}^{(t)}_i$, each honest node $i \in \mathcal{H}$ samples $m$ independent random variables $\{\xi^{(t,l)}_i\}_{l=1}^m$ with each $\xi^{(t,l)}_i \sim \mathcal D_i$, and queries a batch of $m$ stochastic gradients from the oracle $\mathsf O \in \mathcal O$ as
    \[
    \left(\nabla F(x^{(t,1)}_i,\xi^{(t,1)}_i),\cdots, \nabla F(x^{(t,m)}_i,\xi^{(t,m)}_i)\right) \in \mathbb{R}^{d\times m}.
    \]
    With the above stochastic gradients, each honest node $i \in \mathcal{H}$ computes a gradient estimator
    \[
    w_i^t \in {\rm span}\left(\left\{\nabla F \left(x^{(j,l)}_i, \xi^{(j,l)}_i\right): j = 0, \cdots, t, l = 1, \cdots, m\right\}\right) \subseteq \mathbb R^d.
    \]
    In this paper, we consider the gradient estimator $w_i^t$ that can be also regarded as a linear combination of the preceding stochastic gradients, in the form of
    \begin{align}\label{re:p}
    w^t_i = \sum_{j=1}^{t} \sum_{l=1}^m \alpha^{(j,l)} \nabla F(x_i^{(j,l)},\xi_i^{(j,l)})\in \mathbb{R}^{d},
    \end{align}
    in which $\alpha^{(j,l)} \geq 0$ is the coefficient associated with $\nabla F(x_i^{(j,l)},\xi_i^{(j,l)})$. Besides, we assume \(\sum_{j=1}^{t} \sum_{l=1}^m \alpha^{(j,l)} \ge \alpha_{\rm min} > 0\) for any \(t,m\). Such a gradient estimator $w_i^t$ in \eqref{re:p} is highly versatile and reduces to various existing ones through selecting appropriate values for each $\alpha^{(j,l)}$. For instance,
    when $\alpha^{(j,l)}=1/m$ if $j=t$ and $\alpha^{(j,l)}=0$ otherwise, $w_i^t$ reduces to the mini-batch stochastic gradient in the form of
    $(1/m)\sum_{l=1}^m\nabla F(x_i^{(t,l)},\xi_i^{(t,l)})$. If we further assume $m=1$, $w_i^t$ becomes the classical stochastic gradient $\nabla F(x_i^{(t,1)},\xi_i^{(t,1)})$. Likewise, we can also recover the stochastic mo- mentum \citep{polyak1964some}.


    \item {\it Communication.} Each honest node $i \in \mathcal H$ uploads its computed $w^t_i$ to the server. Each Byzantine node $i \in \mathcal B$, however, may upload an arbitrary vector $w^t_i \in \mathbb{R}^d$.
    \item {\it Global variable update.} The server uses a $(\delta_{\max}, \rho)$-robust aggregator $\mathsf A \in \mathcal{A}$ to pro- cess the messages received from all nodes, yielding an aggregated gradient
    \[
    w^t = \mathsf{A}(w^t_1,\cdots,w^t_n).
    \]
    Subsequently, the server updates the variable $x$ using all historical variables ${x^0,\cdots,x^t}$ and all historical aggregated gradients ${w^0,\cdots,w^t}$. Formally,
    \[
    x^{t+1} \in {\rm span}(x^0,\cdots,x^t,w^0,\cdots,w^t).
    \]
    The server then transmits $x^{t+1}$ to all nodes, initiating a new iteration.
\end{itemize}
Such a process is repeated. We denote the output after $t$ iterations as $\hat x^t \in {\rm span}(x^0,\cdots,x^t)$ for any $t \in \mathbb{N}$. In this paper, we study the set of methods that include the above processes, denoted as $\mathcal{M}$.

\begin{remark}
With particular note, each honest node $i \in \mathcal{H}$ is also allowed to compute and upload multiple gradient estimators at each iteration, only bringing a constant to the overall complexity.
\end{remark}



\section{Lower bound of Byzantine-robust distributed stochastic optimization}\label{sec:lb}

Having introduced the definitions of the function, stochastic gradient oracle, robust aggregator, as well as method classes, we are ready to formalize the concept of complexity. We will prove in Section \ref{ssec:nverr} that if the data is heterogeneous, the gradient norm $\|\nabla f(\hat{x}^t)\|$ would be always away from $0$ for Byzantine-robust distributed stochastic optimization -- this rarely happens in analyzing the complexity lower bounds of Byzantine-free distributed stochastic optimization.
We call this gap the \textbf{Byzantine error}. The Byzantine error refers to the residual that does not vanish regardless of the numbers of iterations and oracle queries, quantifying the robustness of a Byzantine-robust distributed stochastic optimization method.
The Byzantine error generated by \(\mathsf M \in \mathcal{M}\) depends on the function \( f \in \mathcal{F} \), the stochastic gradient oracle \(\mathsf O \in \mathcal{O}\), and the robust aggregator \(\mathsf A \in \mathcal{A}\). If any of these three elements changes, the Byzantine error generated by \( \mathsf M \) also varies. Therefore, we denote
\begin{align}\label{nonvanishing}
\epsilon_{\rm bzt}^{\mathsf M}(f,\mathsf O, \mathsf A) := \inf_{t \in \mathbb N} \left\{\E[\|\nabla f(\hat{x}^t)\| \mid \mathsf O, \mathsf A]\right\},
\end{align}
\noindent where $\hat x^t$ is the output of $\mathsf M \in \mathcal M$ after $t$ iterations.
The rest of the error is termed as the \textbf{optimization error}, which is vanishing and can be reduced to zero when increasing the number of iterations or oracle queries to infinity.

We define the oracle query complexity of the method class $\mathcal M$ on the function class $\mathcal F$, the stochastic gradient oracle class $\mathcal O$, and the robust aggregator class $\mathcal A$, to ensure that the optimization error does not exceed a given \(\epsilon\), as
\begin{align}\label{vanishing}
\mathcal K_\epsilon(\mathcal{M,A,F,O}) := \inf_{\mathsf M \in \mathcal M} \sup_{f,\mathsf O,\mathsf A \in \mathcal{F,O,A}} \inf \left\{ K \mid \E[\|\nabla f(\tilde x^K)\|] \leq  \epsilon_{\rm bzt}^{\mathsf M}(f, \mathsf O, \mathsf A) + \epsilon \right\},
\end{align}
where $\tilde x^K$ is the output $\mathsf M \in \mathcal M$ after $K$ oracle queries. Throughout this paper, given \(f \in \mathcal{F}\), \(\mathsf O \in \mathcal{O}\) and \(\mathsf A \in \mathcal{A}\), we call $x \in \mathbb{R}^d$ an {\it $(\epsilon_{\rm bzt}^{\mathsf M}(f, \mathsf O, \mathsf A),\epsilon)$-stationary point} if $\E[\|\nabla f(x)\|] \leq  \epsilon_{\rm bzt}^{\mathsf M}(f, \mathsf O, \mathsf A) + \epsilon$.
%
%
%
%
%
%
%

In this section, we are going to analyze the lower bounds of the Byzantine error in \eqref{nonvanishing} and the oracle query complexity in \eqref{vanishing}. We begin with showing that there is a non-vanishing Byzantine error through an example involving Byzantine nodes and heterogeneous data (\(\zeta^2 > 0\)) in Section \ref{ssec:nverr}. Then, we proceed to analyze the factors influencing the oracle query complexity when the data is homogeneous (\(\zeta^2 = 0\)). To be specific, Section \ref{ssec:verri} considers $\rho=\sigma^2=0$ to focus on the impact of the function class $\mathcal{F}$, Section \ref{ssec:verrs} sets \(\rho = 0\) but \(\sigma^2 > 0\) to highlight the impact of the stochastic gradient oracle class $\mathcal{O}$, while Section \ref{ssec:verrb} lets \(\rho > 0\) and \(\sigma^2 > 0\) so as to explore the impact of the robust aggregator class $\mathcal{A}$. Finally, summing up these results in Section \ref{ssec:flb} yields a lower bound, whose tightness will be proved in Section \ref{sec:ub}.

\subsection{Lower bound of Byzantine error}\label{ssec:nverr}

We start by analyzing the lower bound of the Byzantine error caused by data heterogeneity ($\zeta^2 > 0$), in the presence of the Byzantine nodes. We define this lower bound as
\begin{align}
\epsilon_{\rm bzt}:= \inf_{\mathsf M \in \mathcal M} \sup_{f,\mathsf O,\mathsf A \in \mathcal{F,O,A}} \epsilon_{\rm bzt}^{\mathsf M}(f,\mathsf O, \mathsf A).
\end{align}
The main idea of the analysis is to construct two problems with different objectives and different minima, $f_1 = \frac{1}{|\mathcal H_1|}\sum_{i \in \mathcal H_1} f_{1,i}$ and $f_2 = \frac{1}{|\mathcal H_2|}\sum_{i \in \mathcal H_2} f_{2,i}$, such that there exists a $(\delta_{\rm max},\rho)$-robust aggregator that yields the same result. Formally speaking, at any iteration $t$, any method $\mathsf M \in \mathcal M$, due to the same result $w^t$ from such a $(\delta_{\rm max},\rho)$-robust aggregator ${\mathsf A} \in \mathcal A$, is going to return the same iterate $x^{t+1}$. Therefore, any method $\mathsf M \in \mathcal M$ must inherently incur an error on at least one of the two problems. We emphasize that the error is due to the data heterogeneity (with which we are able to construct two problems having different objectives and different minima) and the Byzantine nodes (with which there exists a robust aggregator yielding the same result).

\begin{lemma}\label{le:nverr}
    Given $\zeta^2 > 0$ and $\delta \in [0,\delta_{\rm max}]$, there exist a distributed problem in the form of \eqref{prob-general} having at least $(1-\delta)n$ honest nodes with function $f \in \mathcal{F}$, and a $(\delta_{\rm max},\rho)$-robust aggre- gator $\mathsf A \in \mathcal A$, such that for any method $\mathsf M \in \mathcal{M}$, the Byzantine error is lower-bounded by
    \[
        \epsilon_{\rm bzt} = \inf_{\mathsf M \in \mathcal M} \|\nabla f(\tilde x)\| = \Omega(\rho^{1/2} \delta^{1/2} \zeta),
    \]
    where $\tilde x$ is the output of $\mathsf M$, irrelevant with the number of iterations and the number of oracle queries.
\end{lemma}
\begin{proof}
    See Appendix \ref{proof:nverr}.
\end{proof}

Lemma \ref{le:nverr} indicates that achieving the exact minimum of \eqref{prob-general} is unattainable when the Byzantine nodes are present and the data is heterogeneous, consistent with the findings reported in \citep{karimireddy2020byzantine}. The major difference between our work and \citep{karimireddy2020byzantine} lies in that the latter defines the lower bound of the Byzantine error as \( \inf_{\mathsf{M}, \mathsf{A'} \in \mathcal{M, A'}} \sup_{f, \mathsf{O} \in \mathcal{F, O}} \epsilon_{\rm bzt}^{\mathsf{M, A'}} (f, \mathsf{O}) \). Therein, \(\mathcal A'\) represents the set of identity-independent robust aggregators whose outputs are independent on the identities of nodes and \( \epsilon_{\rm bzt}^{\mathsf{M, A'}}(f, \mathsf{O}) \) \(:= \inf_{t \in \mathbb{N}} \left\{\mathbb{E}[\|\nabla f(\hat{x}^t)\| \mid \mathsf{O}]\right\} \). Therefore, the lower bound \( \Omega (\delta^{1/2}\zeta) \) established in \citep{karimireddy2020byzantine} only shows the impacts of the estimated fraction of Byzantine nodes $\delta$ and the data heterogeneity $\zeta^2$, while our result also reveals how the robustness coefficient of robust aggregator $\rho$ affects the lower bound. Note that both results are irrelevant to the stochastic gradient variance \(\sigma^2\). In fact, these two lower bounds are tight in their corresponding setups (see Section \ref{sec:ub} for the tightness of our lower bounds), and the influence of \(\sigma^2\) can be eliminated through proper variance reduction techniques.



\subsection{Lower bound of oracle query complexity: Function} \label{ssec:verri}

In this subsection, we investigate the lower bound of the oracle query complexity influenced by the function $f \in \mathcal F$. To this end, we consider the simplest case \(\zeta^2 = \rho = \sigma^2 = 0\) so as to focus on the impact of $\mathcal F$.
Observe that since we assume $\rho=0$, the robust aggregator \(\mathsf A\) is ideal and averages the inputs of the honest nodes. Besides, due to $\zeta^2=\sigma^2=0$, the behaviors of the honest nodes are exactly the same.
Therefore, this case reduces to single-node deterministic optimization and the classical lower bounds are applicable. For strongly convex functions,
according to \citep{nesterov2003introductory}, we have the following lemma.
\begin{lemma}\label{le:lbsc}
    Given \(\zeta^2 = \rho = \sigma^2 = 0\) and \(\delta \in [0,\delta_{\rm max}]\), there exists a distributed problem in the form of \eqref{prob-general} having at least $(1-\delta)n$ honest nodes with function $f \in \mathcal{F}$ and $\mu$-strongly convex $\{f_i(x)\}_{i\in \mathcal H}$, such that for any method $\mathsf M \in \mathcal M$, to achieve a $(0,\epsilon)$-stationary point, the oracle query complexity is at least
    \[
    K = \Omega \left(\sqrt{\kappa}\log\frac{\mu R}{\epsilon}\right),
    \]
    where $\kappa = \frac{L}{\mu} $ is the condition number and $R = \|x^0 - \arg\min_x f(x)\|$.
\end{lemma}




For non-convex functions, by \citep{carmon2021lower}, we have the following lemma.
\begin{lemma}\label{le:lbc}
    Given \(\zeta^2 = \rho = \sigma^2 = 0\) and \(\delta \in [0,\delta_{\rm max}]\), there exists a distributed problem in the form of \eqref{prob-general} having at least $(1-\delta)n$ honest nodes with function $f \in \mathcal{F}$ and non-convex $\{f_i(x)\}_{i\in \mathcal H}$, such that for any method $\mathsf M \in \mathcal M$, to achieve a $(0,\epsilon)$-stationary point, the oracle query complexity is at least
    \[
    K = \Omega\left(\frac{L\Delta}{\epsilon^2}\right),
    \]
    where $\Delta = f(x^0) - \inf_x f(x)$.
\end{lemma}


\subsection{Lower bound of oracle query complexity: Stochastic gradient oracle} \label{ssec:verrs}
In this subsection, we investigate the lower bound of the oracle query complexity influenced by the stochastic gradient oracle. We maintain \(\zeta^2 = \rho = 0\), but consider the case that \(\mathsf O \in \mathcal{O}\) provides noisy stochastic gradients with variance \(\sigma^2 > 0\). This case is essentially a distributed variant of single-node stochastic optimization, whose lower bounds have been analyzed in \citep{foster2019complexity} for strongly convex functions and \citep{arjevani2023lower} for non-convex functions. The results are shown in the following lemmas.

\begin{lemma}\label{le:verrs}
    Given $\zeta^2 = \rho = 0$, $\sigma^2 > 0$ and \(\delta \in [0,\delta_{\rm max}]\), there exist a distributed problem in the form of \eqref{prob-general} having at least $(1-\delta)n$ honest nodes with function $f \in \mathcal{F}$ and $\mu$-strongly convex $\{f_i(x)\}_{i\in \mathcal H}$, and a stochastic gradient oracle $\mathsf O \in \mathcal O$, such that for any method $\mathsf M \in $ $\mathcal M$, to obtain a $(0,\epsilon)$-stationary point, the expected oracle query complexity is at least
    \[
    K = \Omega \left( \frac{\sigma^2}{(1-\delta)n\epsilon^2} \right).
    \]
\end{lemma}

\begin{lemma}\label{le:verrs-nc}
    Given $\zeta^2 = \rho = 0$, $\sigma^2 > 0$ and \(\delta \in [0,\delta_{\rm max}]\), there exist a distributed problem in the form of \eqref{prob-general} having at least $(1-\delta)n$ honest nodes with function $f \in \mathcal{F}$ and non-convex $\{f_i(x)\}_{i\in \mathcal H}$, and a stochastic gradient oracle $\mathsf O \in \mathcal O$, such that for any method $\mathsf M \in \mathcal M$, to obtain a $(0,\epsilon)$-stationary point, the expected oracle query complexity is at least
    \[
    K = \Omega \left( \frac{L\Delta\sigma^2}{(1-\delta)n\epsilon^4} \right).
    \]
\end{lemma}

The proofs are similar to those of single-node stochastic optimization, but each iteration involves a mini-batch of $(1-\delta)n$ stochastic gradients other than one, such that the variance is accordingly reduced. Thus, we omit the proofs.

\subsection{Lower bound of oracle query complexity: Robust aggregator} \label{ssec:verrb}
In this subsection, we analyze the lower bound of the oracle query complexity influenced by the robust aggregator. We maintain \(\zeta^2=0\), but investigate the case that the output of \(\mathsf A \in \mathcal{A}\) can be different from the average of the inputs from the honest nodes with $\rho>0$ and \(\mathsf O \in \mathcal{O}\) provides noisy stochastic gradients with variance \(\sigma^2 > 0\).



For strongly convex $\{f_i(x)\}_{i\in \mathcal H}$, we construct a one-dimensional problem to analyze the lower bound of the optimization error. In this problem, the gradient at the initial point $x^0$ is set to $\nabla f(x^0)=2\epsilon$. We will prove that, when the number of oracle queries is insufficient, there exists a robust aggregator \(\mathsf A \in \mathcal A\) that always returns $0$, causing any method \(\mathsf M \in \mathcal M\) to be stuck at $x^0$. However, our goal is to find a point \(x\) such that \(|\nabla f(x)| \leq \epsilon\). Clearly, \(x^0\) does not satisfy this condition since \(\nabla f(x^0) = 2\epsilon\). This implies that the number of oracle queries must be sufficient to escape from such an initial point.

%
\begin{lemma}\label{le:verrb}
    Given $\zeta^2=0$, $\rho>0$, $\sigma^2 > 0$, and \(\delta \in [0,\delta_{\rm max}]\), there exist a distributed problem in the form of \eqref{prob-general} having at least $(1-\delta)n$ honest nodes with function $f \in \mathcal{F}$ and $\mu$-strongly convex $\{f_i(x)\}_{i\in \mathcal H}$, a stochastic gradient oracle $\mathsf O \in \mathcal O$, and a robust aggregator ${\mathsf A} \in \mathcal A$, such that for any method $\mathsf M \in \mathcal{M}$, to obtain a $(0,\epsilon)$-stationary point, the expected oracle query complexity is at least
    \[
    K = \Omega \left(\frac{\rho\delta \sigma^2}{\epsilon^2}\right).
    \]
\end{lemma}
\begin{proof}
    See Appendix \ref{proof:verrb}.
\end{proof}

Lemma \ref{le:verrb} highlights the impact of a non-ideal robust aggregator $\mathsf A$ on the oracle query complexity. Different from Lemma \ref{le:verrs} where the robust aggregator is ideal such that $\rho=0$, the term $n$ disappears in Lemma \ref{le:verrb}, showing that introducing a robust aggregator to defend against the Byzantine nodes affects the benefit of cooperation. Instead, the parameters \(\rho\) and \(\delta\) that characterize the performance of the robust aggregator appear, suggesting that a class of high-quality robust aggregators are beneficial to the overall complexity.

The absence of $n$ in Lemma \ref{le:verrb} can be explained from the definition of robust aggregators. At the right-hand side of \eqref{eq:agg} in Definition \ref{d:agg}, the term \(\frac{1}{|\mathcal{H}|} \sum_{i \in \mathcal{H}} \|w_i - \bar{w}\|^2\) represents the variation of the honest nodes' gradient estimators. Such a variation cannot be reduced by increasing the number of nodes \(n\), leading to the oracle query complexity in Lemma \ref{le:verrb}.

For \(\zeta^2 = 0\), another lower bound of the oracle query complexity has been established in \citep{alistarh2018byzantine}. However, the definition of the lower bound in \citep{alistarh2018byzantine} is different from ours, but instead in the form of
\[
\inf_{\mathsf{M}, \mathsf{A'} \in \mathcal{M}, \mathcal{A'}} \sup_{f, \mathsf{O} \in \mathcal{F}, \mathcal{O}} \inf \left\{ K \mid \mathbb{E}[\|\nabla f(\tilde{x}^K)\|] \leq \epsilon \right\},
\]
in which \(\mathcal{A'}\) denotes the set of identity-independent robust aggregators. That said, the work of \citep{alistarh2018byzantine} considers the best identity-independent robust aggregator while our work considers the worst robust aggregator satisfying Definition \ref{d:agg}. Note that when \(\zeta^2 = 0\), the Byzantine error reduces to zero and does not appear in the lower bound (see our Lemma \ref{le:nverr}). With the above definition and under the additional assumption of bounded stochastic gradients, the work of \citep{alistarh2018byzantine} establishes a lower bound of \(\Omega\left(\frac{\delta^2 \sigma^2}{\epsilon^2}\right)\) for the oracle query complexity. The differences in the definitions and assumptions result in the difference of the two lower bounds.

Now we turn to consider non-convex $\{f_i(x)\}_{i\in \mathcal H}$. We construct a high-dimensional problem with \(d = \Omega(\epsilon^{-2})\), each node having the same non-convex function $f(x)$. This function, proposed by \citep{carmon2020lower}, exhibits a chain-like structure such that any noiseless oracle query can only discover the index of the next coordinate. Besides, for every \(x \in \mathbb R^d\) with \([x]_d = 0\), \(\|\nabla f(x)\| > \epsilon\). Thus, for deterministic non-convex optimization, to reach the $\epsilon$ accuracy, any method \(\mathsf M \in \mathcal M\) initialized by $x^0 = \mathbf{0}$ has to discover the \(d\)-th coordinate, which requires at least \(d\) oracle queries due to the chain-like structure of \(f(x)\). This yields a lower bound of \( \Omega(\epsilon^{-2}) \) for oracle query complexity.

For non-convex stochastic optimization, a noisy oracle query is designed in \citep{arjevani2023lower} to amplify the lower bound. It discovers the next coordinate with a probability of \(\Theta\left({\epsilon^2}/{\sigma^2}\right) \), meaning that \(\Omega\left(\sigma^2\epsilon^{-2}\right)\) oracle queries are required to discover the next coordinate in expectation. Hence, the total oracle query complexity is \(\Omega\left(\sigma^2\epsilon^{-4}\right)\). Further, for Byzantine-robust non-convex stochastic optimization, we prove that there is a $(\delta_{\rm max},\rho)$-robust aggregator such that $\Omega(\rho\delta\sigma^2\epsilon^{-2})$ oracle queries are required to discover the next coordinate in expectation, leading to the total oracle query complexity of $\Omega(\rho\delta\sigma^2\epsilon^{-4})$ as stated in the following lemma.

\begin{lemma}\label{le:verrb-nc}
    Given $\zeta^2=0$, $\rho>0$, $\sigma^2 > 0$, and \(\delta \in [0,\delta_{\rm max}]\), there exist a distributed problem in the form of \eqref{prob-general} having at least $(1-\delta)n$ honest nodes with function $f \in \mathcal{F}$ and non-convex $\{f_i(x)\}_{i\in \mathcal H}$, a stochastic gradient oracle $\mathsf O \in \mathcal O$, and a robust aggregator ${\mathsf A} \in \mathcal A$, such that for any method $\mathsf M \in \mathcal{M}$, to obtain a $(0,\epsilon)$-stationary point, the expected oracle query complexity is at least
    \[
    K = \Omega \left(\frac{L\Delta\rho\delta \sigma^2}{\epsilon^4}\right),
    \]
    where $\Delta = f(x^0) - \inf_x f(x)$.
\end{lemma}
\begin{proof}
    See Appendix \ref{proof:verrb-nc}.
\end{proof}



\subsection{Final lower bound}\label{ssec:flb}
Now, we sum up the four lower bounds on the Byzantine error and the oracle query complexity established above to yield the final complexity lower bound. The main results are given in Theorem \ref{thm:flb}.
\begin{theorem}\label{thm:flb}
    Given $\zeta^2\ge0$, $\rho\ge0$, $\sigma^2\ge0$, and \(\delta \in [0,\delta_{\rm max}]\), there exist a distributed pro- blem in the form of \eqref{prob-general} having at least $(1-\delta)n$ honest nodes with function $f \in \mathcal{F}$ and $\mu$-strongly convex $\{f_i(x)\}_{i\in \mathcal H}$, a stochastic gradient oracle $\mathsf O \in \mathcal O$ and a robust aggregator ${\mathsf A} \in \mathcal A$, such that for any method $\mathsf M \in \mathcal{M}$, to obtain an $(\epsilon_{\rm bzt},\epsilon)$-stationary point with \(\epsilon_{\rm bzt} = \Omega(\rho^{1/2}\delta^{1/2}\zeta)\), the expected gradient query complexity is at least
    \begin{equation}\label{low-strc}
    K = \Omega \left(\frac{\rho\delta \sigma^2}{\epsilon^2} + \frac{\sigma^2}{(1-\delta)n\epsilon^2} + \sqrt{\kappa}\log\frac{\mu R}{\epsilon}\right).
    \end{equation}
    For non-convex $\{f_i(x)\}_{i\in \mathcal H}$, the expected gradient query complexity is at least
    \begin{equation}\label{low-nc}
    K = \Omega \left(\frac{L\Delta\rho\delta \sigma^2}{\epsilon^4} + \frac{L\Delta\sigma^2}{(1-\delta)n\epsilon^4} + \frac{L\Delta}{\epsilon^2}\right).
    \end{equation}
\end{theorem}

\section{Upper bound of Byzantine-robust distributed stochastic optimization}\label{sec:ub}
In this section, we will verify the tightness of the lower bound established in Section \ref{sec:lb}, via designing methods $\mathsf M \in \mathcal{M}$ with any robust aggregator ${\mathsf A} \in \mathcal A$ and any stochastic gradient oracle $\mathsf O \in \mathcal O$ to solve \eqref{prob-general} with any function $f \in \mathcal{F}$ and to reach the lower bound. This is a nontrivial task and calls for elaborate integration of several fundamental tools. First, for strongly convex functions the traditional stochastic gradient descent method is not optimal, while for non-convex functions, to reach the corresponding lower bound we need to solve a series of strongly convex subproblems as the inner loop (see Algorithm \ref{algo-nc} for reference). These facts necessitate the use of \textbf{Nesterov's acceleration} to attain the lower bounds for both strongly convex and non-convex functions. Second, the Byzantine nodes can utilize the stochastic gradient noise of the honest nodes to cover their attacks and maximize the error of the robust aggregator. Consequently, \textbf{variance reduction} within each honest node is essential for the performance improvement \citep{wu2020federated,karimireddy2021learning,gorbunov2022variance}.

Below, we propose a Byzantine-robust distributed stochastic Nesterov's accelerated me- thod with variance reduction (Byrd-Nester) to serve as the cornerstone of the subsequent optimal method design. Byrd-Nester applies \textbf{Nesterov's acceleration} in a distributed and stochastic manner, and utilizes the mini-batch technique for \textbf{variance reduction}; see Algorithm \ref{subalgo}. We show in the Appendix \ref{appendix:Equivalentform} that an equivalent variant of Algorithm \ref{subalgo} belongs to the method class \(\mathcal M\).

\begin{algorithm}[ht]
    \caption{{\bf By}zantine-{\bf r}obust {\bf d}istributed stochastic {\bf Neste}rov's accelerated method with variance {\bf r}eduction (Byrd-Nester)}
    \label{subalgo}
    \begin{algorithmic}
    \STATE{Input: starting point $x^0$, auxiliary point $y^0=x^0$, maximum number of iterations $T$, batch size $m_0$, $m$, step size $\eta$, $\theta\in(0,1]$}, {$\beta\in[0,1)$}, {$\alpha\in[0,1]$, $\hat s^0 = s_i^0 = \frac{1}{m_0}\sum_{l = 1}^{m_0} \nabla F(y^{0},\xi^{(0,l)}_i)$.}
    \FOR{$t=1,\cdots,T$}
    \FOR{node $i \in \mathcal {H}$}
    \STATE{Independently sample $\{\xi^{(t-1,1)}_i,\cdots,\xi^{(t-1,m)}_i \}$, obtain stochastic gradients from oracle $\mathsf O \in \mathcal{O}$ and calculate
        \begin{align}
             g_i^{t-1} &=  \frac{1}{m}\sum_{l = 1}^m \nabla F(y^{t-1};\xi^{(t-1,l)}_i), \label{mini-batch}\\
             s^t_i &= \beta s^{t-1}_i + \theta g_i^{t-1} .\label{update:s}
        \end{align} \vspace{-1em} }
    \STATE{Send $g_i^{t-1}$ and $s_i^t$ to server.}
    \ENDFOR
    \FOR{node $i \in \mathcal {B}$}
    \STATE{Send arbitrary vector $g_i^{t-1} \in \mathbb R^d$ and $s_i^t \in \mathbb R^d$ to server.}
    \ENDFOR
    \STATE{Server receives $\{ g_i^{t-1}\}_{i=1}^{n}$ and $\{s^t_i\}_{i=1}^{n}$, and updates
    \begin{align}
            s^t &= \beta \hat s^{t-1} + \theta{\mathsf A} (\{ g_i^{t-1}\}_{i=1}^{n}), \label{hat-s-t}\\
            \hat s^t &= (1 - \alpha) s^t + \alpha {\mathsf A} (\{s^t_i\}_{i=1}^{n}), \label{ts+hs}\\
            x^t &= x^{t-1} - \eta \hat s^t, \label{update:x}\\
            y^t &= x^t + \beta (x^t - x^{t-1}).\label{update:y}
    \end{align} \vspace{-1em} }
    \STATE{Server sends $y^t$ to all nodes.}
    \ENDFOR
    \RETURN $\tilde x^K = x^T$ for strongly convex optimization; $\tilde x^K = y^{t'}$ where $t'$ is randomly chosen from \(0,\cdots,T-1\) for non-convex optimization. Here $K$ is the number of oracle queries.
    \end{algorithmic}
\end{algorithm}

At the \(t\)-th iteration, each honest node $i \in \mathcal{H}$ queries a mini-batch of $m$ stochastic gradients at the auxiliary point \(y^{t-1}\), and averages them to calculate the mini-batch stochastic gradient \(g_i^{t-1}\) as in \eqref{mini-batch}. The purpose of this step is to reduce the variance of the stochastic gradient noise. Then, each honest node $i\in \mathcal{H}$ calculates \(s_i^t\), a weighted combination of the historical and current mini-batch stochastic gradients, for the sake of node-level acceleration as in \eqref{update:s}. It is worth noting that \(\beta + \theta\) may exceed 1, representing the aggressive usage of the mini-batch stochastic gradients. After that, each honest node $i \in \mathcal{H}$ sends \(g_i^{t-1}\) and \(s_i^t\) to the server. In contrast, each Byzantine node \(i \in \mathcal B\) may send two arbitrary \(d\)-dimensional vectors \(g_i^{t-1}\) and \(s_i^t\) to the server.

Then, the server aggregates the received $\{ g_i^{t-1}\}_{i=1}^{n}$ and $\{s^t_i\}_{i=1}^{n}$ via a robust aggregator \(\mathsf A \in \mathcal A\). Therein, \eqref{hat-s-t} uses \({\mathsf A} (\{ g_i^{t-1}\}_{i=1}^{n})\) to calculate \(s^t\) for server-level acceleration. Note that the update involves \(\hat{s}^{t-1}\) instead of \(s^{t-1}\). Meanwhile, as shown in \eqref{ts+hs}, \(\hat s^{t}\) is a linear combination of $s^t$ and \({\mathsf A} (\{s^t_i\}_{i=1}^{n})\), parameterized by $\alpha \in [0,1]$ to adjust the balance between node-level and server-level accelerations. In particular, $\alpha=0$ voids node-level acceleration, while $\alpha=1$ renders server-level acceleration ineffective. Such a design offers more flexibility to the proposed method. Using \(\hat s^{t}\), \eqref{update:x} runs a descent step to update $x^t$, while \eqref{update:y} runs an extrapolation step to $y^t$, differentiating the adopted Nesterov's acceleration from the mom- entum acceleration. Finally, the server sends \(y^t\) to all nodes.


Note that the node-level momentum acceleration has also been utilized in stochastic optimization and its Byzantine-robust distributed variant \citep{liu2020improved,karimireddy2021learning} for variance reduction, in the form of \(s_i^t = \beta s_i^{t-1} + (1-\beta) g_i^{t-1}\) that is similar to \eqref{update:s}. However, its variance reduction effect relies on setting \(\beta\) sufficiently close to \(1\) and \(1-\beta\) to \(0\). Our choices of \(\beta\) and \(\theta\) do not satisfy these requirements; see the lemmas, theorems and corollaries in the following subsections. This fact explains why we still need the mini-batch technique for variance reduction, on top of the node-level Nesterov's acceleration.


To better understand the behavior of Algorithm \ref{subalgo} and facilitate the subsequent analysis, we provide a deeper examination of \eqref{update:x}. Observe that \eqref{update:x} is equivalent to
\begin{align}
\label{detailupdate}
    x^t =& x^{t-1} - \eta \hat s^t = x^{t-1} - \eta \bar s^t + \eta \bar s^t - \eta \hat s^t \\
    =& x^{t-1} - \eta \beta \bar s^{t-1} - \frac{\eta\theta}{|\mathcal H|m}\sum_{i \in \mathcal H}\sum_{l=1}^m\nabla F (y^{t-1};\xi^{(t-1,l)}_i)  + \eta \bar s^t - \eta \hat s^t \notag \\
    =&x^{t-1} - \eta \beta \hat s^{t-1} - \eta \theta\nabla f(y^{t-1}) + \eta\theta\nabla f(y^{t-1}) - \frac{\eta\theta}{|\mathcal H|m}\sum_{i \in \mathcal H}\sum_{l=1}^m\nabla F (y^{t-1};\xi^{(t-1,l)}_i) \notag \\
    & + \eta \bar s^t - \eta \hat s^t - \eta\beta (\bar s^{t-1} - \hat s^{t-1}) \notag \\
    =& \underbrace{y^{t-1} - \eta\theta\nabla f(y^{t-1})}_{\text{Nesterov's acceleration}}+ \underbrace{ \eta \bar s^t - \eta \hat s^t -\eta\beta (\bar s^{t-1} - \hat s^{t-1})}_{\text{$\Delta_1^t$: aggregation bias}} \notag \\
    &+ \underbrace{\eta\theta\nabla f(y^{t-1}) - \frac{\eta\theta}{|\mathcal H|m}\sum_{i \in \mathcal H}\sum_{l=1}^m\nabla F(y^{t-1};\xi^{(t-1,l)}_i) }_{\text{$\Delta_2^t$: stochasticity bias}}, \notag
\end{align}
where $\bar s^t= \frac{1}{|\mathcal H|}\sum_{i \in \mathcal H} s_i^t$. According to \eqref{detailupdate}, the update of $x^t$ consists of three parts: Neste- rov's acceleration, aggregation bias $\Delta_1^t$ and stochasticity bias $\Delta_2^t$. If the robust aggregator $\mathsf A$ is ideal such that $\rho=0$, Algorithm \ref{subalgo} reduces to the distributed stochastic Nesterov's accelerated method. If further the stochastic gradient variance $\sigma^2=0$, it turns to the distributed deterministic Nesterov's accelerated method.

\subsection{Strongly convex optimization}


For strongly convex optimization, we first analyze the oracle query complexity of Algorithm \ref{subalgo} and show that it has an $O(\log \epsilon^{-1})$ gap to the $\Omega(\epsilon^{-2})$ lower bound in Theorem \ref{thm:flb}. We set \(\beta\) in Algorithm \ref{subalgo} as
\begin{align}\label{beta}
\beta = \frac{\sqrt{q}-1}{\sqrt{q}+1},
\end{align}
where \(q \geq 1\) is a constant. We will set $q = \frac{L}{\mu\theta}= \frac{\kappa}{\theta}$ in the ensuing analysis. The following lemma provides an effective tool to establish the convergence of Algorithm \ref{subalgo}.



\begin{lemma} \label{prop:gc}
Given $\zeta^2\ge0$, $\rho\ge0$, $\sigma^2\ge0$, $L > 0$ and \(\delta \in [0,\delta_{\rm max}]\), for any distributed problem in the form of \eqref{prob-general} having at least $(1-\delta)n$ honest nodes with any function $f \in \mathcal{F}$ and $\mu$-strongly convex $\{f_i(x)\}_{i\in \mathcal H}$, any stochastic gradient oracle $\mathsf O \in \mathcal O$ and any \((\delta_{\rm max},\rho)\)-robust aggregator \(\mathsf A \in \mathcal A\), if there exist an $\frac{L}{\theta}$-strongly convex function $h^t(x)$, and parameters $v^t$ and $\varepsilon^t$ at each iteration $t$ such that
    \begin{enumerate}[label=(\roman*)]
        \item $\E[h^t(x)] \leq f(x) + \frac{L-\mu\theta}{2\theta}\|x-y^{t-1}\|^2$ for $x=q^{-1/2}x^*+(1-q^{-1/2})x^{t-1}$, \label{ta1}
        \item $\E[f(x^t)] \leq \E[h^t(x^{t*})] + \upsilon^t$, \label{ta2}
        \item $\E[h^t(x^t)] \leq \E[h^t(x^{t*})] + \varepsilon^t$, \label{ta3}
    \end{enumerate}
then with $\beta = \frac{\sqrt{q}-1}{\sqrt{q}+1}$ and \(q = \frac{L}{\mu\theta}= \frac{\kappa}{\theta}\), the iterate \(x^t\) generated by Algorithm \ref{subalgo} satisfies
\begin{align}\label{convergence-guarantee}
\E[f(x^t) - f^*] \leq \left(1 - \frac{1}{2\sqrt{q}}\right)^t\left(2(f(x^0)-f^*)+4\sum_{\tau=1}^t\left(1 - \frac{1}{2\sqrt{q}}\right)^{-\tau}(\upsilon^\tau+\sqrt{q}\varepsilon^\tau)\right).
\end{align}
Therein, $x^* =\arg\min_x f(x)$, $x^{t*} = \arg\min_{x} h^t(x)$ and \(f^* = \inf_x f(x)\).
\end{lemma}

\begin{proof}
    See Appendix \ref{proof:proposition}.
\end{proof}

The result of Lemma \ref{prop:gc} relies on the existence of a proper surrogate function $h^t(x)$ and we will discuss later. In \eqref{convergence-guarantee}, the term of $1-1/(2\sqrt{q})$ implies the accelerated convergence. Nevertheless, the convergence is negatively affected by the residual $\upsilon_\tau+\sqrt{q}\varepsilon_\tau$, in which \(\upsilon^t\) measures the appropriateness of \(h^t(x)\) as a surrogate function (see Lemma 2.2.1 in \citep{nesterov2003introductory}), while \(\varepsilon^t\) quantifies the gap between \(x^t\) and the minimizer of $h^t(x)$.

Now, we design a set of \(\{h^t(x), \upsilon^t, \varepsilon^t\}_{t=1}^T\) that satisfy the three conditions in Lemma \ref{prop:gc}.
First, the surrogate function $h^t(x)$ is given by
\begin{equation}\label{h-t}
h^t(x) := f(y^{t-1}) + \langle \nabla f(y^{t-1}), x - y^{t-1} \rangle + \frac{L}{2\theta}\|x - y^{t-1}\|^2,
\end{equation}
where $\theta$ has been introduced in \eqref{update:s} and \eqref{hat-s-t}. For such a surrogate function, $x^{t*} = y^{t-1} - \frac{\theta}{L} \nabla f(y^{t-1})$. Second, we set \(\upsilon^t = \varepsilon^t  = \frac{L}{\theta}\E[\|\Delta_1^t\|^2+\|\Delta_2^t\|^2]\). Below, we verify the three con- ditions in Lemma \ref{prop:gc} one by one.
\begin{enumerate}[label=(\roman*)]
    \item It obviously holds from the strong convexity of $f$.
    \item Consider
        \[
        f(x^t) \leq h^t(x^t) \leq h^t(x^{t*}) + \frac{L}{2\theta} \|x^t - x^{t*}\|^2,
        \]
        where the first inequality follows from the $L$-smoothness of $f$ and the second inequality follows from the $\frac{L}{\theta}$-smoothness of $h^t$. Taking expectations and letting the step size $\eta = \frac{1}{L}$, we have
        \begin{align*}
            \E[f(x^t)] \leq &  \E[h^t(x^{t*})] + \frac{L}{2\theta} \E[\|x^t - x^{t*}\|^2]
            = \E[h^t(x^{t*})] + \frac{L}{2\theta} \E[\|\Delta_1^t+\Delta_2^t\|^2] \\
            \leq & \E[h^t(x^{t*})] + \upsilon^t.
        \end{align*}
        where the equality is due to \eqref{detailupdate} and the fact of $x^{t*} = y^{t-1} - \frac{\theta}{L} \nabla f(y^{t-1})$.
    \item Again, consider
        \[
        h^t(x^t) \leq h^t(x^{t*}) + \frac{L}{2\theta} \|x^t - x^{t*}\|^2,
        \]
        that we have derived from the $\frac{L}{\theta}$-smoothness of $h^t$. Following the similar derivation as in (ii), we have
        \begin{align*}
            \E[h^t(x^t)] \leq &  \E[h^t(x^{t*})] + \frac{L}{2\theta} \E[\|x^t - x^{t*}\|^2]
            = \E[h^t(x^{t*})] + \frac{L}{2\theta} \E[\|\Delta_1^t+\Delta_2^t\|^2] \\
            \leq & \E[h^t(x^{t*})] + \varepsilon^t.
        \end{align*}
\end{enumerate}

To establish the convergence of Algorithm \ref{subalgo}, it remains to bound $\upsilon^t$ and $\varepsilon^t$; that is, to bound $\frac{L}{\theta}\E[\|\Delta_1^t\|^2+\|\Delta_2^t\|^2]$. In Appendices \ref{proof:v} and \ref{proof:Delta2}, we respectively bound the agg- regation bias \(\E[\|\Delta_1^t\|^2]\) and the stochasticity bias \(\E[\|\Delta_2^t\|^2]\). With them, we have
\begin{align}
\label{eq:uvbound}
\upsilon^t = \varepsilon^t & = \frac{L}{\theta} \E[\|\Delta_1^t\|^2+\|\Delta_2^t\|^2] \\
& \leq \frac{1}{L\theta}\left(\frac{3\chi_4\rho\delta\sigma^2}{m}\left(1+\frac{1}{(1-\delta)n}\right)+\frac{\theta^2\sigma^2}{(1-\delta)nm}+3\chi_5\rho\delta\zeta^2\right), \notag
\end{align}
where \(\chi_4, \chi_5 \ge 0\) are some constants.
It is worth noting that the first two terms at the right-hand side of \eqref{eq:uvbound} are controlled by the batch size $m$, implying that we can reduce the effect of the variance $\sigma^2$ by increasing $m$. The last term is a constant error, causing the Byzantine error in Section \ref{ssec:nverr}.

Hence, we establish the convergence result of Algorithm \ref{subalgo} in terms of both the function value and the gradient norm as follows.

\begin{theorem}\label{thm:sc}
    Given $\zeta^2\ge0$, $\rho\ge0$, $\sigma^2\ge0$, $L > 0$ and \(\delta \in [0,\delta_{\rm max}]\), for any distributed problem in the form of \eqref{prob-general} having at least $(1-\delta)n$ honest nodes with any function $f \in \mathcal{F}$ and $\mu$-strongly convex $\{f_i(x)\}_{i\in \mathcal H}$, any stochastic gradient oracle $\mathsf O \in \mathcal O$ and any \((\delta_{\rm max},\rho)\)-robust aggregator \(\mathsf A \in \mathcal A\), consider Algorithm \ref{subalgo} with the step size $\eta=\frac{1}{L}$. With $\beta = \frac{\sqrt{q}-1}{\sqrt{q}+1}$ and \(q = \frac{L}{\mu\theta}= \frac{\kappa}{\theta}\), if the parameters $\alpha$ and $\theta$ meet the requirements in Lemma \ref{le:v}, then under Assumptions \ref{ass:basic}--\ref{ass:u}, the iterate \(x^t\) generated by Algorithm \ref{subalgo} satisfies
    \begin{align}\label{f-f*}
    \E[f(x^t) - f^*]
    \leq 2\left(1 - \frac{1}{2\sqrt{q}}\right)^t\Delta+\frac{16}{\mu\theta^2}\left(\frac{6\chi_4\rho\delta\sigma^2}{m}+\frac{\theta^2\sigma^2}{(1-\delta)nm}+3\chi_5\rho\delta\zeta^2\right),
    \end{align}
    \begin{align}\label{gradnorm}
    \E[\|\nabla f(x^t)\|^2] \leq 2\left(1 - \frac{1}{2\sqrt{q}}\right)^t L^2R^2+\frac{32\kappa}{\theta^2}\left(\frac{6\chi_4\rho\delta\sigma^2}{m}+\frac{\theta^2\sigma^2}{(1-\delta)nm}+3\chi_5\rho\delta\zeta^2\right),
    \end{align}
    where \(\kappa = \frac{L}{\mu}\), $\Delta = f(x^0)-f^*$ and $R = \|x^0-x^*\|$.
\end{theorem}

\begin{proof}
    Combining Lemmas \ref{prop:gc}, \ref{le:v} and \ref{le:Delta2}, as well as utilizing the closed form formula for the sum of a geometric series, we obtain \eqref{f-f*}. Further from \(\|\nabla f(x^t)\|^2 \leq 2L \left( f(x^t) - f^*\right)\) and $f(x^0) - f^* \leq LR^2/2$, we obtain \eqref{gradnorm}.
\end{proof}

We specify the values of $\alpha$ and $\theta$ to finalize the oracle query complexity of Algorithm \ref{subalgo}.


\begin{corollary}\label{rm:sc}
    Under the conditions in Theorem \ref{thm:sc}, we define $\kappa=\frac{L}{\mu}$ and \(q = \frac{L}{\mu\theta}= \frac{\kappa}{\theta}\), set $\alpha = 0$, $\theta = 1$ and $\beta = \frac{\sqrt{q}-1}{\sqrt{q}+1} = \frac{\sqrt{\kappa}-1}{\sqrt{\kappa}+1}$, as well as set the parameters $T$ and $m$ as
    \[
    T = 2\sqrt{\kappa}\log \frac{4L^2R^2}{\epsilon^2}, \quad m = m_0 = 64\kappa\left(3\rho \delta (1+\frac{1}{(1-\delta)n})+\frac{1}{(1-\delta)n}\right)\frac{\sigma^2}{\epsilon^2}.
    \]
    Then, the output of Algorithm \ref{subalgo} defined by $\tilde x^K = x^T$ satisfies
    \[
    \E[\|\nabla f(\tilde x^K)\|] \leq \epsilon + 4\sqrt{6}\kappa^{1/2}\rho^{1/2}\delta^{1/2}\zeta,
    \]
    with the oracle query complexity
    \begin{align}\label{clp-algo1}
    K = m_0 + mT = O\left(\kappa^{3/2}\left( \rho \delta +\frac{1}{(1 - \delta)n}\right)\frac{\sigma^2}{\epsilon^2}\cdot \log \frac{LR}{\epsilon}\right).
    \end{align}
\end{corollary}

\begin{proof}
    When $\alpha = 0$ and $\theta = 1$, we have $\chi_4 = \chi_5 = 1$. Substituting the values of $T$ and $m$ into \eqref{gradnorm} yields the above result.
\end{proof}


As shown in Corollary \ref{rm:sc}, the oracle query complexity of Algorithm \ref{subalgo} in terms of $\epsilon$ is $K = mT = O(\epsilon^{-2}\log \epsilon^{-1})$, with an $O(\log \epsilon^{-1})$ gap to the $ \Omega(\epsilon^{-2})$ lower bound in Theorem \ref{thm:flb}. To close this gap, we apply the restart technique and use an increasing batch size, yielding the optimal Algorithm \ref{algo:restart}.

\begin{algorithm}[htbp]
    \caption{Byrd-Nester with restart (Byrd-reNester)}
    \label{algo:restart}
    \begin{algorithmic}
    \STATE{Input: initial point $z(0)$, $T(1)$ in \eqref{t1m1} and $T(p) = \lceil \frac{2L^{1/2}}{\mu^{1/2}}\log 8 \rceil$ for $p \geq 2$, $m(p) = 2^{p-1}$, and $P = \max\{\lceil \log_2 \frac{4L(\epsilon(1))^2}{\epsilon^2} \rceil,1\}$ with $(\epsilon(1))^2$ in \eqref{eps1}.}
    \FOR{$p=1,\cdots,P$}
    \STATE{Output $z(p)$ from Byrd-Nester with $x^0=z(p-1)$, $T = T(p)$ and $m = m(p)$.}
    \ENDFOR
    \RETURN $\tilde x^K = z(P)$.
    \end{algorithmic}
\end{algorithm}

Here, \( T(p) \), \( m(p) \) and \( \epsilon(p) \) respectively represent the maximum number of Byrd-Nester calls, the batch size, the expected optimization error for the $p$-th call of Byrd-Nester in Algorithm \ref{algo:restart}. The parameters $\alpha$, $\theta$ and $\beta$ are the same as those in Corollary \ref{rm:sc}. The optimal oracle query complexity of Algorithm \ref{algo:restart} is established in the following Theorem.

\begin{theorem}\label{thm:sc-restart}
    Under the conditions in Theorem \ref{thm:sc}, we define $\kappa=\frac{L}{\mu}$ and \(q = \frac{L}{\mu\theta}= \frac{\kappa}{\theta}\), set $\alpha = 0$, $\theta = 1$ and $\beta = \frac{\sqrt{q}-1}{\sqrt{q}+1} = \frac{\sqrt{\kappa}-1}{\sqrt{\kappa}+1}$.
    Then, the output of Algorithm \ref{algo:restart} defined by $\tilde x^K$ satisfies
    \begin{align}\label{nverr-rs}
    \E[\|\nabla f(\tilde x^K)\|] \leq 8\sqrt{2}\kappa^{1/2}\rho^{1/2}\delta^{1/2}\zeta + \epsilon,
    \end{align}
    with the oracle query complexity
    \begin{align}\label{clp-rs}
    K = \sum_{p=1}^P m(p)T(p) = O\left(\kappa^{1/2}\log \frac{LR}{\epsilon} + \kappa^{3/2}\left( \rho \delta +\frac{1}{(1 - \delta)n}\right)\frac{\sigma^2}{\epsilon^2}\right).
    \end{align}
\end{theorem}
\begin{proof}
    See Appendix \ref{proof:sc-restart}.
\end{proof}

The oracle query complexity established in Theorem \ref{thm:sc-restart} is optimal, exactly matching the lower bound for strongly convex optimization in Theorem \ref{thm:flb}. Below, we demonstrate that Algorithm \ref{algo:restart} is also optimal in the following special cases.

\begin{itemize}
    \item {\bf Specialization to} $\zeta^2 = 0$. Setting $\zeta^2 = 0$ yields
        \(
            \E[\|\nabla f(\tilde x^K)\|] \leq \epsilon
        \)
    and the oracle query complexity is \[
    O\left(\frac{\rho\delta \sigma^2}{\epsilon^2} + \frac{\sigma^2}{(1-\delta)n\epsilon^2} + \kappa^{1/2}\log\frac{L R}{\epsilon}\right).
    \] This complexity matches the lower bound for strongly convex optimization with data homogeneity (\(\zeta^2=0\)), as shown in Theorem \ref{thm:flb}.
    \item {\bf Specialization to} $\delta = 0$. Setting $\delta = 0$ yields \(
            \E[\|\nabla f(\tilde x^K)\|] \leq \epsilon
        \) and the oracle query complexity is
    \begin{align*}
    O\left(\kappa^{1/2}\log \frac{LR}{\epsilon} + \frac{\kappa^{3/2}}{n}\frac{\sigma^2}{\epsilon^2}\right).
    \end{align*}
    This complexity matches the lower bound for strongly convex optimization in the absence of Byzantine nodes (\(\delta=0\)), as shown in Theorem \ref{thm:flb}.
    \item {\bf Specialization to} $\sigma^2 = 0$. Setting $\sigma^2 = 0$, $P=1$, $m(1) = 1$, and
    $T(1) = 2{\kappa}^{1/2}$ $\log \frac{2L^2R^2}{\epsilon^2}$ yields
    \(
    \E[\|\nabla f(\tilde x^K)\|] \leq 8\sqrt{2}\kappa^{1/2}\rho^{1/2}\delta^{1/2}\zeta + \epsilon
    \)
    and the oracle query complexity is
    \[
    O\left(2{\kappa}^{1/2}\log \frac{LR}{\epsilon}\right).
    \]
    This complexity matches the lower bound for strongly convex optimization without stochasticity (\(\sigma^2=0\)), as shown in Theorem \ref{thm:flb}.
    \item {\bf Specialization to} $\sigma^2 = \zeta^2 = 0$ {\bf or} $\sigma^2 = \delta = 0$. Setting $\sigma^2 = \zeta^2 ~ ({\rm or}~ \delta) = 0$, $K=1$, $m(1) = 1$ and
    \(
        T(1) = 2\sqrt{\kappa}\log \frac{2L^2R^2}{\epsilon^2}
    \) yields
    \(
        \E[\|\nabla f(\tilde x^K)\|] \leq \epsilon
    \)
    and the oracle query complexity is
    \[
    O\left(2\kappa^{1/2}\log \frac{LR}{\epsilon}\right).
    \]
    This complexity matches the lower bound for strongly convex optimization without stochasticity but with data homogeneity ($\sigma^2 = \zeta^2 = 0$) or without stochasticity but in the absence of Byzantine nodes ($\sigma^2 = \delta = 0$), as shown in Theorem \ref{thm:flb}.
\end{itemize}

%


\subsection{Non-convex optimization}\label{ssec:nc}
For non-convex optimization, we begin with establishing the convergence of Algorithm \ref{subalgo}.
\begin{theorem}\label{thm:nc1}
    Given $\zeta^2\ge0$, $\rho\ge0$, $\sigma^2\ge0$, $L > 0$ and \(\delta \in [0,\delta_{\rm max}]\), for any distributed problem in the form of \eqref{prob-general} having at least $(1-\delta)n$ honest nodes with any function $f \in \mathcal{F}$, any stochastic gradient oracle $ \mathsf O \in \mathcal O$ and any \((\delta_{\rm max},\rho)\)-robust aggregator \(\mathsf A \in \mathcal A\), consider Algorithm \ref{subalgo} with the step size $\eta$ set in \eqref{stepsize}.
    With $\beta = 1 - 12L\eta$, if batch size $m=O(1)$, $m_0 = m/(L^2\eta^2)$ and the parameters $\alpha$, $\theta$ meet the requirements in Lemma \ref{le:v} with $\chi_1+\chi_2 = O(L\eta)$ and $\chi_3 = O(1) \geq 0$, as well as
    \begin{align*}
        (1-\theta-\beta)^2 \leq \chi_6 (1-\beta)^2,\\
        \chi_7 \leq \frac{1}{3}-6\chi_6-(\theta+\beta^2+\theta\beta-1)^2,
    \end{align*}
    for some $\chi_6 \geq 0$ and $\chi_7 = \Theta (1) > 0$, then the iterate \(y^t\) generated by Algorithm \ref{subalgo} satisfies
    \begin{align*}
    \frac{1}{T}\sum_{t=1}^T\|\nabla f( y^{t-1})\|^2 &\leq O \left( \sqrt{\frac{L\Delta+\sigma^2/n}{T}}\sqrt{\left(\frac{1}{(1-\delta)n}+\rho\delta\right)  \sigma^2} + \frac{L\Delta}{T}+\frac{\sigma^2}{Tn} +\rho\delta\zeta^2 \right),
    \end{align*}
    where $\Delta = f(x^0)-f^*$.
\end{theorem}
\begin{proof}
    See Appendix \ref{section:Theorem 16}.
\end{proof}

Based on Theorem \ref{thm:nc1}, we specify the values of $\alpha$ and $\theta$ to determine the oracle query complexity of Algorithm \ref{subalgo} in the following corollary.
\begin{corollary}\label{cor:nc1}
    Under conditions of Theorem \ref{thm:nc1}, let $\theta = 1 - \beta$ and $\alpha = 1$ so that $\chi_1 = 1 - \beta = 12L\eta$, $\chi_2=\chi_6=0$, $\chi_3=1$ and $\chi_7 = 1/3$. Consider Algorithm \ref{subalgo} with the step size $\eta$ set in \eqref{stepsize}. The output of Algorithm \ref{subalgo} defined by $\tilde x^K$ satisfies
    \[
    \E [\|\nabla f( \tilde x^K)\|] \leq \sqrt{210} \rho^{1/2}\delta^{1/2}\zeta + \epsilon,
    \]
    with the oracle query complexity
    \[
    K = m_0 + mT = O \left(\frac{L\Delta\rho\delta\sigma^2}{\epsilon^4}+\frac{L\Delta\sigma^2}{(1-\delta)n\epsilon^4}+\frac{\rho\delta\sigma^4}{(1-\delta)n\epsilon^4}+\frac{L\Delta}{\epsilon^2}+\frac{\sigma^2}{(1-\delta)n\epsilon^2}\right).
    \]
\end{corollary}

In Corollary \ref{cor:nc1}, \(\theta = 1 - \beta = 12L\eta = O(\frac{1}{\sqrt{T}})\), vanishing as $T$ goes to infinity. Therefore, the node-level Nesterov's acceleration in
\eqref{update:s} is also effective for variance reduction \citep{liu2020improved,karimireddy2021learning}, as shown in Lemma \ref{lem:sgdm-byz-error}. For this reason, Algorithm \ref{subalgo} no longer requires a large batch size $m$. However, there is also an $O(\frac{\rho\delta\sigma^4}{(1-\delta)n\epsilon^4}+\frac{\sigma^2}{(1-\delta)n\epsilon^2})$ gap between the oracle query complexity in Corollary \ref{cor:nc1} and the lower bound in Theorem \ref{thm:flb}. The key idea to close this gap is to approximately solve a series of strongly convex surrogates, each calling Byrd-reNester once, via an inexact proximal point algorithm outlined in Algorithm \ref{algo-nc}. We establish the oracle query complexity of Algorithm \ref{algo-nc} in the following theorem.




\begin{algorithm}[ht]
    \caption{Inexact Proximal Point Algorithm with Byrd-reNester}
    \label{algo-nc}
    \begin{algorithmic}
    \STATE{Input: initial point $\varkappa^0$, maximum number of Byrd-reNester calls $\Gamma$.}
    \FOR{$\gamma=1,\cdots,\Gamma$}
    \STATE{Set $f_i^\gamma(z) = f_i(z) + L\|z-\varkappa^{\gamma-1}\|^2$ for all $i \in \mathcal{H}$.}
    \STATE{Output $\varkappa^\gamma$ by applying Byrd-reNester to $f^\gamma(z)=\frac{1}{|\mathcal{H}|} \sum_{i \in \mathcal{H}} f_i^{\gamma}(z)$ with $z(0)=\varkappa^{\gamma-1}$.}
    \ENDFOR
    \RETURN $\tilde x^K = \varkappa^{\gamma'}$, where $\gamma'$ is randomly chosen from $1,\cdots,\Gamma$.
    \end{algorithmic}
\end{algorithm}

\begin{theorem}\label{thm:nc}
    Given $\zeta^2\ge0$, $\rho\ge0$, $\sigma^2\ge0$, $L > 0$, and \(\delta \in [0,\delta_{\rm max}]\), for any distributed problem in the form of \eqref{prob-general} having at least $(1-\delta)n$ honest nodes with any function $f \in \mathcal{F}$, any stochastic gradient oracle $\mathsf O \in \mathcal O$ and any \((\delta_{\rm max},\rho)\)-robust aggregator \(\mathsf A \in \mathcal A\), consider Algorithm \ref{algo-nc} with $\Gamma = \lceil 32L\Delta\epsilon^{-2} \rceil$ where $\Delta = f(\varkappa^0)-f^*$, the step size $\eta=\frac{1}{3L}$, $\alpha=0$, $\beta = \frac{\sqrt{3}-1}{\sqrt{3}+1}$, and $\theta = 1$. Then the output of Algorithm \ref{algo-nc} defined by $\tilde x^K$ satisfies
    \begin{align}\label{nverr-nc}
    \E[\|\nabla f(\tilde x^K)\|] \leq 16\sqrt{5}\rho^{1/2}\delta^{1/2}\zeta + \epsilon,
    \end{align}
    with oracle query complexity
    \begin{align}\label{clp-nc}
    K = O\left(\frac{L\Delta\rho\delta\sigma^2}{\epsilon^4}+\frac{L\Delta\sigma^2}{(1-\delta)n\epsilon^4}+\frac{L\Delta}{\epsilon^2}\log\frac{L\Delta(1+\rho\delta\zeta^2)}{\epsilon^2}\right).
    \end{align}
\end{theorem}

\begin{proof}
    See Appendix \ref{section:Theorem 18}.
\end{proof}

Comparing the upper bound of oracle query complexity in Theorem \ref{thm:nc} with the lower bound in Theorem \ref{thm:flb}, we notice that there is only a logarithmic factor gap between their third terms. The source of the third term in \ref{clp-nc} is that Algorithm \ref{algo-nc} solves \( O(\epsilon^{-2}) \) strongly convex surrogates, each surrogate calls Byrd-reNester once, each Byrd-reNester calls Byrd-Nester once, and the complexity of Byrd-Nester is \( O(\log \epsilon^{-2}) \). Nevertheless, when $\epsilon$ is small, the third term is not dominant compared to the first and second terms, such that the logarithmic factor gap is negligible.




For the following special cases, the established oracle query complexity of Algorithm \ref{algo-nc} remains optimal (up to a logarithmic factor).
\begin{itemize}
    \item {\bf Specialization to} $\zeta^2 = 0$. Setting $\zeta^2 = 0$ in \eqref{nverr-nc} yields
        \(
            \E[\|\nabla f(\tilde x^K)\|] \leq \epsilon
        \)
    and the oracle query complexity is \[
    O\left(\frac{L\Delta\rho\delta\sigma^2}{\epsilon^4}+\frac{L\Delta\sigma^2}{(1-\delta)n\epsilon^4}+\frac{L\Delta}{\epsilon^2}\log\frac{L\Delta}{\epsilon^2}\right).
    \] This complexity matches the lower bound for non-convex optimization with data homogeneity (\(\zeta^2=0\)), as shown in Theorem \ref{thm:flb}.
    \item {\bf Specialization to} $\delta = 0$. Setting $\delta = 0$ in \eqref{nverr-nc} and \eqref{clp-nc} demonstrate that to achieve \(
            \E[\|\nabla f(\tilde x^K)\|] \leq \epsilon
        \), the oracle query complexity is
    \begin{align*}
    O\left(\frac{L\Delta\sigma^2}{n\epsilon^4}+\frac{L\Delta}{\epsilon^2}\log\frac{L\Delta}{\epsilon^2}\right).
    \end{align*}
    This complexity matches the lower bound for non-convex optimization in the absence of Byzantine nodes (\(\delta=0\)), as shown in Theorem \ref{thm:flb}.
    \item {\bf Specialization to} $\sigma^2 = 0$. Setting $\sigma^2 = 0$ in \eqref{nverr-nc} and \eqref{clp-nc} yields
    \(
    \E[\|\nabla f(\tilde x^K)\|] \leq  16\sqrt{5}\rho^{1/2}\delta^{1/2}\zeta + \epsilon
    \)
    and the oracle query complexity is
    \[
    O\left(\frac{L\Delta}{\epsilon^2}\log\frac{L\Delta(1+\rho\delta\zeta^2)}{\epsilon^2}\right).
    \]
    This complexity matches the lower bound for non-convex optimization without sto- chasticity (\(\sigma^2=0\)), as shown in Theorem \ref{thm:flb}.
    \item {\bf Specialization to} $\sigma^2 = \zeta^2 = 0$ {\bf or} $\sigma^2 = \delta = 0$. Setting $\sigma^2 = \zeta^2 ({\rm or}~ \delta) = 0$ in \eqref{nverr-nc} and \eqref{clp-nc} yields
    \(
        \E[\|\nabla f(\tilde x^K)\|] \leq \epsilon
    \)
    and the oracle query complexity is
    \[
    O\left(\frac{L\Delta}{\epsilon^2}\log\frac{L\Delta}{\epsilon^2}\right).
    \]
    This complexity exactly matches the lower bound for non-convex optimization without stochasticity but with data homogeneity ($\sigma^2 = \zeta^2 = 0$) or without stochasticity but in the absence of Byzantine nodes ($\sigma^2 = \delta = 0$), as shown in Theorem \ref{thm:flb}.
\end{itemize}

\section{Numerical experiments}\label{sec:exp}


In this section, we conduct extensive numerical experiments to evaluate the performance of Algorithm \ref{subalgo}. Here we do not consider Algorithms \ref{algo:restart} and \ref{algo-nc}, which exhibit strong theoretical guarantees at the cost of complicated hyperparameter tuning.

\begin{figure}[ht]
    \hspace{-0.01\linewidth}
    \includegraphics[width=\linewidth]{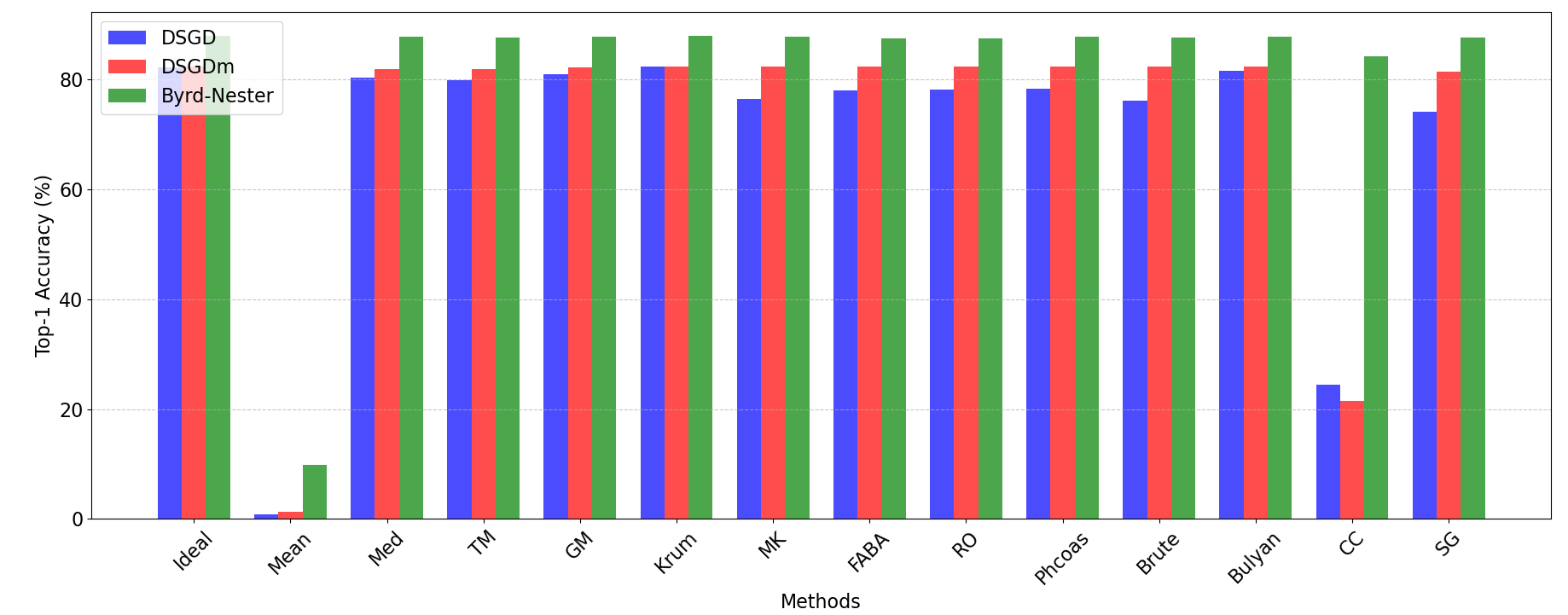}
    \caption{Worst-case maximum top-1 accuracy of DSGD, DSGDm and Algorithm 1.}
    \label{fig:worst}
\end{figure}

\noindent {\bf Experimental setup.} We consider two tasks, logistic regression and convolutional neural network training. For the first task, we consider a distributed network of 10 nodes within which 2 are Byzantine. For the second task, we consider a distributed network of 30 nodes within which 5 are Byzantine. The training dataset is MNIST with 10 classes, each having 6,000 training samples. The entire training dataset is sorted by labels, divided into chunks equal to the number of honest nodes, allocated to different honest nodes, and then shuffled within each honest node.
%

\noindent {\bf Byzantine attacks.} We implement nine Byzantine attacks, including ``Gaussian Attack (GA) \citep{ye2024ge}", ``Sign Flipping (SF) \citep{li2019rsa}", ``Label Flipping (LF) \citep{xiao2012adversarial}", ``Sample Duplicating \citep{prakash2020mitigating}", ``Zero Value", ``Isolation \citep{song2020analyzing}", ``A Little is Enough (ALIE) \citep{baruch2019little}", ``Inner Product Manipulation (IPM) \citep{xie2020fall}", and ``Bit Flipping (BF) \citep{rakin2019bit}".

\noindent {\bf Robust aggregation rules.} We implement fourteen robust aggregation rules, including ``Ideal", ``Mean", ``Median (Med) \citep{yin2018byzantine}", ``Trimmed Mean (TM) \citep{yin2018byzantine}", ``Krum \citep{blanchard2017machine}", ``Multi Krum (MK) \citep{blanchard2017machine}", ``FABA \citep{xia2019faba}", ``Remove Outliers (RO) \citep{xia2020defenses}", ``Phocas \citep{xie2018phocas}", ``Brute \citep{guerraoui2018hidden}", ``Bulyan \citep{guerraoui2018hidden}", ``Centered Clipping(CC) \citep{karimireddy2021learning}", ``Geometric Median (GM) \citep{wu2020federated}", and ``Sign Guard (SG) \citep{xu2021signguard}''.

\noindent {\bf Baselines.} We choose Byzantine-robust distributed mini-batch SGD (DSGD) and its momentum variant (DSGDm, \cite{karimireddy2020byzantine}) as the baselines. In the three algorithms, the step size is set to 0.1, the batch size is 32, and the total number of epoches is 45. Because the combinations of the compared algorithms, Byzantine attacks and robust aggregation rules are immense, below we only demonstrate some of the results. More results can be found via running our source code at \url{https://github.com/sqkkk/Byrd-Nester}

\begin{figure}[h!]
    \hspace{-0.1\linewidth}
    \includegraphics[width=1.2\linewidth]{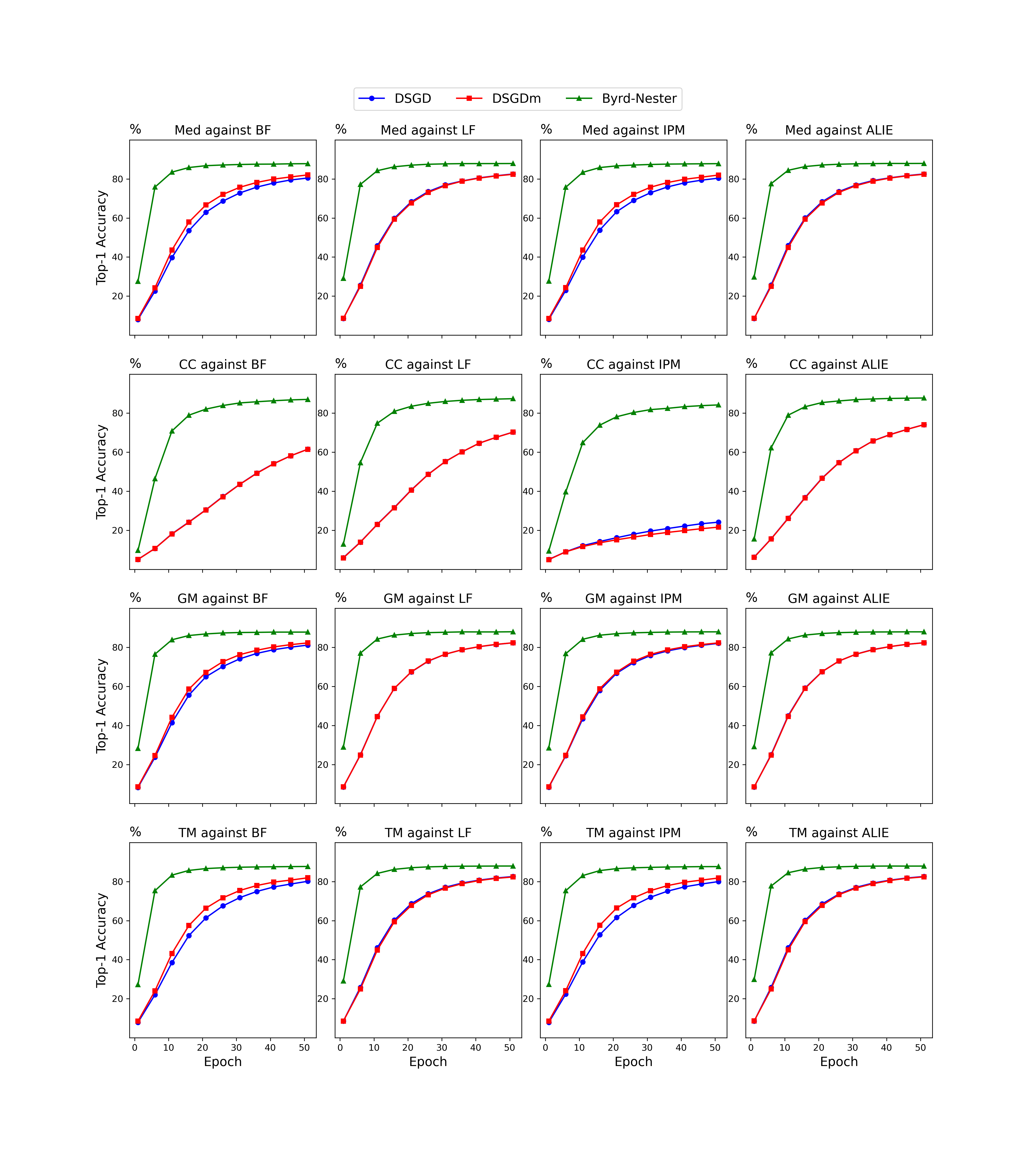}\vspace{-1cm}
    \caption{Top-$1$ test accuracies of DSGD, DSGDm and Algorithm 1 with Med, CC, GM and TM for logistic regression, under BF, LF, IPM and ALIE attacks.}
    \label{fig:expsc}
\end{figure}

\subsection{Logistic regression}

First, we consider logistic regression with squared $l_2$ regularization. For each combination of the three compared algorithms and the fourteen robust aggregation rules, under each of the nine Byzantine attacks, we record the maximum accuracy obtained within the total number of epoches. Then, we depict the minimum of the nine maximum accuracies (termed as the worst-case maximum accuracy) in Figure \ref{fig:worst}. This performance metric is of practical importance, as it reflects the ability of each combination under the worst-case attack. We observe that Algorithm \ref{subalgo} outperforms DSGD and DSGDm when combined with most of the robust aggregation rules.


Figure \ref{fig:expsc} depicts the convergence of the three compared algorithms under BF, LF, IPM and ALIE attacks when combined with Med, CC, GM and TM aggregation rules. DSGD and DSGDm perform similarly, whereas Algorithm \ref{subalgo} enjoys faster convergence and higher accuracy thanks to its effective usage of the historical stochastic gradients.

\begin{figure}[h!]
    \hspace{-0.1\linewidth}
    \includegraphics[width=1.2\linewidth]{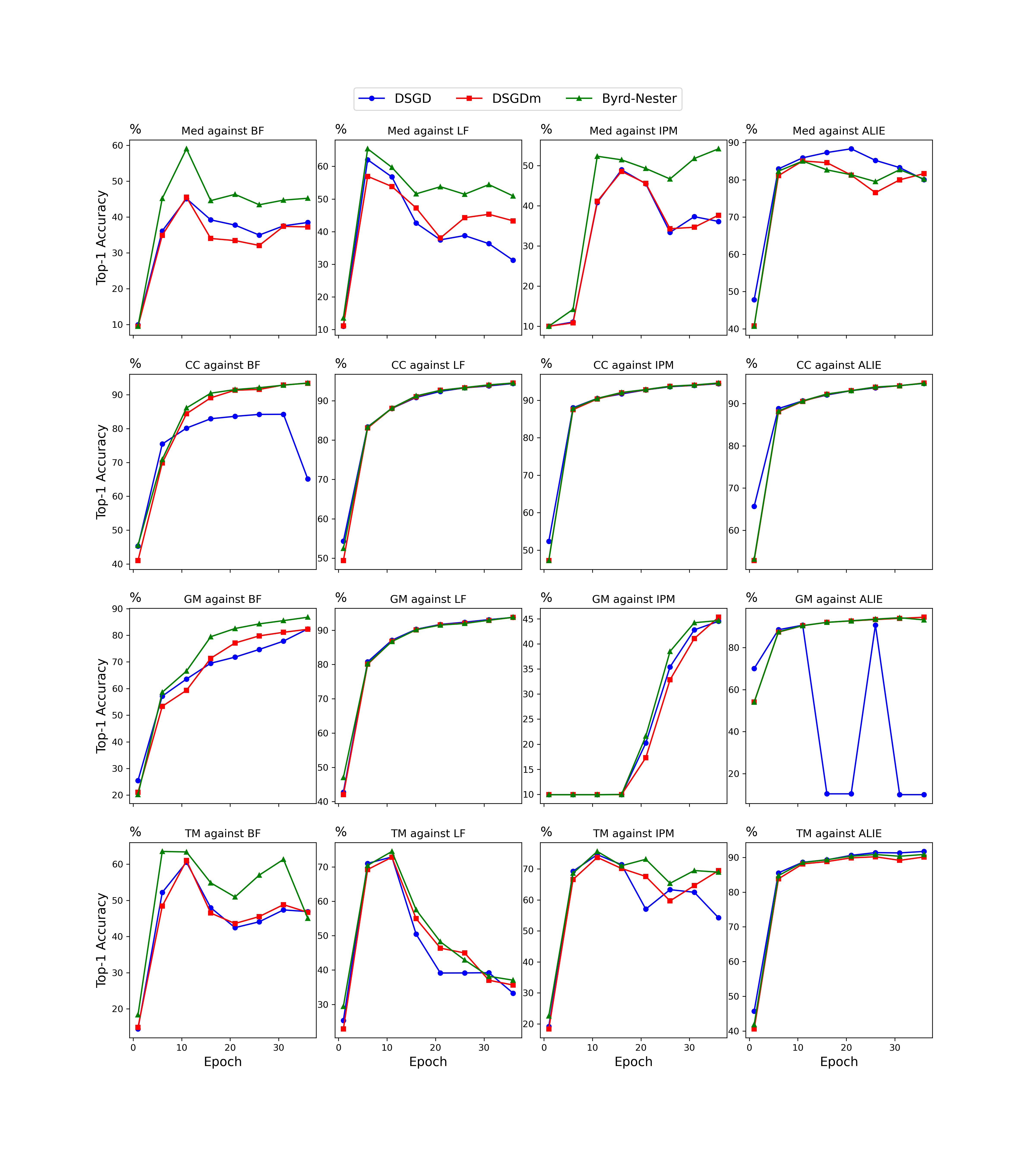}\vspace{-1cm}
    \caption{Top-$1$ test accuracies of DSGD, DSGDm and Algorithm 1 with Med, CC, GM and TM for convolutional neural network training, under BF, LF, IPM and ALIE attacks.}
    \label{fig:expnc}
\end{figure}

\subsection{Convolutional neural network training}

Second, we consider training a convolutional neural network that consists of two convolutional layers, followed by two fully connected layers. Figure \ref{fig:expnc} depicts the top-1 test accuracies of the three compared algorithms under BF, LF, IPM and ALIE attacks when combined with Med, CC, GM and TM aggregation rules. Algorithm \ref{subalgo} gains the best performance in most cases, while DSGDm is also competitive.

\section{Conclusions} \label{sec:con}

We established tight lower bounds for Byzantine-robust distributed first-order stochastic optimization methods in both strongly convex and non-convex stochastic optimization. A key observation was that with data heterogeneity, the convergence error contains a non-vanishing Byzantine error and a vanishing optimization error. Therefore, we respectively established the lower bounds on the Byzantine error and the oracle query complexity to achieve an arbitrarily small optimization error. In contrast, the analysis in \citep{alistarh2018byzantine} was confined to homogeneous data distribution and did not account for the Byzantine error, while the work of \citep{karimireddy2020byzantine} did not explore the oracle query complexity. We also observed significant discrepancies between our established lower bounds and the existing upper bounds. To fill this gap, we leveraged the techniques of Nesterov's acceleration and variance reduction to develop novel Byzantine-robust distributed stochastic optimization methods that provably match these lower bounds, up to logarithmic factors. This fact implies that our established lower bounds are tight, and our proposed methods can simultaneously attain the optimal Byzantine robustness and the optimal oracle query complexity. Our future work is to explore the extension to Byzantine-robust decentralized stochastic optimization without the aid of any server.

\newpage

\vskip 0.2in
\bibliography{sample}

\begin{thebibliography}{67}
\providecommand{\natexlab}[1]{#1}
\providecommand{\url}[1]{\texttt{#1}}
\expandafter\ifx\csname urlstyle\endcsname\relax
  \providecommand{\doi}[1]{doi: #1}\else
  \providecommand{\doi}{doi: \begingroup \urlstyle{rm}\Url}\fi

\bibitem[Alistarh et~al.(2018)Alistarh, Allen-Zhu, and
  Li]{alistarh2018byzantine}
Dan Alistarh, Zeyuan Allen-Zhu, and Jerry Li.
\newblock Byzantine stochastic gradient descent.
\newblock \emph{Advances in Neural Information Processing Systems}, 31, 2018.

\bibitem[Allouah et~al.(2023)Allouah, Farhadkhani, Guerraoui, Gupta, Pinot, and
  Stephan]{allouah2023fixing}
Youssef Allouah, Sadegh Farhadkhani, Rachid Guerraoui, Nirupam Gupta,
  Rafa{\"e}l Pinot, and John Stephan.
\newblock Fixing by mixing: A recipe for optimal byzantine ml under
  heterogeneity.
\newblock In \emph{International Conference on Artificial Intelligence and
  Statistics}, pages 1232--1300, 2023.

\bibitem[Arjevani et~al.(2023)Arjevani, Carmon, Duchi, Foster, Srebro, and
  Woodworth]{arjevani2023lower}
Yossi Arjevani, Yair Carmon, John~C Duchi, Dylan~J Foster, Nathan Srebro, and
  Blake Woodworth.
\newblock Lower bounds for non-convex stochastic optimization.
\newblock \emph{Mathematical Programming}, 199\penalty0 (1-2):\penalty0
  165--214, 2023.

\bibitem[Attias et~al.(2022)Attias, Kontorovich, and
  Mansour]{attias2022improved}
Idan Attias, Aryeh Kontorovich, and Yishay Mansour.
\newblock Improved generalization bounds for adversarially robust learning.
\newblock \emph{Journal of Machine Learning Research}, 23\penalty0
  (175):\penalty0 1--31, 2022.

\bibitem[Baruch et~al.(2019)Baruch, Baruch, and Goldberg]{baruch2019little}
Gilad Baruch, Moran Baruch, and Yoav Goldberg.
\newblock A little is enough: Circumventing defenses for distributed learning.
\newblock \emph{Advances in Neural Information Processing Systems}, 32, 2019.

\bibitem[Blanchard et~al.(2017)Blanchard, El~Mhamdi, Guerraoui, and
  Stainer]{blanchard2017machine}
Peva Blanchard, El~Mahdi El~Mhamdi, Rachid Guerraoui, and Julien Stainer.
\newblock Machine learning with adversaries: Byzantine tolerant gradient
  descent.
\newblock \emph{Advances in Neural Information Processing Systems}, 30, 2017.

\bibitem[Bottou et~al.(2018)Bottou, Curtis, and
  Nocedal]{bottou2018optimization}
L{\'e}on Bottou, Frank~E Curtis, and Jorge Nocedal.
\newblock Optimization methods for large-scale machine learning.
\newblock \emph{SIAM review}, 60\penalty0 (2):\penalty0 223--311, 2018.

\bibitem[Brown et~al.(2020)Brown, Mann, Ryder, Subbiah, Kaplan, Dhariwal,
  Neelakantan, Shyam, Sastry, Askell, et~al.]{brown2020language}
Tom Brown, Benjamin Mann, Nick Ryder, Melanie Subbiah, Jared~D Kaplan, Prafulla
  Dhariwal, Arvind Neelakantan, Pranav Shyam, Girish Sastry, Amanda Askell,
  et~al.
\newblock Language models are few-shot learners.
\newblock \emph{Advances in Neural Information Processing Systems}, 33, 2020.

\bibitem[Cao and Lai(2020)]{cao2020distributed}
Xinyang Cao and Lifeng Lai.
\newblock Distributed approximate {N}ewton's method robust to {B}yzantine
  attackers.
\newblock \emph{IEEE Transactions on Signal Processing}, 68:\penalty0
  6011--6025, 2020.

\bibitem[Carmon et~al.(2020)Carmon, Duchi, Hinder, and
  Sidford]{carmon2020lower}
Yair Carmon, John~C Duchi, Oliver Hinder, and Aaron Sidford.
\newblock Lower bounds for finding stationary points i.
\newblock \emph{Mathematical Programming}, 184\penalty0 (1-2):\penalty0
  71--120, 2020.

\bibitem[Carmon et~al.(2021)Carmon, Duchi, Hinder, and
  Sidford]{carmon2021lower}
Yair Carmon, John~C Duchi, Oliver Hinder, and Aaron Sidford.
\newblock Lower bounds for finding stationary points ii: first-order methods.
\newblock \emph{Mathematical Programming}, 185\penalty0 (1-2):\penalty0
  315--355, 2021.

\bibitem[Chen et~al.(2018)Chen, Wang, Charles, and
  Papailiopoulos]{chen2018draco}
Lingjiao Chen, Hongyi Wang, Zachary Charles, and Dimitris Papailiopoulos.
\newblock Draco: Byzantine-resilient distributed training via redundant
  gradients.
\newblock In \emph{International Conference on Machine Learning}, pages
  903--912, 2018.

\bibitem[Chen et~al.(2017)Chen, Su, and Xu]{chen2017distributed}
Yudong Chen, Lili Su, and Jiaming Xu.
\newblock Distributed statistical machine learning in adversarial settings:
  Byzantine gradient descent.
\newblock \emph{Proceedings of the ACM on Measurement and Analysis of Computing
  Systems}, 1\penalty0 (2):\penalty0 1--25, 2017.

\bibitem[Data and Diggavi(2021)]{data2021byzantine}
Deepesh Data and Suhas Diggavi.
\newblock Byzantine-resilient sgd in high dimensions on heterogeneous data.
\newblock In \emph{2021 IEEE International Symposium on Information Theory},
  pages 2310--2315, 2021.

\bibitem[Egger et~al.(2025)Egger, Bakshi, and Bitar]{egger2025byzantine}
Maximilian Egger, Mayank Bakshi, and Rawad Bitar.
\newblock Byzantine-resilient zero-order optimization for
  communication-efficient heterogeneous federated learning.
\newblock \emph{arXiv preprint arXiv:2502.00193}, 2025.

\bibitem[El-Mhamdi et~al.(2021)El-Mhamdi, Farhadkhani, Guerraoui, Guirguis,
  Hoang, and Rouault]{el2021collaborative}
El~Mahdi El-Mhamdi, Sadegh Farhadkhani, Rachid Guerraoui, Arsany Guirguis,
  L{\^e}-Nguy{\^e}n Hoang, and S{\'e}bastien Rouault.
\newblock Collaborative learning in the jungle (decentralized, byzantine,
  heterogeneous, asynchronous and nonconvex learning).
\newblock \emph{Advances in Neural Information Processing Systems},
  34:\penalty0 25044--25057, 2021.

\bibitem[Fang et~al.(2018)Fang, Li, Lin, and Zhang]{fang2018spider}
Cong Fang, Chris~Junchi Li, Zhouchen Lin, and Tong Zhang.
\newblock Spider: Near-optimal non-convex optimization via stochastic
  path-integrated differential estimator.
\newblock \emph{Advances in Neural Information Processing Systems}, 31, 2018.

\bibitem[Farhadkhani et~al.(2022)Farhadkhani, Guerraoui, Gupta, Pinot, and
  Stephan]{farhadkhani2022byzantine}
Sadegh Farhadkhani, Rachid Guerraoui, Nirupam Gupta, Rafael Pinot, and John
  Stephan.
\newblock Byzantine machine learning made easy by resilient averaging of
  momentums.
\newblock In \emph{International Conference on Machine Learning}, pages
  6246--6283, 2022.

\bibitem[Foster et~al.(2019)Foster, Sekhari, Shamir, Srebro, Sridharan, and
  Woodworth]{foster2019complexity}
Dylan~J Foster, Ayush Sekhari, Ohad Shamir, Nathan Srebro, Karthik Sridharan,
  and Blake Woodworth.
\newblock The complexity of making the gradient small in stochastic convex
  optimization.
\newblock In \emph{Conference on Learning Theory}, pages 1319--1345, 2019.

\bibitem[Ghosh et~al.(2020)Ghosh, Maity, and Mazumdar]{ghosh2020distributed}
Avishek Ghosh, Raj~Kumar Maity, and Arya Mazumdar.
\newblock Distributed newton can communicate less and resist byzantine workers.
\newblock \emph{Advances in Neural Information Processing Systems},
  33:\penalty0 18028--18038, 2020.

\bibitem[Gorbunov et~al.(2022)Gorbunov, Horv{\'a}th, Richt{\'a}rik, and
  Gidel]{gorbunov2022variance}
Eduard Gorbunov, Samuel Horv{\'a}th, Peter Richt{\'a}rik, and Gauthier Gidel.
\newblock Variance reduction is an antidote to byzantines: Better rates, weaker
  assumptions and communication compression as a cherry on the top.
\newblock 2022.

\bibitem[Guerraoui et~al.(2018)Guerraoui, Rouault, et~al.]{guerraoui2018hidden}
Rachid Guerraoui, S{\'e}bastien Rouault, et~al.
\newblock The hidden vulnerability of distributed learning in byzantium.
\newblock In \emph{International Conference on Machine Learning}, pages
  3521--3530, 2018.

\bibitem[Guerraoui et~al.(2024)Guerraoui, Gupta, and
  Pinot]{guerraoui2023byzantine}
Rachid Guerraoui, Nirupam Gupta, and Rafael Pinot.
\newblock Byzantine machine learning: A primer.
\newblock \emph{ACM Computing Surveys}, 56\penalty0 (7):\penalty0 1--39, 2024.

\bibitem[Huang et~al.(2022)Huang, Chen, Yin, and Yuan]{huang2022lower}
Xinmeng Huang, Yiming Chen, Wotao Yin, and Kun Yuan.
\newblock Lower bounds and nearly optimal algorithms in distributed learning
  with communication compression.
\newblock \emph{Advances in Neural Information Processing Systems}, 35, 2022.

\bibitem[Karimireddy et~al.(2021)Karimireddy, He, and
  Jaggi]{karimireddy2021learning}
Sai~Praneeth Karimireddy, Lie He, and Martin Jaggi.
\newblock Learning from history for byzantine robust optimization.
\newblock In \emph{International Conference on Machine Learning}, pages
  5311--5319, 2021.

\bibitem[Karimireddy et~al.(2022)Karimireddy, He, and
  Jaggi]{karimireddy2020byzantine}
Sai~Praneeth Karimireddy, Lie He, and Martin Jaggi.
\newblock Byzantine-robust learning on heterogeneous datasets via bucketing.
\newblock \emph{International Conference on Learning Representations}, 2022.

\bibitem[Koushkbaghi et~al.(2024)Koushkbaghi, Safi, Amani, Jalili, and
  Yu]{koushkbaghi2024byzantine}
Sajad Koushkbaghi, Mostafa Safi, Ali~Moradi Amani, Mahdi Jalili, and Xinghuo
  Yu.
\newblock Byzantine-resilient second-order consensus in networked systems.
\newblock \emph{IEEE Transactions on Cybernetics}, 2024.

\bibitem[Lamport et~al.(1982)Lamport, Shostak, and Pease]{lamport1982byzantine}
Leslie Lamport, Robert Shostak, and Marshall Pease.
\newblock The byzantine generals problem.
\newblock \emph{ACM Transactions on Programming Languages and Systems},
  4\penalty0 (3):\penalty0 382--401, 1982.

\bibitem[Lewis et~al.(2023)Lewis, Varadharajan, and Noman]{lewis2023attacks}
Cody Lewis, Vijay Varadharajan, and Nasimul Noman.
\newblock Attacks against federated learning defense systems and their
  mitigation.
\newblock \emph{Journal of Machine Learning Research}, 24\penalty0
  (30):\penalty0 1--50, 2023.

\bibitem[Li et~al.(2019)Li, Xu, Chen, Giannakis, and Ling]{li2019rsa}
Liping Li, Wei Xu, Tianyi Chen, Georgios~B Giannakis, and Qing Ling.
\newblock Rsa: Byzantine-robust stochastic aggregation methods for distributed
  learning from heterogeneous datasets.
\newblock In \emph{Proceedings of the AAAI conference on artificial
  intelligence}, volume~33, 2019.

\bibitem[Li et~al.(2021)Li, Bao, Zhang, and Richt{\'a}rik]{li2021page}
Zhize Li, Hongyan Bao, Xiangliang Zhang, and Peter Richt{\'a}rik.
\newblock Page: A simple and optimal probabilistic gradient estimator for
  nonconvex optimization.
\newblock In \emph{International Conference on Machine Learning}, pages
  6286--6295, 2021.

\bibitem[Lian et~al.(2017)Lian, Zhang, Zhang, Hsieh, Zhang, and
  Liu]{lian2017can}
Xiangru Lian, Ce~Zhang, Huan Zhang, Cho-Jui Hsieh, Wei Zhang, and Ji~Liu.
\newblock Can decentralized algorithms outperform centralized algorithms? a
  case study for decentralized parallel stochastic gradient descent.
\newblock \emph{Advances in neural information processing systems}, 30, 2017.

\bibitem[Liu et~al.(2020)Liu, Gao, and Yin]{liu2020improved}
Yanli Liu, Yuan Gao, and Wotao Yin.
\newblock An improved analysis of stochastic gradient descent with momentum.
\newblock \emph{Advances in Neural Information Processing Systems}, 33, 2020.

\bibitem[Liu et~al.(2024)Liu, Wang, and Yuan]{liu2024badsampler}
Yi~Liu, Cong Wang, and Xingliang Yuan.
\newblock Badsampler: Harnessing the power of catastrophic forgetting to poison
  byzantine-robust federated learning.
\newblock In \emph{Proceedings of the 30th ACM SIGKDD Conference on Knowledge
  Discovery and Data Mining}, pages 1944--1955, 2024.

\bibitem[Lu and De~Sa(2021)]{lu2021optimal}
Yucheng Lu and Christopher De~Sa.
\newblock Optimal complexity in decentralized training.
\newblock In \emph{International Conference on Machine Learning}, pages
  7111--7123, 2021.

\bibitem[Mahloujifar et~al.(2019)Mahloujifar, Mahmoody, and
  Mohammed]{mahloujifar2019universal}
Saeed Mahloujifar, Mohammad Mahmoody, and Ameer Mohammed.
\newblock Universal multi-party poisoning attacks.
\newblock In \emph{International Conference on Machine Learning}, pages
  4274--4283, 2019.

\bibitem[Nemirovski and Yudin(1983)]{nemirovskij1983problem}
Arkadi~Semenovich Nemirovski and David~Borisovich Yudin.
\newblock \emph{Problem complexity and method efficiency in optimization}.
\newblock Wiley-Interscience, 1983.

\bibitem[Nesterov(2003)]{nesterov2003introductory}
Yurii Nesterov.
\newblock \emph{Introductory lectures on convex optimization: A basic course},
  volume~87.
\newblock Springer Science \& Business Media, 2003.

\bibitem[OpenAI et~al.(2024)]{openai2024gpt4technicalreport}
OpenAI et~al.
\newblock Gpt-4 technical report.
\newblock \emph{ArXiv}, 2303:\penalty0 08774, 2024.

\bibitem[Peng et~al.(2024)Peng, Li, and Ling]{peng2024mean}
Jie Peng, Weiyu Li, and Qing Ling.
\newblock Mean aggregator is more robust than robust aggregators under label
  poisoning attacks.
\newblock \emph{arXiv preprint arXiv:2404.13647}, 2024.

\bibitem[Peng et~al.(2025)Peng, Li, Vlaski, and Ling]{peng2025mean}
Jie Peng, Weiyu Li, Stefan Vlaski, and Qing Ling.
\newblock Mean aggregator is more robust than robust aggregators under label
  poisoning attacks on distributed heterogeneous data.
\newblock \emph{Journal of Machine Learning Research}, 26\penalty0
  (27):\penalty0 1--51, 2025.

\bibitem[Pillutla et~al.(2022)Pillutla, Kakade, and
  Harchaoui]{pillutla2022robust}
Krishna Pillutla, Sham~M Kakade, and Zaid Harchaoui.
\newblock Robust aggregation for federated learning.
\newblock \emph{IEEE Transactions on Signal Processing}, 70:\penalty0
  1142--1154, 2022.

\bibitem[Polyak(1964)]{polyak1964some}
Boris~T Polyak.
\newblock Some methods of speeding up the convergence of iteration methods.
\newblock \emph{Ussr Computational Mathematics and Mathematical physics},
  4\penalty0 (5):\penalty0 1--17, 1964.

\bibitem[Prakash and Avestimehr(2020)]{prakash2020mitigating}
Saurav Prakash and Amir~Salman Avestimehr.
\newblock Mitigating byzantine attacks in federated learning.
\newblock \emph{arXiv preprint arXiv:2010.07541}, 2020.

\bibitem[Rakin et~al.(2019)Rakin, He, and Fan]{rakin2019bit}
Adnan~Siraj Rakin, Zhezhi He, and Deliang Fan.
\newblock Bit-flip attack: Crushing neural network with progressive bit search.
\newblock In \emph{International Conference on Computer Vision}, pages
  1211--1220, 2019.

\bibitem[Reddi et~al.(2020)Reddi, Charles, Zaheer, Garrett, Rush,
  Kone{\v{c}}n{\`y}, Kumar, and McMahan]{reddi2020adaptive}
Sashank Reddi, Zachary Charles, Manzil Zaheer, Zachary Garrett, Keith Rush,
  Jakub Kone{\v{c}}n{\`y}, Sanjiv Kumar, and H~Brendan McMahan.
\newblock Adaptive federated optimization.
\newblock \emph{arXiv preprint arXiv:2003.00295}, 2020.

\bibitem[Ross(2014)]{ross2014introduction}
Sheldon~M Ross.
\newblock \emph{Introduction to probability models}.
\newblock Academic Press, 2014.

\bibitem[Scaman et~al.(2017)Scaman, Bach, Bubeck, Lee, and
  Massouli{\'e}]{scaman2017optimal}
Kevin Scaman, Francis Bach, S{\'e}bastien Bubeck, Yin~Tat Lee, and Laurent
  Massouli{\'e}.
\newblock Optimal algorithms for smooth and strongly convex distributed
  optimization in networks.
\newblock In \emph{International Conference on Machine Learning}, pages
  3027--3036, 2017.

\bibitem[Song et~al.(2020)Song, Wang, Zhang, Song, Wang, Ren, and
  Qi]{song2020analyzing}
Mengkai Song, Zhibo Wang, Zhifei Zhang, Yang Song, Qian Wang, Ju~Ren, and
  Hairong Qi.
\newblock Analyzing user-level privacy attack against federated learning.
\newblock \emph{IEEE Journal on Selected Areas in Communications}, 38\penalty0
  (10):\penalty0 2430--2444, 2020.

\bibitem[Su and Vaidya(2016)]{su2016fault}
Lili Su and Nitin~H Vaidya.
\newblock Fault-tolerant multi-agent optimization: optimal iterative
  distributed algorithms.
\newblock In \emph{Proceedings of the 2016 ACM symposium on principles of
  distributed computing}, pages 425--434, 2016.

\bibitem[Woodworth and Srebro(2016)]{woodworth2016tight}
Blake~E Woodworth and Nati Srebro.
\newblock Tight complexity bounds for optimizing composite objectives.
\newblock \emph{Advances in Neural Information Processing Systems}, 29, 2016.

\bibitem[Wu et~al.(2020)Wu, Ling, Chen, and Giannakis]{wu2020federated}
Zhaoxian Wu, Qing Ling, Tianyi Chen, and Georgios~B Giannakis.
\newblock Federated variance-reduced stochastic gradient descent with
  robustness to byzantine attacks.
\newblock \emph{IEEE Transactions on Signal Processing}, 68:\penalty0
  4583--4596, 2020.

\bibitem[Xia et~al.(2019)Xia, Tao, Hao, and Li]{xia2019faba}
Qi~Xia, Zeyi Tao, Zijiang Hao, and Qun Li.
\newblock Faba: an algorithm for fast aggregation against byzantine attacks in
  distributed neural networks.
\newblock In \emph{International Joint Conference on Artificial Intelligence},
  2019.

\bibitem[Xia et~al.(2020)Xia, Tao, and Li]{xia2020defenses}
Qi~Xia, Zeyi Tao, and Qun Li.
\newblock Defenses against byzantine attacks in distributed deep neural
  networks.
\newblock \emph{IEEE Transactions on Network Science and Engineering},
  8\penalty0 (3):\penalty0 2025--2035, 2020.

\bibitem[Xiao et~al.(2012)Xiao, Xiao, and Eckert]{xiao2012adversarial}
Han Xiao, Huang Xiao, and Claudia Eckert.
\newblock Adversarial label flips attack on support vector machines.
\newblock In \emph{European Conference on Artificial Intelligence}, pages
  870--875. 2012.

\bibitem[Xiao et~al.(2023)Xiao, Shao, Lin, Huo, and Liu]{xiao2023bce}
Yiming Xiao, Haidong Shao, Jian Lin, Zhiqiang Huo, and Bin Liu.
\newblock Bce-fl: a secure and privacy-preserving federated learning system for
  device fault diagnosis under non-iid condition in iiot.
\newblock \emph{IEEE Internet of Things Journal}, 2023.

\bibitem[Xie et~al.(2018)Xie, Koyejo, and Gupta]{xie2018phocas}
Cong Xie, Oluwasanmi Koyejo, and Indranil Gupta.
\newblock Phocas: dimensional byzantine-resilient stochastic gradient descent.
\newblock \emph{arXiv preprint arXiv:1805.09682}, 2018.

\bibitem[Xie et~al.(2019)Xie, Koyejo, and Gupta]{xie2019zeno}
Cong Xie, Sanmi Koyejo, and Indranil Gupta.
\newblock Zeno: Distributed stochastic gradient descent with suspicion-based
  fault-tolerance.
\newblock In \emph{International Conference on Machine Learning}, pages
  6893--6901. PMLR, 2019.

\bibitem[Xie et~al.(2020)Xie, Koyejo, and Gupta]{xie2020fall}
Cong Xie, Oluwasanmi Koyejo, and Indranil Gupta.
\newblock Fall of empires: Breaking byzantine-tolerant sgd by inner product
  manipulation.
\newblock In \emph{Uncertainty in Artificial Intelligence}, pages 261--270,
  2020.

\bibitem[Xu et~al.(2022)Xu, Huang, Song, and Lan]{xu2021signguard}
Jian Xu, Shao-Lun Huang, Linqi Song, and Tian Lan.
\newblock Signguard: Byzantine-robust federated learning through collaborative
  malicious gradient filtering.
\newblock \emph{International Conference on Distributed Computing Systems},
  2022.

\bibitem[Yang and Li(2023)]{yang2023buffered}
Yi-Rui Yang and Wu-Jun Li.
\newblock Buffered asynchronous sgd for byzantine learning.
\newblock \emph{Journal of Machine Learning Research}, 24\penalty0
  (204):\penalty0 1--62, 2023.

\bibitem[Ye and Ling(2024)]{ye2024ge}
Haoxiang Ye and Qing Ling.
\newblock On the generalization error of {B}yzantine-resilient decentralized
  learning.
\newblock \emph{International Conference on Acoustics, Speech and Signal
  Processing}, 2024.

\bibitem[Ye and Ling(2025)]{ye2025generalization}
Haoxiang Ye and Qing Ling.
\newblock Generalization error matters in decentralized learning under
  byzantine attacks.
\newblock \emph{IEEE Transactions on Signal Processing}, 2025.

\bibitem[Yin et~al.(2018)Yin, Chen, Kannan, and Bartlett]{yin2018byzantine}
Dong Yin, Yudong Chen, Ramchandran Kannan, and Peter Bartlett.
\newblock Byzantine-robust distributed learning: Towards optimal statistical
  rates.
\newblock In \emph{International Conference on Machine Learning}, pages
  5650--5659, 2018.

\bibitem[Yuan et~al.(2022)Yuan, Huang, Chen, Zhang, Zhang, and
  Pan]{yuan2022revisiting}
Kun Yuan, Xinmeng Huang, Yiming Chen, Xiaohan Zhang, Yingya Zhang, and Pan Pan.
\newblock Revisiting optimal convergence rate for smooth and non-convex
  stochastic decentralized optimization.
\newblock \emph{Advances in Neural Information Processing Systems}, 35, 2022.

\bibitem[Zhang et~al.(2020)Zhang, Lu, Yu, Li, Liu, Lo, Chen, Xu, and
  Zhu]{zhang2020blockchain}
Weishan Zhang, Qinghua Lu, Qiuyu Yu, Zhaotong Li, Yue Liu, Sin~Kit Lo, Shiping
  Chen, Xiwei Xu, and Liming Zhu.
\newblock Blockchain-based federated learning for device failure detection in
  industrial iot.
\newblock \emph{IEEE Internet of Things Journal}, 8\penalty0 (7):\penalty0
  5926--5937, 2020.

\bibitem[Zhu et~al.(2023)Zhu, Wang, Pang, Wang, Jiao, Song, and
  Jordan]{zhu2023byzantine}
Banghua Zhu, Lun Wang, Qi~Pang, Shuai Wang, Jiantao Jiao, Dawn Song, and
  Michael~I Jordan.
\newblock Byzantine-robust federated learning with optimal statistical rates.
\newblock In \emph{International Conference on Artificial Intelligence and
  Statistics}, pages 3151--3178, 2023.

\end{thebibliography}

\newpage

\appendix
\section{Robustness coefficient of Centered Clipping (CC)}
\label{appendix:A}

We consider single-iteration CC with a sufficiently good starting point. Single-iteration CC outputs
\[
w = \frac{1}{n}\sum_{i=1}^n s_i ~~{\rm with}~~ s_i =  v + (w_i-v)\min \left(1,\frac{\tau}{\|w_i-v\|}\right),
\]
where $v$ is the starting point satisfying {$\|v - \bar w\|^2 \leq \frac{1}{|\mathcal H|}\sum_{i \in \mathcal H}\|w_i - \bar w\|^2$} and $\tau \geq 0$ is the clipping threshold. The output can also be rewritten as
\[
w = (1-\delta) \frac{1}{|\mathcal H|}\sum_{i \in \mathcal H}s_i + \delta  \frac{1}{|\mathcal B|} \sum_{i \in \mathcal B}s_i.
\]

Hence, we have
\begin{align}
    \|w - \bar w\|^2 &= \|(1-\delta) \frac{1}{|\mathcal H|}\sum_{i \in \mathcal H}s_i + \delta  \frac{1}{|\mathcal B|} \sum_{i \in \mathcal B}s_i - \bar w\|^2\\
    & \leq 2(1-\delta)^2 \|\frac{1}{|\mathcal H|}\sum_{i \in \mathcal H}(s_i - w_i) \|^2 + 2\delta^2 \| \frac{1}{|\mathcal B|} \sum_{i \in \mathcal B}(s_i- \bar w)\|^2 \notag \\
    & \leq 2(1-\delta)^2 \frac{1}{|\mathcal H|}\sum_{i \in \mathcal H}\|s_i - w_i \|^2 + 2\delta^2 \frac{1}{|\mathcal B|} \sum_{i \in \mathcal B}\| s_i- \bar w\|^2. \notag
\end{align}
For $i \in \mathcal H$, if $w_i$ is not clipped, we have $s_i = w_i$. Otherwise, we have
\[
\|s_i - w_i\| \leq  \frac{\| v - w_i\|^2}{\tau}  \leq  \frac{2\|v - \bar w\|^2 +  2\|w_i - \bar w\|^2}{\tau}.
\]
For $i \in \mathcal B$, we have
\[
\|s_i - \bar w\|^2 \leq  2\| s_i - v\|^2 + 2 \|v - \bar w\|^2 \leq 2\tau^2 + 2 \|v - \bar w\|^2.
\]

Therefore, setting $\tau^2 = \frac{1-\delta}{\delta}\sqrt{\frac{2}{|\mathcal H|}\sum_{i \in \mathcal H}(\|v - \bar w\|^2 + \|w_i - \bar w \|^2)}$ yields
\begin{align}
    \|w - \bar w\|^2 & \leq 2(1-\delta)^2 \frac{1}{|\mathcal H|}\sum_{i \in \mathcal H}\frac{(2\|v - \bar w\|^2 +  2\|w_i - \bar w\|^2)^2}{\tau^2} + 4\delta^2\tau^2 + 4 \delta^2\|v - \bar w\|^2\\
    &\leq \frac{8\sqrt{2}\delta(1-\delta)}{\sqrt{|\mathcal H|}}\sum_{i \in \mathcal H}(\|v - \bar w\|^2 +  \|w_i - \bar w\|^2)+4\delta^2 \|v - \bar w\|^2 \notag \\
    &\leq \left(\frac{8\sqrt{2}\delta(1-\delta)}{\sqrt{|\mathcal H|}}+4\delta^2+8\sqrt{2}\delta(1-\delta)\sqrt{|\mathcal H|}\right)\frac{1}{|\mathcal H|}\sum_{i \in \mathcal H}\|w_i - \bar w\|^2 \notag \\
    &\leq \frac{18\sqrt{2}\delta\sqrt{|\mathcal H|}}{|\mathcal H|}\sum_{i \in \mathcal H}\|w_i - \bar w\|^2, \notag
\end{align}
which completes the analysis.

\section{Proofs of main results}\label{appendix:proofs}

\subsection{Proof of Lemma \ref{le:nverr}}\label{proof:nverr}
\begin{proof}
We prove Lemma \ref{le:nverr} through constructing two one-dimensional deterministic problems without any Byzantine nodes, such that any method $\mathsf M \in \mathcal{M}$, equipped with a certain $(\delta_{\rm max},\rho)$-robust aggregator $\mathsf A \in \mathcal A$, cannot distinguish the two. Recall that the estimated fraction of the Byzantine nodes $\delta$ satisfies $0 < \delta \leq \delta_{\rm max} < 0.5$. In the following proof, we assume that \(\delta n\) is an integer. Otherwise, the conclusion still holds true if we rounding down \(\delta n\) in the derivation.

In the first problem, the function and the gradient of node \(i\) are respectively defined as
\[
f_{1,i}(x) =   \begin{cases}
             \frac{1}{2}x^2 - \delta^{-1/2} \zeta x, ~~~ &  i =1,\cdots,\delta n,\\
            \frac{1}{2}x^2, &  i =\delta n + 1,\cdots, n,
             \end{cases}
\]
\[
\nabla f_{1,i}(x) =   \begin{cases}
            x - \delta^{-1/2} \zeta, ~~~ &  i =1,\cdots,\delta n,\\
            x, &  i =\delta n + 1,\cdots, n,
             \end{cases}
\]
where $x \in \mathbb R$. Therefore, we have
\[
f_1(x) = \frac{1}{n}\sum_{i = 1}^n f_{1,i}(x) = \frac{1}{2}x^2 - \delta^{1/2} \zeta x \quad \text{and} \quad \nabla f_1(x) = \frac{1}{n}\sum_{i = 1}^n \nabla f_{1,i}(x) = x - \delta^{1/2} \zeta.
\]
It is easy for us to verify that $f_1 \in \mathcal F$. Assumption \ref{ass:basic} is obviously satisfied, and Assumption \ref{ass:i} holds from
\[
\frac{1}{n}\sum_{i = 1}^n \|\nabla f_{1,i}(x) - \nabla f_1(x)\|^2 = \delta \zeta^2 (\delta^{-1/2} - \delta^{1/2})^2 + (1 - \delta)\zeta^2 \delta = (1 - \delta)\zeta^2 \leq \zeta^2.
\]
In the second problem, the function and gradient of node \(i\) are respectively defined as
\[
f_{2,i}(x) = f_{1,i}(x) + \alpha_{\rm min} \rho^{1/2}\delta^{1/2}\zeta x, ~~~~ i =1,\cdots, n,
\]
\[
\nabla f_{2,i}(x) = \nabla f_{1,i}(x) + \alpha_{\rm min} \rho^{1/2}\delta^{1/2}\zeta, ~~~~ i =1,\cdots, n,
\]
where $0 < \alpha_{\rm min} \leq \sum_{j=1}^t \sum_{l=1}^m \alpha^{(j,l)}$ and $\{\alpha^{(j,l)}\}$ are the weights in \eqref{re:p}.
Therefore, we have
\[
f_2(x) = f_1(x) + \alpha_{\rm min} \rho^{1/2}\delta^{1/2}\zeta x \quad \text{and} \quad \nabla f_2(x) = \nabla f_1(x) + \alpha_{\rm min} \rho^{1/2}\delta^{1/2}\zeta.
\]
Again, Assumptions \ref{ass:basic} and \ref{ass:i} are both satisfied. Observe that the minimizers of $f_1$ and $f_2$ are different.

For any method $\mathsf M \in \mathcal{M}$, according to the update rule in \eqref{re:p}, we obtain $w^t_{2,i} = w^t_{1,i} + \alpha_{\rm min} \rho^{1/2}\delta^{1/2}\zeta$ and $\bar w_2^t = \bar w_1^t + \alpha_{\rm min} \rho^{1/2}\delta^{1/2}\zeta$.

Now, we construct a certain $(\delta_{\rm max},\rho)$-robust aggregator $\mathsf A \in \mathcal A$, which always outputs
\[
w^t =\bar w_1^t + \frac{\alpha_{\rm min}}{2}\rho^{1/2}\delta^{1/2}\zeta = \bar w_2^t - \frac{\alpha_{\rm min}}{2}\rho^{1/2}\delta^{1/2}\zeta,
\]
given the inputs of either $\{w^t_{1,i}\}$ or $\{w^t_{2,i}\}$. To prove that such an aggregator $\mathsf A$ is indeed $(\delta_{\rm max},\rho)$-robust, we refer to
\begin{align}
    \|w^t - \bar w_1^t\|^2 =\frac{\alpha_{\rm min}^2}{4}\rho\delta \zeta^2 \leq \alpha_{\rm min}^2\rho\delta \zeta^2(1-\delta) \leq \frac{\rho\delta}{n}\sum_{i = 1}^n \|w^t_{1,i} - \bar w_1^t\|^2,\label{eq:s-s1}\\
    \|w^t - \bar w_2^t\|^2 =\frac{\alpha_{\rm min}^2}{4}\rho\delta \zeta^2 \leq \alpha_{\rm min}^2\rho\delta \zeta^2(1-\delta) \leq \frac{\rho\delta}{n}\sum_{i = 1}^n \|w^t_{2,i} - \bar w_2^t\|^2.\label{eq:s-s2}
\end{align}
For the inequalities, we use the fact that $\delta \leq \delta_{\rm max} < 0.5$. For the last inequality in \eqref{eq:s-s1}, we recall that
\[
w_{1,i}^{t} = \sum_{j=1}^{t} \sum_{l=1}^{m} \alpha^{(j,l)} \nabla f_{1,i} ( x^{(j,l)} )
\]
and
\[
\nabla f_{1,i}(x) - \nabla f_{1}(x) = \left\{
\begin{array}{ll}
    \zeta\delta^{1/2} - \zeta\delta^{-1/2}, ~~~ & i = 1, \dots, \delta n, \\
    \zeta\delta^{1/2}, & i = \delta n+1, \dots, n.
\end{array}
\right.
\]
Then letting \(\alpha = \sum_{j=1}^{t} \sum_{l=1}^{m} \alpha^{(j,l)}\), we have
\begin{align*}
    \frac{1}{n} \sum_{i=1}^{n} \| w_{1,i}^{t} - \bar w_{1}^{t} \|^2 &= \frac{1}{n} \sum_{i=1}^{\delta n} \alpha^2 \zeta^2 \| \delta^{1/2} - \delta^{-1/2} \|^2 + \frac{1}{n} \sum_{i=\delta n+1}^{n} \alpha^2 \| \zeta\delta^{1/2} \|^2\\
    &= \alpha^2 \delta \zeta^2 \left( \delta^{1/2} - \delta^{-1/2} \right)^2 + \alpha^2 ( 1 - \delta ) \zeta^2 \delta\\
    &= \alpha^2 \zeta^2 \left[ (1- \delta)^2 + (1-\delta) \delta \right] = \alpha^2 \zeta^2 (1-\delta) \geq \alpha_{\rm min}^2 \zeta^2 (1-\delta).
\end{align*}
The same derivation holds for the last equality in \eqref{eq:s-s2}.

Since the $(\delta_{\rm max},\rho)$-robust aggregator $\mathsf A \in \mathcal A$ yields the same $w^t$ for the two problems, any method $\mathsf M \in \mathcal{M}$ shall return the same $x^{t+1}$. In consequence, any method $\mathsf M \in \mathcal{M}$ is unable to distinguish the two problems and the Byzantine error occurs on at least one of them. For any output $\tilde{x}$ that is irrelevant with the number of iterations and the number of oracle queries, we have
\begin{align*}
    \max_{j\in\{1,2\}} \| \nabla f_j ( \tilde x ) \|^2 &\geq \frac{1}{2} \| \nabla f_1 ( \tilde x ) \|^2 + \frac{1}{2} \| \nabla f_2 ( \tilde x ) \|^2\\
    &= \frac{1}{2} (\tilde x-\zeta\delta^{1/2})^2 + \frac{1}{2} (\tilde x-\zeta\delta^{1/2}+\alpha_{\rm min}\rho^{1/2}\zeta\delta^{1/2})^2\\
    &= (\tilde x - \zeta\delta^{1/2})^2 + (\tilde x - \zeta\delta^{1/2})\alpha_{\rm min}\rho^{1/2}\zeta\delta^{1/2} + \frac{\alpha_{\rm min}^2}{2}\rho\delta\zeta^2\\
    &= (\tilde x-\zeta\delta^{1/2}+\frac{\alpha_{\rm min}}{2}\rho^{1/2}\zeta\delta^{1/2})^2 + \frac{\alpha_{\rm min}^2}{4}\rho\delta\zeta^2
    \ge \frac{\alpha_{\rm min}^2}{4}\rho\delta\zeta^2.
\end{align*}
which together with \(\alpha_{\rm min} > 0 \) establishes the lower bound.
\end{proof}

\subsection{Proof of Lemma \ref{le:verrb}}\label{proof:verrb}
\begin{proof}
    We prove Lemma \ref{le:verrb} via constructing a one-dimensional stochastic problem without any Byzantine nodes, such that any method $\mathsf M \in \mathcal{M}$, equipped with a certain $(\delta_{\rm max},\rho)$-robust aggregator $\mathsf A \in \mathcal A$, will be stuck at \(x^0\) with \(\nabla f(x^0) = 2\epsilon\), as long as the number of oracle queries is insufficient. Recall that the estimated fraction of the Byzantine nodes $\delta$ satisfies $0 < \delta \leq \delta_{\rm max} < 0.5$. For simplicity we assume $x^0 = 0$, but the conclusion remains the same for arbitrary $x^0$.

    In the constructed problem, all nodes have identical functions and gradients, given by
    \begin{align*}
    f_i(x) = f(x) = \frac{L}{2}x^2 + 2\epsilon x = \E_\xi [F(x,\xi) := \frac{L}{2}x^2 + 2\xi x], ~~\forall i = 1,\cdots,n, \\
    \nabla f_i(x) = \nabla f(x) = Lx + 2\epsilon, ~~\forall i = 1,\cdots,n,
    \end{align*}
    where {$x \in \mathbb R$} and $\xi \in \Xi \subset \mathbb R$ is a random vector satisfying  $\E_\xi [\xi] = \epsilon$ and $\E_\xi [\|\xi - \epsilon\|^2] = \sigma^2/4$.
    In any method $\mathsf M \in \mathcal{M}$, each node only has access to the stochastic gradient $\nabla F(x,\xi) = Lx + 2\xi$ from a given oracle. Note that $\|\nabla f(0)\|=2\epsilon > \epsilon$ and $\nabla F(0,\xi_i) = 2\xi_i$. It is easy to verify that Assumptions \ref{ass:basic}, \ref{ass:i} and \ref{ass:u} are all satisfied.

    Given the initial point $x^0 = 0$, if there exists a $(\delta_{\rm max},\rho)$-robust aggregator $\mathsf A \in \mathcal A$ always returning 0, then the output $\tilde{x}$ of any method $\mathsf M \in \mathcal{M}$ is also 0 such that $\|\nabla f(\tilde{x})\|=2\epsilon > \epsilon$. To avoid such an undesired circumstance, the inequality \eqref{eq:agg} in Definition \ref{d:agg} should not hold, or equivalently we must have
    \begin{align}
    \label{eq:31}
    \|0 - \bar w\|^2 > \frac{\rho\delta}{n} \sum_{i=1}^n \|w_i - \bar w\|^2,
    \end{align}
    where $\bar w = \frac{1}{n}\sum_{i=1}^n\sum_{l=1}^m\alpha^l \nabla F(0,\xi^l_i)$ and $w_i = \sum_{l=1}^m\alpha^l \nabla F(0,\xi^l_i)$. Here we omit the iteration index $t$ for simplicity. Taking expectations on both sides of \eqref{eq:31} yields
    \begin{align}
    \label{eq:verrb1}
    \E\left[\left(\frac{1}{n}\sum_{i=1}^n\sum_{l=1}^m\alpha^l \nabla F(0,\xi^l_i)\right)^2\right] &> \E \left[\frac{\rho\delta}{n}\sum_{i=1}^n \left(\sum_{l=1}^m\alpha^l \nabla F(0,\xi^l_i) - \frac{1}{n}\sum_{i=1}^n\sum_{l=1}^m\alpha^l \nabla F(0,\xi^l_i)\right)^2\right] \notag \\
    &= \frac{\rho\delta(n-1)}{n} \E \left[\left(2\sum_{l=1}^m\alpha^l \xi^l_i  - 2\epsilon\sum_{l=1}^m\alpha^l\right)^2\right] \notag \\
    &= \frac{\rho\delta(n-1)\sigma^2}{n}\sum_{l=1}^m (\alpha^l)^2,
    \end{align}
    where the first equality comes from the relation between total variance and sample variance (see Chapter 2.6.1 in \citep{ross2014introduction})
    and the second is due to the independence of $\xi_i^l$.

    For the L.H.S. of \eqref{eq:verrb1}, we have
    \begin{align}
    \label{eq:verrb1l}
    &\quad ~ \mathbb E \left[\left(\frac{1}{n}\sum_{i=1}^n\sum_{l=1}^m \alpha^l \nabla F(0,\xi^l_i)\right)^2\right] \\
    &= \mathbb E \left[\left(\frac{1}{n}\sum_{i=1}^n\sum_{l=1}^m \alpha^l \nabla F(0,\xi^l_i) - \frac{1}{n}\sum_{i=1}^n\sum_{l=1}^m \alpha^l \nabla f(0) + \sum_{l=1}^m\alpha^l \nabla f(0) \right)^2\right] \notag\\
    &= \mathbb E \left[\left(\frac{1}{n}\sum_{i=1}^n\sum_{l=1}^m \alpha^l \nabla F(0,\xi^l_i) - \frac{1}{n}\sum_{i=1}^n\sum_{l=1}^m \alpha^l \nabla f(0)\right)^2\right] + \left( \sum_{l=1}^m\alpha^l \nabla f(0) \right)^2 \notag\\
    &= \frac{1}{n^2} \mathbb E \left[\sum_{i=1}^n\sum_{l=1}^m (\alpha^l)^2 \left(\nabla F(0,\xi^l_i) -  \nabla f(0)\right)^2\right] + \left( \sum_{l=1}^m\alpha^l \nabla f(0) \right)^2 \notag\\
    &\leq \frac{\sigma^2}{n} \sum_{l=1}^m (\alpha^l)^2 + 4\epsilon^2\left( \sum_{l=1}^m\alpha^l \right)^2. \notag
    \end{align}
    Substituting \eqref{eq:verrb1l} into \eqref{eq:verrb1} and reorganizing the terms, we obtain
    \[
    \frac{\left( \sum_{l=1}^m\alpha^l \right)^2}{\sum_{l=1}^m(\alpha^l)^2} > \frac{\rho\delta(n-1)\sigma^2}{4\epsilon^2 n} - \frac{\sigma^2}{4\epsilon^2 n}.
    \]
    Then from the Cauchy-Schwarz inequality, we have
    \[
    m \ge \frac{\left( \sum_{l=1}^m\alpha^l \right)^2}{\sum_{l=1}^m (\alpha^l)^2} > \frac{\rho\delta(n-1)\sigma^2}{4\epsilon^2 n} - \frac{\sigma^2}{4\epsilon^2n} = \Omega \left(\frac{\rho\delta\sigma^2}{\epsilon^2}\right),
    \]
    which implies that for any method $\mathsf M \in \mathcal{M}$, the number of oracle queries must be at least $\Omega (\frac{\rho\delta\sigma^2}{\epsilon^2})$ to obtain a $(0,\epsilon)$-stationary point. This completes the proof.
\end{proof}

\subsection{Proof of Lemma \ref{le:verrb-nc}}\label{proof:verrb-nc}
\begin{proof}
    We prove Lemma \ref{le:verrb-nc} through constructing a $d$-dimensional stochastic problem without any Byzantine nodes, where \(d = \Theta(\epsilon^{-2})\). The function has a chain-like structure, and \(\|\nabla f(x)\| > \epsilon\) if \([x]_d = 0\). Given this problem, any method $\mathsf M \in \mathcal{M}$, equipped with a certain $(\delta_{\rm max},\rho)$-robust aggregator $\mathsf A \in \mathcal A$, must query $\Omega (\frac{\rho\delta\sigma^2}{\epsilon^2})$ times in expectation to identify the next coordinate. Thus, given $x^0 = \mathbf{0}$, the overall oracle query complexity is $\Omega (\frac{\rho\delta\sigma^2}{\epsilon^4})$. Recall that the estimated fraction of the Byzantine nodes $\delta$ satisfies $0 < \delta \leq \delta_{\rm max} < 0.5$. For simp- licity we assume $x^0 = \mathbf{0}$, but the conclusion remains the same for arbitrary $x^0$.

    In the constructed problem, all nodes have identical functions, given by
    \[
    f(x) = f_i(x) = \frac{L\nu^2}{152}h\left(\frac{x}{\nu}\right)~~{\rm with}~~\nu = \frac{152}{L}\cdot 2\epsilon,  ~~\forall i = 1,\cdots,n,
    \]
    Therein, $h : \mathbb R^d \to \mathbb R$ is defined as
    \[
    h(x) := -\Psi(1)\Phi([x]_1)+\sum_{j=2}^d \left[\Psi(-[x]_{j-1})\Phi(-[x]_j)-\Psi([x]_{j-1})\Phi([x]_j)\right]~~{\rm with}~~d=\left\lfloor \frac{L\Delta}{7296\epsilon^2} \right\rfloor,
    \]
    while $\Psi: \mathbb R \to \mathbb R$ and $\Phi: \mathbb R \to \mathbb R$ are defined as
    \[
    \Psi(a) = \left\{
        \begin{aligned}
            &0, &a\leq 1/2,\\
            &\exp{\left(1-\frac{1}{(2a-1)^2}\right)},&a>1/2,
        \end{aligned}
        \right.
    ~~{\rm and}~~
    \Phi(a) = \sqrt{e}\int_{-\infty}^a e^{-\frac{\tau^2}{2}}d\tau.
    \]
    According to \citep{arjevani2023lower}, $f_i$ is $L$-smooth and satisfies Assumption \ref{ass:basic}. In addition, Assumption \ref{ass:i} obviously holds.

    Below, we construct a stochastic gradient oracle $\nabla F(x,\xi)$ that satisfies Assumption \ref{ass:u} for $f(x)$, in the form of
    \begin{align}
    \label{eq:nablaF}
    \nabla F(x,\xi) = \frac{L\nu}{152}\cdot \nabla H\left(\frac{x}{\nu},\xi\right) = 2\epsilon \cdot \nabla H\left(\frac{x}{\nu},\xi\right),
    \end{align}
    where
    \begin{align}
    \label{eq:nablaH}
    \nabla_j H\left(\frac{x}{\nu},\xi\right) := \nabla_j h(\frac{x}{\nu}) \cdot \left(1+\mathds{1}\{j > {\rm prog}_{\frac{1}{2}}(\frac{x}{\nu})\}\left(\frac{\xi}{p}-1\right)\right).
    \end{align}
    Here, $\mathds{1}$ is the indicator function that returns 1 if the argument holds true and 0 otherwise; the random variable {$\xi \sim$ Bernoulli$(p)$}; ${\rm prog}_{\frac{1}{2}}(\frac{x}{\nu}) := \max\{j = 0, 1, \cdots, d \mid |\frac{[x]_j}{\nu}|>\frac{1}{2}\}$ denotes the largest index of $ |\frac{x}{\nu}| $ whose element is larger than $1/2$ -- we additionally define $|\frac{[x]_0}{\nu}| = 1$ so as to return index $0$ if no other elements are qualified. Note that the constant $1/2$ can be replaced by any other constant larger than 0. According to Lemma 3 in \citep{arjevani2023lower}, we can show that such a stochastic gradient oracle satisfies Assumption \ref{ass:u}, as
    \[
    \mathbb E \left[\nabla F\left(x,\xi\right) \right] =  \nabla f(x) \quad \text{and} \quad \mathbb E \left[\|\nabla F\left(x,\xi\right) - \nabla f(x)\|^2\right] \leq \sigma^2,
    \]
    when $\frac{1}{p} = \frac{\sigma^2}{2116\epsilon^2}+1$.

    Recall that $x$ is initialized as \( \mathbf{0}\). Now we construct a \((\delta_{\rm max}, \rho)\)-robust aggregator \(\mathsf{A} \in \mathcal{A}\) with which the coordinates of $x$ are updated sequentially, from $1$ to $d$. When $x = \mathbf{0}$, the stochastic gradient oracle defined in \eqref{eq:nablaF} returns 0 except for the first coordinate. Thus, for any method \(\mathsf M \in \mathcal{M}\), the elements of all $\{w_i\}$ are 0 except for the first coordinate according to \eqref{re:p}. Therefore, if there exists a $(\delta_{\rm max},\rho)$-robust aggregator $\mathsf A \in \mathcal A$ returning 0 for the first coordinate but the true averages for the other coordinates, the aggregated $w$ remain \( \mathbf{0}\). In consequence, the variable $x$ and the output $\tilde{x}$ are also \( \mathbf{0}\). This is an undesired circumstance since \(\|\nabla f(x)\| > \epsilon\) if \([x]_d = 0\). To avoid this circumstance and discover the first coordinate, we must have
    \[
    \|0 - [\bar w]_1\|^2 > \frac{\rho\delta}{n}\sum_{i = 1}^n \|[w_i]_1 - [\bar w]_1\|^2.
    \]
    Likewise, given the currently undiscovered coordinate $j \in [1, \cdots, d]$, such that $[x]_{j'} = 0$ for $j' \geq j$, the stochastic gradient oracle defined in \eqref{eq:nablaF} returns 0 for the coordinates $j' > j$. Thus, for any method \(\mathsf M \in \mathcal{M}\), the elements of all $\{w_i\}$ are 0 for the coordinates $j' > j$ according to \eqref{re:p}. Therefore, if there exists a $(\delta_{\rm max},\rho)$-robust aggregator $\mathsf A \in \mathcal A$ returning 0 for the $j$-th coordinate but the true averages for the coordinates $j' \neq j$, the aggregated $[w]_j$ remain \(0\) for the coordinates $j' \geq j$. In consequence, $[x]_{j'}$ and $[\tilde{x}]_{j'}$ are also \(0\) for the coordinates $j' \geq j$. As discussed before, this is an undesired circumstance since \(\|\nabla f(x)\| > \epsilon\) if \([x]_d = 0\). To avoid this circumstance and discover the $j$-th coordinate, we must have
    \begin{align}
    \label{eq:aa37}
    \|0 - [\bar w]_j\|^2 > \frac{\rho\delta}{n}\sum_{i = 1}^n \|[w_i]_j - [\bar w]_j\|^2, \quad j = 1, \cdots, d.
    \end{align}
    Otherwise, the method \(\mathsf M \in \mathcal{M}\) shall get stuck in the $j$-th coordinate. Readers are referred to Figure \ref{fig:flowchart} for the evolution of the iterates.
    \begin{figure}
        \centering
        \includegraphics[width=1\linewidth]{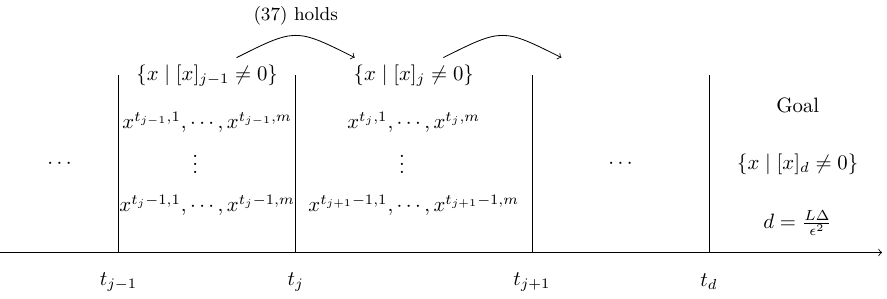}
        \caption{Evolution of the iterates.}
        \label{fig:flowchart}
    \end{figure}

    By definition \([\bar w]_j = \frac{1}{n}\sum_{i=1}^n\sum_{l=1}^m \alpha^l  \nabla_j F(x^l,\xi^l_i)\) and \([w_i]_j = \sum_{l=1}^m \alpha^l \nabla_j F(x^l,\xi^l_i)\) in \eqref{eq:aa37}. Here and thereafter, we omit the iteration index \(t\) for simplicity. Observe that \(\nabla_j F(x^l,\xi^l_i) \neq 0\) only if \([x^l]_{j-1} \neq 0\). We denote the number of the sampled \(\{x^l\}_{l=1}^m\) satisfying \([x^l]_{j-1} \neq 0\) as \(m_j\) and suppose that they are \(\{x^l\}_{l=m-m_j}^m\). Thus, \([\bar w]_j = \frac{1}{n}\sum_{i=1}^n\sum_{l=m-m_j}^m \alpha^l \nabla_j F(x^l,\xi^l_i)\) and \([w_i]_j = \sum_{l=m-m_j}^m \alpha^l \nabla_j F(x^l,\xi^l_i)\). Within the summands, $\alpha_l \geq 0$ for $l = m-m_j, \cdots, m$ but at least one of them is positive. Otherwise, \eqref{eq:aa37} does not hold. Substituting $[\bar w]_j$ and $[w_i]_j$ into \eqref{eq:aa37} and taking expectations, we have
    \begin{align}
    \label{eq:aa38}
    &\mathbb E \left[\left(\frac{1}{n}\sum_{i=1}^n\sum_{l=m-m_j}^m \alpha^l \nabla_j F(x^l,\xi^l_i)\right)^2\right] \\
    > ~&\mathbb E \left[\frac{\rho \delta}{n} \sum_{i=1}^n \left(\sum_{l=m-m_j}^m \alpha^l \nabla_j F(x^l,\xi^l_i) - \frac{1}{n}\sum_{i=1}^n\sum_{l=m-m_j}^m \alpha^l \nabla_j F(x^l,\xi^l_i) \right)^2 \right]. \notag
    \end{align}
    For the L.H.S. of \eqref{eq:aa38}, we have
    \begin{align*}
    &\quad ~\mathbb E \left[\left(\frac{1}{n}\sum_{i=1}^n\sum_{l=m-m_j}^m \alpha^l \nabla_j F(x^l,\xi^l_i)\right)^2\right] \\
    &= \mathbb E \left[\left(\frac{1}{n}\sum_{i=1}^n\sum_{l=m-m_j}^m \alpha^l \nabla_j F(x^l,\xi^l_i) - \frac{1}{n}\sum_{i=1}^n\sum_{l=m-m_j}^m \alpha^l \nabla_j f(x^l) + \sum_{l=m-m_j}^m\alpha^l \nabla_j f(x^l) \right)^2\right] \\
    &= \mathbb E \left[\left(\frac{1}{n}\sum_{i=1}^n\sum_{l=m-m_j}^m \alpha^l \nabla_j F(x^l,\xi^l_i) - \frac{1}{n}\sum_{i=1}^n\sum_{l=m-m_j}^m \alpha^l \nabla_j f(x^l)\right)^2\right] + \left( \sum_{l=m-m_j}^m\alpha^l \nabla_j f(x^l) \right)^2 \\
    &= \frac{1}{n^2} \mathbb E \left[\sum_{i=1}^n\sum_{l=m-m_j}^m (\alpha^l)^2 \left(\nabla_j F(x^l,\xi^l_i) -  \nabla_j f(x^l)\right)^2\right] + \left( \sum_{l=m-m_j}^m\alpha^l \nabla_j f(x^l) \right)^2 \\
    &= \frac{\sum_{l=m-m_j}^m \left( \alpha^l \nabla_j f(x^l) \right)^2}{n}\frac{1-p}{p} + \left( \sum_{l=m-m_j}^m\alpha^l \nabla_j f(x^l) \right)^2,
    \end{align*}
    where the second equality comes from $\E [\nabla F(x,\xi)] = \nabla f(x)$, the third equality is due to the independence of $\xi_i^l$ and the last equality follows from $$
    \E \left[\left(\nabla_j F(x^l,\xi^l_i) -  \nabla_j f(x^l)\right)^2\right] = (\nabla_j f(x^l))^2\frac{\E[(\xi^l_i-p)^2]}{p^2} = (\nabla_j f(x^l))^2\frac{1-p}{p}.
    $$
    For the R.H.S. of \eqref{eq:aa38}, we have
    \begin{align*}
    &\mathbb E \left[\frac{\rho \delta}{n} \sum_{i=1}^n \left(\sum_{l=m-m_j}^m \alpha^l \nabla_j F(x^l,\xi^l_i) - \frac{1}{n}\sum_{i=1}^n\sum_{l=m-m_j}^m \alpha^l \nabla_j F(x^l,\xi^l_i) \right)^2 \right] \\
    = ~ &\frac{\rho \delta (n-1)}{n} \mathbb E \left[ \left(\sum_{l=m-m_j}^m \alpha^l \nabla_j F(x^l,\xi^l_i) - \sum_{l=m-m_j}^m \alpha^l \nabla_j f(x^l) \right)^2 \right] \\
    = ~ &\frac{\rho \delta (n-1)}{n} \mathbb E \left[ \sum_{l=m-m_j}^m (\alpha^l)^2 \left(\nabla_j F(x^l,\xi^l_i) - \nabla_j f(x^l) \right)^2 \right] \\
    = ~ &\frac{\rho\delta (n-1)\sum_{l=m-m_j}^m \left( \alpha^l \nabla_j f(x^l) \right)^2}{n}\frac{1-p}{p},
    \end{align*}
    where the first equality uses the relation between total variance and sample variance and the second is also due to the independence of $\xi_i^l$.

    Therefore, \eqref{eq:aa38} becomes
    \[
    \frac{\left( \sum_{l=m-m_j}^m\alpha^l \nabla_j f(x^l) \right)^2}{\sum_{l=m-m_j}^m \left( \alpha^l \nabla_j f(x^l) \right)^2} > \frac{\rho\delta (n-1)}{n}\frac{1-p}{p} - \frac{1}{n}\frac{1-p}{p}.
    \]
    Since $\frac{1-p}{p} = \frac{\sigma^2}{2116\epsilon^2}$, we have
    \[
    m_j > \frac{\rho\delta\sigma^2}{2116\epsilon^2}-\frac{\rho\delta\sigma^2}{2116\epsilon^2n} - \frac{\sigma^2}{2116\epsilon^2n} = \Omega (\frac{\rho\delta\sigma^2}{\epsilon^2}),~\forall j = 1, \cdots, d.
    \]
    Recall that we need to sequentially update to the $d$-th coordinate with $d=\Omega(L\Delta\epsilon^{-2})$, the overall oracle query complexity is $\sum_{j=1}^d m_j=\Omega(\frac{L\Delta\rho\delta\sigma^2}{\epsilon^4})$.
\end{proof}

\subsection{Proof of Lemma~\ref{prop:gc}}
\label{proof:proposition}
\begin{proof}
We begin with introducing several auxiliary sequences. The first sequence is
\begin{align}
   \bar x^t := x^t + (\sqrt{q} - 1)(x^t - x^{t-1}),~~\forall t = 1, 2, \cdots, \label{eq:defs}
\end{align}
with $\bar x^0=x^0$. According to the definition of $\beta$ in \eqref{beta} and the definition of $\bar x^t$ in \eqref{eq:defs}, we further introduce
\begin{align*}
 y^t := & x^t + \beta (x^t - x^{t-1}) \\
  = & x^t + \frac{\sqrt{q}-1}{\sqrt{q}+1} (x^t - x^{t-1})  \\
  = & x^t + \frac{\sqrt{q}-1}{q-1}\cdot (\sqrt{q}-1) (x^t - x^{t-1})  \\
  = & \frac{\sqrt{q}-1}{q-1} \bar x^t + (1-\frac{\sqrt{q}-1}{q-1})x^t, ~~\forall t = 1, 2, \cdots.
\end{align*}
We also introduce $$z^{t-1} := \frac{x^*}{\sqrt{q}} + (1-\frac{1}{\sqrt{q}}) x^{t-1}, ~~\forall t = 1, 2, \cdots.$$

From the $\frac{L}{\theta}$-strong convexity of $h^t$, we have
\begin{align}
\label{eq:auxhk}
    h^t(z^{t-1})  \geq & h^{t*} + \frac{L}{2\theta}\|z^{t-1} -x^{t*}\|^2 \\
            = & h^{t*} + \frac{L}{2\theta}\|z^{t-1}-x^t\|^2 + \underbrace{\frac{L}{2\theta}\|x^t-x^{t*}\|^2 +\frac{L}{\theta}\langle z^{t-1}-x^t, x^t-x^{t*}\rangle}_{- \Pi^t}. \notag
\end{align}
Therefore, we have
\begin{align}
\label{eq:sufficient_descent}
     \E[f(x^t)] & \leq \E[h^{t*}] + \upsilon^t \\
                & \leq \E[ h^t(z^{t-1})] -\E\left[ \frac{L}{2\theta}\|z^{t-1}-x^t\|^2 \right] + \E[\Pi^t] + \upsilon^t \notag \\
                & \leq \E[f(z^{t-1})] + \E\left[ \frac{L-\mu\theta}{2\theta}\|z^{t-1}-y^{t-1}\|^2 \right] -\E\left[ \frac{L}{2\theta}\|z^{t-1}-x^t\|^2 \right] + \E[\Pi^t] + \upsilon^t, \notag
\end{align}
where the first and last inequalities are respectively due to \ref{ta2} and \ref{ta1} in the hypothesis, while the second inequality comes from taking expectation on \eqref{eq:auxhk}. Note that the expectation is conditioned on all random variables prior to $t-1$.

Now we bound the first three terms at the R.H.S. of \eqref{eq:sufficient_descent}. We have
\begin{align}
\label{eq:relations}
  \E[f(z^{t-1})] & \leq \frac{1}{\sqrt{q}}f^* + (1-\frac{1}{\sqrt{q}}) \E[f(x^{t-1})] - \frac{\mu (\sqrt{q}-1)}{2q}R^{t-1}, \\
  \E[\|z^{t-1} - y^{t-1}\|^2] & \leq \frac{1}{\sqrt{q}}\cdot \frac{R^{t-1}}{q+\sqrt{q}} + \frac{1}{\sqrt{q}}\cdot \frac{\bar R^{t-1}}{\sqrt{q}+1}, \notag \\
  \E[\|z^{t-1} - x^t\|^2] & = \frac{\bar R^t}{q}, \notag
\end{align}
where $R^t = \E[\|x^*- x^t\|^2]$ and $ \bar R^t = \E[\|x^*- \bar x^t\|^2]$.
The first inequality in \eqref{eq:relations} is due to the {strong} convexity of $f$, the last one can be
obtained from the definition of~$\bar x^t$ in~(\ref{eq:defs}) after simple derivation, and the second one can be derived through
\begin{equation*}
   \begin{split}
      \Vert z^{t-1} - y^{t-1} \Vert^2 & = \left\| \frac{x^*-x^{t-1}}{q+\sqrt{q}}+\frac{x^*-\bar x^{t-1}}{\sqrt{q}+1} \right\|^2 \\
      & = \frac{1}{q} \left\| \left(1-\frac{\sqrt{q}}{\sqrt{q}+1}\right)(x^*-x^{t-1})+\frac{\sqrt{q}}{\sqrt{q}+1} (x^*-\bar x^{t-1}) \right\|^2 \\
      & \leq \frac{1}{q} \cdot \left(1-\frac{\sqrt{q}}{\sqrt{q}+1}\right)\Vert x^*-x^{t-1} \Vert^2 +\frac{1}{q} \cdot \frac{\sqrt{q}}{\sqrt{q}+1} \Vert x^*-\bar x^{t-1} \Vert^2\\
      & =  \frac{1}{\sqrt{q}} \cdot \frac{1}{q+\sqrt{q}}  \Vert x^*-x^{t-1} \Vert^2 + \frac{1}{\sqrt{q}} \cdot \frac{1}{\sqrt{q}+1} \Vert x^*-\bar x^{t-1} \Vert^2.
   \end{split}
\end{equation*}

Substituting \eqref{eq:relations} into \eqref{eq:sufficient_descent} and using the fact that $q=\frac{L}{\theta \mu}$, we have
\begin{align}
\label{eq:37}
\E[f(x^t)-f^*] + \frac{\mu}{2}\bar R^t   \leq & (1-\frac{1}{\sqrt{q}}) \E[f(x^{t-1})-f^*] - \frac{\mu (\sqrt{q}-1)}{2q}R^{t-1} \\
& + \frac{L-\mu\theta}{2\theta}
   \frac{1}{\sqrt{q}}\cdot \frac{R^{t-1}}{q+\sqrt{q}} + \frac{L-\mu\theta}{2\theta}\frac{1}{\sqrt{q}}\cdot \frac{\bar R^{t-1}}{\sqrt{q}+1}
      + \E[\Pi^t] +  {\upsilon^t} \notag \\
     = & (1-\frac{1}{\sqrt{q}}) \E[f(x^{t-1})-f^*] + \frac{L-\mu\theta}{2\theta}\frac{1}{\sqrt{q}}\cdot \frac{\bar R^{t-1}}{\sqrt{q}+1}  + \E[\Pi^t] +  {\upsilon^t}. \notag
\end{align}
Further defining a Lyapunov function
\begin{equation}
   S^t = (1-\frac{1}{\sqrt{q}}) \E[f(x^t)-f^*] + \frac{L-\mu\theta}{2\theta}\frac{1}{\sqrt{q}}\cdot \frac{\bar R^{t}}{\sqrt{q}+1}, \quad  {t = 1, 2, \cdots,} \label{eq:Sk}
\end{equation}
from \eqref{eq:37} we have
 \begin{equation}
   (1-\frac{1}{\sqrt{q}})^{-1}S^t \leq S^{t-1}
        + \E[\Pi^t] + {\upsilon^t}.\label{eq:lyapunov}
\end{equation}

For the term $\E[\Pi^t]$ in \eqref{eq:lyapunov}, we know that
\begin{align}
\label{eq:pi}
   \Pi^t & = - \frac{L}{2\theta}\|x^t-x^{t*}\|^2  - \frac{L}{\theta}\langle z^{t-1} - x^t, x^t - x^{t*} \rangle \\
                 & = - \frac{L}{2\theta}\|x^t-x^{t*}\|^2  -  \frac{L}{\theta\sqrt{q}} \langle x^* - \bar x^t, x^t - x^{t*} \rangle \notag \\
                 & \leq - \frac{L}{2\theta}\|x^t-x^{t*}\|^2  +  \frac{L}{\theta\sqrt{q}} \|x^*-\bar x^t\| \| x^t - x^{t*}\| \notag \\
       & \leq   \left(2\sqrt{q}-1\right)\frac{L}{2\theta}\|x^t-x^{t*}\|^2  + \frac{ L}{4\theta q\sqrt{q}}\|x^*-\bar x^t\|^2   ~~\text{(using Young's inequality)} \notag \\
       & \leq   \left(2\sqrt{q}-1\right)(h^t(x^t)-h^{t*})  + \frac{ L}{4\theta q\sqrt{q}}\|x^*-\bar x^t\|^2   ~~\text{(using $2\sqrt{q} \geq 1$ and strong convexity)} \notag \\
       & =   \left(2\sqrt{q}-1\right)(h^t(x^t)-h^{t*})  + \frac{L-\mu\theta}{4\theta\sqrt{q}(q-1)}\|x^*-\bar x^t\|^2. \notag
   \end{align}
Taking expectation on \eqref{eq:pi}, noticing that the quadratic term involving $\|x^*-\bar x^t\|^2$ is smaller than $S^t/2(\sqrt{q}-1)$ in expectation (from the definition of $S^t$ in~(\ref{eq:Sk})) and using (iii) in the hypothesis, we obtain
\begin{align}
\label{eq:42}
\E[\Pi^t] \leq \left(2\sqrt{q}-1\right)\varepsilon^t + \frac{S^t}{2(\sqrt{q}-1)}.
\end{align}

Substituting \eqref{eq:42} into \eqref{eq:lyapunov}, we have
\begin{align*}
S^t \leq \frac{2\sqrt{q}-2}{2\sqrt{q}-1}\left(S^{t-1}   + \upsilon^t + \left(2\sqrt{q}-1\right)\varepsilon^t\right).
\end{align*}
Unrolling the recursion yields
\begin{align}
\label{eq:43}
S^t & \leq \left(\frac{2\sqrt{q}-2}{2\sqrt{q}-1}\right)^t \left
  (S^0 + \sum_{\tau=1}^t \left(\frac{2\sqrt{q}-1}{2\sqrt{q}-2}\right)^{\tau-1}\left( \upsilon^\tau - \varepsilon^\tau + 2\sqrt{q}\varepsilon^\tau \right)
  \right)\\
    & = \left(1-\frac{1}{2\sqrt{q}-1}\right)^t S^0 + \sum_{\tau=1}^t \left(1-\frac{1}{2\sqrt{q}-1}\right)^{t-\tau+1}\left( \upsilon^\tau - \varepsilon^\tau +  2\sqrt{q}\varepsilon^\tau\right) \notag \\
    & \leq (1-\frac{1}{\sqrt{q}}) \left(2\left(1-\frac{1}{2\sqrt{q}}\right)^t(f(x^0)-f^*) + 2\sum_{\tau=1}^t \left(1-\frac{1}{2\sqrt{q}}\right)^{t-\tau}\left( \upsilon^\tau - \varepsilon^\tau +  2\sqrt{q}\varepsilon^\tau\right)  \right) \notag \\
    & \leq (1-\frac{1}{\sqrt{q}}) \left(2\left(1-\frac{1}{2\sqrt{q}}\right)^t(f(x^0)-f^*) + 4\sum_{\tau=1}^t \left(1-\frac{1}{2\sqrt{q}}\right)^{t-\tau}\left( \upsilon^\tau +  \sqrt{q}\varepsilon^\tau\right)  \right), \notag
\end{align}
where the second inequality uses the fact that $1-\frac{1}{2\sqrt{q}-1} \leq 2(1-\frac{1}{\sqrt{q}})$ and
\begin{equation*}
\begin{split}
    S^0 & = (1-\frac{1}{\sqrt{q}})(f(x^0)-f^*) + \frac{L-\mu\theta}{2\theta\sqrt{q}(\sqrt{q}+1)}\|x^0-x^*\|^2 \\
        & = (1-\frac{1}{\sqrt{q}})(f(x^0)-f^*) + \frac{\mu}{2}(1-\frac{1}{\sqrt{q}})\|x^0-x^*\|^2 \\
        & \leq 2(1-\frac{1}{\sqrt{q}})(f(x^0)-f^*).
\end{split}
\end{equation*}

From the definition of $S^t$ in~(\ref{eq:Sk}), we have $(1-\frac{1}{\sqrt{q}})\E[f(x^t)-f^*] \leq S^t$. Combining this inequality with \eqref{eq:43}, we obtain \eqref{convergence-guarantee} and complete the proof of Lemma~\ref{prop:gc}.
\end{proof}

\subsection{Bound of aggregation bias \(\E[\|\Delta_1^t\|^2]\)} \label{proof:v}
\begin{lemma} \label{le:v} Consider Algorithm \ref{subalgo} with the step size $\eta=\frac{1}{L}$. If Assumptions \ref{ass:i} and \ref{ass:u} hold, while $\alpha \in [0,1]$, $\beta \in [0,1)$ and $\theta \in (0,1]$ satisfy
\begin{align}\label{alphabeta}
    &\chi_0\beta^2 = (1-\alpha)\beta(\theta+\beta), \\
    &\chi_1 \geq \frac{1}{1-\chi_0\beta^2}\frac{\alpha\theta^2}{1-\beta^2}\geq 0, \notag\\
    &\chi_2 \geq \frac{(1-\alpha)\theta(\theta+\beta)}{1-\chi_0\beta^2} \geq 0, \notag\\
    &\chi_3 \geq \frac{1}{1-\chi_0\beta^2}\frac{\alpha\theta(\theta+\beta)}{1-\beta(\theta+\beta)}\geq 0, \notag\\
    &\chi_4 \geq 2\alpha\beta^2(\chi_1+\chi_2+\frac{\alpha\beta^2m}{(1-\beta^2)m_0}) + \frac{2\alpha\theta^2}{1-\beta^2} + (1-\alpha)\theta^2, \notag\\
    &\chi_5 \geq 2\alpha\beta^2(\chi_2+\chi_3+1) + \frac{2\alpha\theta(\theta+\beta)}{1-\beta(\theta+\beta)} + (1-\alpha)\theta^2, \notag
\end{align}
with some $\chi_1$, $\chi_2$, $\chi_3$, $\chi_4$, $\chi_5 \geq 0$, then for any robust aggregator $\mathsf A \in \mathcal A$, we have
    \begin{align}\label{re:s-bs}
    \E[ \|\Delta_1^t\|^2] &\leq \frac{1}{L^2}\left(\frac{3\chi_4\rho\delta\sigma^2}{m}\left(1+\frac{1}{(1-\delta)n}\right)+3\chi_5\rho\delta\zeta^2\right).
    \end{align}
\end{lemma}

\begin{proof}
    Using Young's inequality and {$\eta=\frac{1}{L}$}, we have
    \begin{align}
    \label{delta1}
       & \|\Delta_1^t\|^2 = \frac{1}{L^2}\left\|\beta (\hat s^{t-1}-\bar s^{t-1}) +  (\bar s^t - \hat s^t)\right\|^2\\
       & = \frac{1}{L^2} \left\|\alpha\beta \hat s^{t-1}+\frac{\theta}{|\mathcal H|}\sum_{i \in \mathcal H}  g_i^{t-1}-\alpha {\mathsf A}(\{s_i^t\}_{i=1}^n)-(1-\alpha)\theta{\mathsf A}(\{g_i^{t-1}\}_{i=1}^n)\right\|^2 \notag \\
       & \leq \frac{\alpha}{L^2} \left\|\beta \hat s^{t-1}+\frac{\theta}{|\mathcal H|}\sum_{i \in \mathcal H}  g_i^{t-1}- {\mathsf A}(\{s_i^t\}_{i=1}^n)\right\|^2 +\frac{(1-\alpha)\theta^2}{L^2}\left\|{\mathsf A}(\{g_i^{t-1}\}_{i=1}^n) - \frac{1}{|\mathcal H|}\sum_{i \in \mathcal H}  g_i^{t-1}\right\|^2 \notag \\
       & \leq \frac{2\alpha\beta ^2}{L^2} \left\|\hat s^{t-1} - \bar s^{t-1}\right\|^2+\frac{2\alpha}{L^2}\left\|\bar s^t- {\mathsf A}(\{s_i^t\}_{i=1}^n)\right\|^2 + \frac{(1-\alpha)\theta^2}{L^2}\left\|{\mathsf A}(\{g_i^{t-1}\}_{i=1}^n) - \frac{1}{|\mathcal H|}\sum_{i \in \mathcal H}  g_i^{t-1} \right\|^2. \notag
    \end{align}
    Taking expectation on the first term, we have
    \begin{align}
    \label{s-bars}
        \E [\left\|\hat s^{t} - \bar s^{t}\right\|^2] &= \E [\|- \alpha(\bar s^t - {\mathsf A}(\{s_i^t\}_{i=1}^n))+(1-\alpha)( s^t - \bar s^t)\|^2]\\
        & \leq \E [\alpha\| \bar s^t - {\mathsf A}(\{s_i^t\}_{i=1}^n)\|^2+(1-\alpha)\| s^t - \bar s^t\|^2] \notag \\
        & \leq \E [\alpha\| \bar s^t - {\mathsf A}(\{s_i^t\}_{i=1}^n)\|^2+(1-\alpha)\beta(\theta+\beta) \|\hat s^{t-1} - \bar s^{t-1}\|^2 \notag \\
        &\quad +(1-\alpha)\theta(\theta+\beta)\|{\mathsf A}(\{g_i^{t-1}\}_{i=1}^n) - \frac{1}{|\mathcal H|}\sum_{i \in \mathcal H}  g_i^{t-1}\|^2] \notag.
    \end{align}

    We first consider the second term at the R.H.S. of \eqref{delta1}, which is also the first term at the R.H.S. of \eqref{s-bars}. From Definition \ref{d:agg} of the $(\delta_{\rm max},\rho)$-robust aggregator $\mathsf A$, we have
    \begin{align}
    \E [\left\| \bar s^{t} - {\mathsf A}(\{s_i^t\}_{i=1}^n)\right\|^2] \leq \frac{\rho \delta}{|\mathcal H|} \sum_{i \in \mathcal H} \E[\left\|\bar s^t - s_i^t\right\|^2]. \label{tildes-bars}
    \end{align}
    We proceed to bound $\frac{1}{|\mathcal H|}\sum_{i \in \mathcal H}\E[\|\bar s^t - s^t_i\|^2]$. Letting $\E_{\xi^{[t]}}:= \E_{\xi^t,\xi^{t-1},\cdots,\xi^0}$, from the definition of $s_i^t$ in \eqref{update:s}, we obtain
    \begin{align*}
        \E_{\xi^{[t-1]}} [\|s^t_i - \E_{\xi^{[t-1]}}[s_i^t]\|^2] &= \E_{\xi^{[t-2]}}\E_{\xi^{t-1}} [\|\beta  (s^{t-1}_i - \E_{\xi^{[t-2]}}[s^{t-1}_i])  + \theta( g_i^{t-1}-\nabla f_i(y^{t-1})\|^2] \\
        &\leq \E_{\xi^{[t-2]}} [\|\beta  (s^{t-1}_i - \E_{\xi^{[t-2]}}[s^{t-1}_i])\|^2] + \frac{\theta^2}{m}\sigma^2,
    \end{align*}
    where the cross term $\E_{\xi^{t-1}}[\beta\theta\langle s^{t-1}_i - \E_{\xi^{[t-2]}}[s^{t-1}_i],  g_i^{t-1}-\nabla f_i(y^{t-1}\rangle] = 0$ due to the unbiasedness of $ g_i^{t-1}$. Unrolling the above recursion yields
    \[
    \E [\|s^t_i - \E[s_i^t]\|^2] \leq \frac{\theta^2}{(1-\beta^2)m} \sigma^2 + \beta^{2t} \E[\|s_i^0-\E[s_i^0]\|^2]\leq \frac{\theta^2}{(1-\beta^2)m} \sigma^2 + \frac{\beta^{2t}}{m_0}\sigma^2,
    \]
    where the last inequality is due to \(s_i^0 = \frac{1}{m_0}\sum_{l=1}^{m_0} \nabla F(y^0,\xi_i^{(0,l)}).\) Hence, for $\bar s^t = \frac{1}{|\mathcal H|} \sum_{i \in \mathcal H} s_i^t$, we have
    \begin{align*}
        \E[\|\bar s^t - \E[\bar s^t]\|^2] &\le \frac{\theta^2}{(1-\beta^2)(1-\delta)nm} \sigma^2 + \frac{\beta^{2t}}{(1-\delta)n} \E[\|s_i^0-\E[s_i^0]\|^2]\\
        &\le \frac{\theta^2}{(1-\beta^2)(1-\delta)nm} \sigma^2 + \frac{\beta^{2t}}{(1-\delta)nm_0} \sigma^2.
    \end{align*}
    Then for $\E [\|\E[\bar s^t] - \E[s_i^t]\|^2]$, it holds that
    \begin{align*}
        \frac{1}{|\mathcal H|}\sum_{i \in \mathcal H}[\left\|\E[\bar s^t] - \E[s_i^t]\right\|^2] &= \frac{1}{|\mathcal H|}\sum_{i \in \mathcal H}[\|\beta(\E[\bar s^{t-1}]-\E[s_i^{t-1}])+\theta(\nabla f(y^{t-1}) - \nabla f_i(y^{t-1}))\|^2] \\
        &\leq \frac{\beta(\theta+\beta)}{|\mathcal H|}\sum_{i \in \mathcal H}[\|(\E[\bar s^{t-1}]-\E[s_i^{t-1}])\|^2] + \theta(\theta+\beta)\zeta^2.
    \end{align*}
    Unrolling the above recursion yields
    \begin{align*}
    \frac{1}{|\mathcal H|}\sum_{i \in \mathcal H}[\left\|\E[\bar s^t] - \E[s_i^t]\right\|^2] &\leq \frac{\theta(\theta+\beta)}{1-\beta(\theta+\beta)}\zeta^2+ \frac{\beta^t(\theta+\beta)^t}{|\mathcal H|}\sum_{i \in \mathcal H}[\|(\E[\bar s^0]-\E[s_i^0])\|^2]\\
    &\leq \frac{\theta(\theta+\beta)}{1-\beta(\theta+\beta)}\zeta^2+ \beta^t(\theta+\beta)^t\zeta^2.
    \end{align*}
    Therefore, we have
    \begin{align} \label{si-bars}
        \frac{1}{|\mathcal H|}\sum_{i \in \mathcal H}\E[\|\bar s^t - s_i^t\|^2] &\le \frac{3}{|\mathcal H|}\sum_{i \in \mathcal H} \E[\|s_i^t - \E[s_i^t]\|^2] + \frac{3}{|\mathcal H|}\sum_{i \in \mathcal H} \E[\|\E[\bar s^t] - \E[s_i^t]\|^2] + 3 \E[\|\bar s^t - \E[\bar s^t]\|^2] \notag \\
        &\le \frac{3}{m}\left(1+\frac{1}{(1-\delta)n}\right)\frac{\theta^2}{1-\beta^2}\sigma^2+3\frac{\theta(\theta+\beta)}{1-\beta(\theta+\beta)}\zeta^2 \notag\\
        &\quad ~+ \frac{3\beta^{2t}}{(1-\delta)nm_0} \sigma^2 + 3\beta^t(\theta+\beta)^t\zeta^2.
    \end{align}

    Next, consider the third term at the R.H.S. of \eqref{delta1}, which is also the third term at the R.H.S. of \eqref{s-bars}. Using Definition \ref{d:agg}, we have
    \begin{align}
    \label{eq:aa55}
    &~\E[\|{\mathsf A}(\{ g_i^{t-1}\}) - \frac{1}{|\mathcal H|}\sum_{i \in \mathcal H}  g_i^{t-1}\|^2] \leq \frac{\rho \delta}{|\mathcal H|}\sum_{i \in \mathcal H} \E[\| g_i^{t-1} - \frac{1}{|\mathcal H|}\sum_{i \in \mathcal H}  g_i^{t-1}\|^2]\\
    \leq &~\frac{\rho \delta}{|\mathcal H|}\sum_{i \in \mathcal H} \E[\| g_i^{t-1} - \nabla f_i(y^{t-1}) + \nabla f_i(y^{t-1}) - \nabla f(y^{t-1}) + \nabla f(y^{t-1}) - \frac{1}{|\mathcal H|}\sum_{i \in \mathcal H}  g_i^{t-1}\|^2] \notag \\
    \leq &~\frac{3\rho\delta\sigma^2}{m}\left(1+\frac{1}{(1-\delta)n}\right)+3\rho\delta\zeta^2. \notag
    \end{align}

    Substituting \eqref{tildes-bars}--\eqref{eq:aa55} into \eqref{s-bars} yields
    \begin{align}
    &\quad~\E [\left\|\hat s^{t} - \bar s^{t}\right\|^2] \\
    &\leq (1-\alpha)\beta (\theta+\beta) \E[\|\hat s^{t-1} - \bar s^{t-1}\|^2] \notag  \\
    &\quad + \frac{3\rho\delta\sigma^2}{m}\left(1+\frac{1}{(1-\delta)n}\right)\left(\frac{\alpha\theta^2}{1-\beta^2}+(1-\alpha)\theta(\theta+\beta)+\frac{\alpha \beta^{2t}m}{m_0}\right)\notag \\
    &\quad +3\rho\delta\zeta^2\left(\frac{\alpha\theta(\theta+\beta)}{1-\beta(\theta+\beta)}+(1-\alpha)\theta(\theta+\beta)+\alpha\beta^t(\theta+\beta)^tm\right)\notag \\
    &\leq \frac{3\rho\delta\sigma^2}{m}\left(1+\frac{1}{(1-\delta)n}\right)\left(\frac{1}{1-(1-\alpha)\beta(\theta+\beta)}\frac{\alpha\theta^2}{1-\beta^2}+\frac{(1-\alpha)\theta(\theta+\beta)}{1-(1-\alpha)\beta(\theta+\beta)}+\frac{\alpha\beta^2m}{(1-\beta^2)m_0}\right)\notag \\
    &\quad +3\rho\delta\zeta^2\left(\frac{1}{1-(1-\alpha)\beta(\theta+\beta)}\frac{\alpha\theta(\theta+\beta)}{1-\beta(\theta+\beta)}+\frac{(1-\alpha)\theta(\theta+\beta)}{1-(1-\alpha)\beta(\theta+\beta)}+1\right)\notag \\
    &\leq \frac{3\rho\delta\sigma^2}{m}\left(1+\frac{1}{(1-\delta)n}\right)\left(\chi_1+\chi_2+\frac{\alpha\beta^2m}{(1-\beta^2)m_0}\right)+3\rho\delta\zeta^2\left(\chi_3+\chi_2+1\right). \notag
    \label{ress}
    \end{align}
    Here, the second inequality uses
    \begin{align*}
    \sum_{\tau=1}^{t-1} (1-\alpha)^\tau\beta^\tau(\theta+\beta)^\tau \leq \frac{1}{1-(1-\alpha)\beta(\theta+\beta)}\,,
    \end{align*}
    \begin{align*}
        \sum_{\tau=1}^{t} \beta^{2(t-\tau)}(1-\alpha)^\tau\beta^\tau(\theta+\beta)^\tau &\leq \beta^{2t} \sum_{\tau=1}^{t} (1-\alpha)^\tau\beta^{-\tau}(\theta+\beta)^\tau \\&= \beta^{2t} \sum_{\tau=1}^{t} \chi_0^{\tau}
        \leq \beta^{2t} \sum_{\tau=1}^{t} \beta^{-2\tau}\leq \frac{\beta^2}{1-\beta^2}\,,
    \end{align*}
    and
    \begin{align*}
        \sum_{\tau=1}^{t} \beta^{t-\tau}(\theta+\beta)^{t-\tau}(1-\alpha)^\tau\beta^\tau(\theta+\beta)^\tau &\leq \beta^{t}(\theta+\beta)^t \sum_{\tau=1}^{t} (1-\alpha)^\tau \leq \frac{1}{\alpha}.
    \end{align*}
    Substituting \eqref{tildes-bars}-\eqref{ress} into \eqref{delta1} and thanks to \eqref{alphabeta}, we obtain the desired result.
\end{proof}

\subsection{Bound of stochastic bias \(\E[\|\Delta_2^t\|^2]\)}\label{proof:Delta2}
\begin{lemma}\label{le:Delta2}
    Consider Algorithm \ref{subalgo} with the step size $\eta=\frac{1}{L}$. If Assumption \ref{ass:u} holds, then we have
    \begin{align}
    \E[\|\Delta_2^t\|^2] \leq \frac{\theta^2\sigma^2}{L^2|\mathcal H|m}.
    \end{align}
\end{lemma}
\begin{proof} {It follows from $\eta=\frac{1}{L}$ that}
    \begin{equation*}
        \begin{aligned}
            \E[\|\Delta_2^t\|^2] =& \frac{1}{\tilde L^2}\E\left[\left\|\nabla f(y^{t-1}) - \frac{1}{|\mathcal H|m}\sum_{i \in \mathcal H}\sum_{l=1}^m\nabla F(y^{t-1};\xi^{(t-1,l)}_i) \right\|^2\right]\\
            =& \frac{1}{\tilde L^2}\E\left[\left\|\frac{1}{|\mathcal H|}\sum_{i \in \mathcal H}\nabla f_i(y^{t-1}) - \frac{1}{|\mathcal H|m}\sum_{i \in \mathcal H}\sum_{j=1}^m\nabla F(y^{t-1};\xi^{(t-1,l)}_i) \right\|^2\right]\\
            =& \frac{1}{\tilde L^2|\mathcal H|^2m^2}\sum_{i \in \mathcal H}\sum_{j=1}^m\E\left[\left\|\nabla f_i(y^{t-1}) - \nabla F(y^{t-1};\xi^{(t-1,l)}_i) \right\|^2\right] \leq  \frac{\sigma^2}{\tilde L^2|\mathcal H|m},
        \end{aligned}
    \end{equation*}
    which completes the proof.
\end{proof}

\subsection{Proof of Theorem \ref{thm:sc-restart}}\label{proof:sc-restart}
\begin{proof}
For notational convenience, in this proof we define \(z_p = z(p)\), \(\epsilon_p = \epsilon(p)\), \(T_p = T(p)\), and \(m_p = m(p)\) for all \(p \geq 1\). We consider two cases, \(\sigma^2 = 0\) and \(\sigma^2 \neq 0\).

The proof sketch is as follows: Algorithm \ref{algo:restart} performs $P$ calls of Algorithm \ref{subalgo} to achieve $\E[f(z_P) - f^*] \leq \epsilon_P^2 + O(\rho \delta \zeta^2) \leq \epsilon^2/2L + O(\rho \delta \zeta^2)$. In the \(p\)-th call, Algorithm \ref{subalgo} starts from a point $z_{p-1}$ satisfying $\E[f(z_{p-1}) - f^*] \leq \epsilon_{p-1}^2 + O(\rho \delta \zeta^2)$ and then generates a new point $z_p$ satisfying $\E[f(z_p) - f^*] \leq \epsilon_{p}^2 + O(\rho \delta \zeta^2)$ with $\epsilon_{p}^2 = \epsilon_{p-1}^2/2$, ensuring the desired convergence.

To be specific, in the first call, we set
\begin{align}\label{eps1}
\epsilon_1^2 = \frac{32}{\mu}\left(3 \rho \delta (1+\frac{1}{(1 - \delta)n})+\frac{1}{(1 - \delta)n}\right)\sigma^2.
\end{align}
We run Algorithm \ref{subalgo} with $m_1=1$ for $T_1$ iterations, where
\begin{align}\label{t1m1}
T_1 = \min \left\{\lceil 2\sqrt{\kappa}\log \frac{2LR^2}{\epsilon_1^2} \rceil, \lceil 2\sqrt{\kappa}\log \frac{4L^2R^2}{\epsilon^2} \rceil \right\}.
\end{align}

When \(\epsilon_1^2 \leq \frac{\epsilon^2}{2L}\), we have
\[
T_1 = \lceil 2\sqrt{\kappa}\log \frac{4L^2R^2}{\epsilon^2} \rceil ~~{\rm and}~~ P = \max \left\{\left\lceil \log_2 \frac{4L\epsilon_1^2}{\epsilon^2} \right\rceil, 1\right\}=1.
\]
Therefore, by Theorem \ref{thm:sc} with $\alpha = 0$ and $\theta = 1$, we obtain
\[
\E[\|\nabla f(z_1)\|] \leq \E[\sqrt{2L(f(z_1)-f^*)}] \leq 2\sqrt{6}\kappa^{1/2}\rho^{1/2}\delta^{1/2}\zeta + \epsilon.
\]
It means that \( z_1 \) is the desired \( \tilde{x}^K \) that satisfies \eqref{nverr-rs}, and the required oracle query complexity is \( m_1 T_1 \).

Otherwise, when \(\epsilon_1^2 > \frac{\epsilon^2}{2L}\), by Theorem \ref{thm:sc} with $\alpha = 0$ and $\theta = 1$, we have $\chi_4=\chi_5=1$ and hence $\E[f(z_1) - f^*] \leq \epsilon_1^2 + 48\mu^{-1} \rho \delta \zeta^2 $.
Now we set $\epsilon_2^2 = \epsilon_1^2/2$ and increase the batch size to $m_2 = 2m_1$. Then using Theorem \ref{thm:sc} again, we have
\begin{align*}
\E[f(z_2) - f^*] &\leq 2\left(1 - \frac{1}{2\sqrt{\kappa}}\right)^{T_2}\E[f(z_1)-f^*]\\
&\quad ~+16\mu^{-1}\left(\frac{3\rho\delta\sigma^2}{{m_2}}(1+\frac{1}{(1-\delta)n})+\frac{\sigma^2}{(1-\delta)n{m_2}}+3\rho\delta\zeta^2\right)\\
&\leq2\left(1 - \frac{1}{2\sqrt{\kappa}}\right)^{T_2}(\epsilon_1^2 + 48\mu^{-1}\rho \delta \zeta^2)+\frac{\epsilon_2^2}{2} + 48\mu^{-1}\rho \delta \zeta^2.
\end{align*}
With $T_2 = \lceil 2\sqrt{\kappa}\log 8 \rceil$, we have $\E[f(z_2) - f^*] \leq \epsilon_2^2 + 48(1+ {4}^{-1})\mu^{-1}\rho \delta \zeta^2$. Similarly, in the $p$-th call, by setting $\epsilon_p^2 = \epsilon_{p-1}^2/2$, $m_p = 2m_{p-1}=2^{p-1}$ and $T_p = \lceil 2\sqrt{\kappa}\log 8 \rceil$, we have
\begin{align}\label{f:redisna}
\E[f(z_p) - f^*] \leq \epsilon_p^2 + 48\sum_{j=1}^p 4^{1-p}\mu^{-1}\rho \delta \zeta^2 \leq \epsilon_p^2 + 64\mu^{-1}\rho \delta \zeta^2.
\end{align}
Letting $p=P$ in \eqref{f:redisna}, we have
\begin{align*}
    \E[\|\nabla f(z_P)\|] &\leq \E[\sqrt{2L(f(z_P)-f^*)}]\\
    &\leq \sqrt{128\kappa\rho\delta\zeta^2 + \epsilon^2}\\
    &\leq 8\sqrt{2}\kappa^{1/2}\rho^{1/2}\delta^{1/2}\zeta + \epsilon,
\end{align*}
Moreover, the oracle query complexity of the $p$-th ($p\ge2$) call is upper bounded by $m_pT_p = 2^{p-1}\lceil 2\sqrt{\kappa}\log 8 \rceil$, which results in the overall oracle query complexity
\[
O\left(2\sqrt{\kappa}\log \frac{4L^2R^2}{\epsilon^2} + \sum_{p=2}^P 2^{p-1}\lceil 2\sqrt{\kappa}\log 8 \rceil\right) ~~~{\rm with}~~~ P = \left\lceil \log_2 \frac{4L\epsilon_1^2}{\epsilon^2} \right\rceil
\]
to reach $\epsilon_P^2 \leq \epsilon^2/2L$ thanks to $\epsilon^2 < \epsilon_1^2L$.
With elementary calculations, the overall oracle query complexity can be rewritten as
\begin{align}\label{complexity:redisna}
O\left(2\sqrt{\kappa}\log \frac{4L^2R^2}{\epsilon^2} + 128\kappa^{3/2}\left(3 \rho \delta (1+\frac{1}{(1 - \delta)n})+\frac{1}{(1 - \delta)n}\right)\frac{\sigma^2}{\epsilon^2}\right),
\end{align}
which completes the proof.
\end{proof}

\subsection{Proof of Theorem~\ref{thm:nc1}}
\label{section:Theorem 16}
To prove Theorem \ref{thm:nc1}, we summarize the update rules in Algorithm \ref{subalgo} as
\begin{align*}
    \hat s^t&= \alpha {\mathsf A}(\{\beta s_i^{t-1}+\theta g_i^{t-1}\}_{i=1}^n) + (1-\alpha)\beta s^{t-1} + (1-\alpha)\theta {\mathsf A}(\{g_i^{t-1}\}_{i=1}^n),\\
    x^t &= x^{t-1} - \eta \hat s^t,\\
    y^t &= x^t + \beta(x^t-x^{t-1}),
\end{align*}
which further indicate that
\begin{align*}
    y^t &= x^t + \beta(x^t-x^{t-1}) = x^t - \eta \beta \hat s^t\\
    & = x^{t-1} - \eta \hat s^t - \eta \beta \hat s^t\\
    & = x^{t-1} - \eta \beta \hat s^{t-1} + \eta \beta \hat s^{t-1} - \eta \hat s^t - \eta \beta \hat s^t\\
    & = y^{t-1} - \eta ((1+\beta)\hat s^t - \beta \hat s^{t-1}) = y^{t-1} - \eta p^t,
\end{align*}
with $p^t:=( 1+\beta)\hat s^t - \beta \hat s^{t-1}$.
Before proving Theorem \ref{thm:nc1}, we give two auxiliary lemmas.
    \begin{lemma}\label{lem:sgdm-byz-descent}
        Consider Algorithm \ref{subalgo} and suppose that Assumption \ref{ass:basic} holds. Setting {$\beta \in [0,1]$} and $\eta \leq \frac{1}{L}$, for any $t\geq 1$ we have
        \begin{align}\label{sgdm-byz-descent}
            \E[f( y^t)] \leq f( y^{t-1}) - \frac{\eta}{2}\|\nabla f( y^{t-1})\|^2 + \eta\E[\|\bar e^t\|^2] + \eta\E[\| p^t - \bar p^t\|^2]\,.
        \end{align}
        where $\bar e^t:=\bar p^t - \nabla f( y^{t-1}) $ and $\bar p^t := (1+\beta)\bar s^t - \beta \bar s^{t-1}$.
    \end{lemma}
    \begin{proof}
        By the smoothness of the function $f$ and the server update, we have
        \begin{align*}
            f( y^t) & \leq f( y^{t-1}) - \eta \langle \nabla f( y^{t-1}), p^t\rangle + \frac{L \eta^2}{2} \| p^t\|^2                                   \\
                     & \leq f( y^{t-1}) - \eta \langle \nabla f( y^{t-1}), p^t\rangle + \frac{\eta}{2} \| p^t\|^2                                   \\
                     & = f( y^{t-1}) + \frac{\eta}{2} \| p^t - \nabla f( y^{t-1})\|^2 - \frac{\eta}{2}\|\nabla f( y^{t-1})\|^2               \\
                     & = f( y^{t-1}) + \frac{\eta}{2} \| p^t - \bar p^t + \bar p^t - \nabla f( y^{t-1}) \|^2 - \frac{\eta}{2}\|\nabla f( y^{t-1})\|^2 \\
                     & \leq f( y^{t-1}) + \eta\|\bar e^t\|^2 + \eta \| p^t - \bar p^t\|^2 - \frac{\eta}{2}\|\nabla f( y^{t-1})\|^2.
        \end{align*}
        Taking conditional expectations on both sides yields \eqref{sgdm-byz-descent}.
    \end{proof}

    \begin{lemma}\label{lem:sgdm-byz-error}
        Consider Algorithm \ref{subalgo} and suppose that Assumptions \ref{ass:basic}--\ref{ass:u} hold. Setting $\beta$, $\theta$ and $\eta$ such that $\beta = 1- 12L\eta \geq \frac{1}{2}$, \(0\leq \theta \leq 2\) and $(1-\theta-\beta)^2 \leq \chi_6 (1-\beta)^2$, for any $t \geq 2$ we have
        \begin{align}\label{sgdm-byz-error}
            \E \|\bar e^{t}\|^2 \leq & \beta\E\|\bar e^{t-1}\|^2 + \frac{3(1-\beta)}{8} \|p^{t-1} - \bar p^{t-1}\|^2  \\
                                & +(12\chi_6+\frac{3}{8})(1-\beta)\|\nabla f( y^{t-2})\|^2 + \frac{5\theta^2 \sigma^2}{(1-\delta)nm}. \notag
        \end{align}
    \end{lemma}
    \begin{proof}
        According to \eqref{update:s} and the definition of $\bar p^t := (1+\beta)\bar s^t - \beta \bar s^{t-1}$, we have
        \begin{align*}
            \bar p^t - \beta \bar p^{t-1} &= (1+\beta)\bar s^t - \beta \bar s^{t-1} - \beta \bar p^{t-1}\\
            & = \beta \bar s^t + \theta \bar g^{t-1} - \beta \bar p^{t-1}\\
            & = \theta \bar g^{t-1} + \beta (\beta \bar s^{t-1} + \theta \bar g^{t-1}) - \beta \bar p^{t-1}\\
            & = \theta \bar g^{t-1} + \beta (\beta \bar s^{t-1} + \theta \bar g^{t-1}) - \beta (\beta \bar s^{t-1} + \theta \bar g^{t-2})\\
            & = \theta \left(\bar g^{t-1} + \beta(\bar g^{t-1} - \bar g^{t-2}) \right),
        \end{align*}
        where \(\bar g^t = \frac{1}{|\mathcal H|}\sum_{i \in \mathcal H}g_i^t\). Therefore, according to the definition of $\bar s$, we have
        \begin{align*}
            \E \|\bar e^{t}\|^2 & =  \E \|\bar p^t - \nabla f( y^{t-1})\|^2                                                                                \\
            & = \E \|\beta \bar p^{t-1} + \theta \left( \bar g^{t-1} + \beta(\bar g^{t-1} - \bar g^{t-2}) \right) - \nabla f( y^{t-1})\|^2                                                       \\
            & = \E \|\beta (\bar p^{t-1} - \nabla f( y^{t-2})) + \theta(1+\beta)(\bar g^{t-1} - \nabla f( y^{t-1})) - \theta\beta (\bar g^{t-2} - \nabla f( y^{t-2}) )\\
            &\quad~ -\beta(1-\theta)(\nabla f( y^{t-1}) - \nabla f( y^{t-2})) + (\theta+\beta-1)\nabla f( y^{t-1}) \|^2
            \\
            & \leq \beta^2 (1+\frac{1-\beta}{2}) \| \bar p^{t-1} - \nabla f( y^{t-2})\|^2 + \theta^2(1+\beta)^2\|\bar g^{t-1} - \nabla f( y^{t-1})\|^2\\
            &\quad~ + \theta^2\beta^2\|\bar g^{t-2} - \nabla f( y^{t-2})\|^2+ 2\beta^2(1-\theta)^2(1+\frac{2}{1-\beta})\|\nabla f( y^{t-1}) - \nabla f( y^{t-2})\|^2 \\
            &\quad~+2(1-\theta-\beta)^2(1+\frac{2}{1-\beta})\|\nabla f( y^{t-1}) \pm \nabla f( y^{t-2})\|^2
            \\
            & \leq \beta^2 (1+\frac{1-\beta}{2}) \| \bar p^{t-1} - \nabla f( y^{t-2})\|^2 + \theta^2(1+\beta)^2\|\bar g^{t-1} - \nabla f( y^{t-1})\|^2\\
            &\quad~ + \theta^2\beta^2\|\bar g^{t-2} - \nabla f( y^{t-2})\|^2+\frac{12}{1-\beta}(1-\theta-\beta)^2\|\nabla f( y^{t-2})\|^2 \\
            &\quad~+ \frac{6}{1-\beta}(\beta^2(1-\theta)^2+2(1-\theta-\beta)^2)\|\nabla f( y^{t-1}) - \nabla f( y^{t-2})\|^2
            \\
            & \leq \beta^2 (1+\frac{1-\beta}{2})\E \|\bar e^{t-1}\|^2 + \frac{5\theta^2 \sigma^2}{(1-\delta)nm}+\frac{12}{1-\beta}(1-\theta-\beta)^2\|\nabla f( y^{t-2})\|^2
            \\
            &\quad ~ + \frac{6L^2\eta^2}{1-\beta}(\beta^2(1-\theta)^2+2(1-\theta-\beta)^2)\|p^{t-1} \pm \bar p^{t-1} \pm \nabla f( y^{t-2})\|^2\\
            & \leq \beta^2 (1+\frac{1-\beta}{2})\E \|\bar e^{t-1}\|^2 + \frac{5\theta^2 \sigma^2}{(1-\delta)nm}+\frac{12}{1-\beta}(1-\theta-\beta)^2\|\nabla f( y^{t-2})\|^2
            \\
            &\quad ~ + \frac{6L^2\eta^2}{1-\beta}(\beta^2(1-\theta)^2+2(1-\theta-\beta)^2)(3\|p^{t-1} - \bar p^{t-1}\|^2 + 3\| \bar e^{t-1}\|^2 + 3\|\nabla f( y^{t-2})\|^2)\\
            & \leq \beta^2\left(\frac{3-\beta}{2}+\frac{54L^2\eta^2}{1-\beta}\right)\E \|\bar e^{t-1} \|^2+ \frac{5\theta^2 \sigma^2}{(1-\delta)nm}
            + \frac{54L^2\eta^2}{1-\beta}\|p^{t-1} - \bar p^{t-1}\|^2
            \\
            &\quad ~ +\frac{6}{1-\beta}\left(2\chi_6 (1-\beta)^2+9L^2\eta^2\right)\|\nabla f( y^{t-2})\|^2
            \\
            & \leq \beta\E\|\bar e^{t-1}\|^2 + \frac{3(1-\beta)}{8} \|p^{t-1} - \bar p^{t-1}\|^2    +(12\chi_6+\frac{3}{8})(1-\beta)\|\nabla f( y^{t-2})\|^2
              + \frac{5\theta^2 \sigma^2}{(1-\delta)nm}\,.
        \end{align*}
        Here, the second to last inequality uses
        \[
        (1-\theta)^2 \leq 1,~~(1-\theta-\beta)^2 \leq \min\{\beta^2,\chi_6 (1-\beta)^2\},~~\frac{1}{2}\leq \beta \leq 1.
        \]
        The last inequality is due to \(\beta = 1 - 12 L\eta\). This completes the proof.
    \end{proof}

\begin{proof}{\bf of Theorem \ref{thm:nc1}}
    Multiplying \eqref{sgdm-byz-error} by \(\frac{\eta}{1-\beta}\) and adding it to \eqref{sgdm-byz-descent} yield
    \begin{align}
    \label{lyapunov}
    \E[f( y^t)] + \frac{\eta\beta}{1-\beta}\E[\|\bar e^t\|^2] &\leq f( y^{t-1}) - \frac{\eta}{2}\|\nabla f( y^{t-1})\|^2 + \eta\E[\| p^t - \bar p^t\|^2] \\
    &\quad ~ + \frac{\eta\beta}{1-\beta}\E[\|\bar e^{t-1}\|^2] + \frac{3\eta}{8} \E[\|p^{t-1} - \bar p^{t-1}\|^2 ] \notag \\
    &\quad ~ +(12\chi_6+\frac{3}{8})\eta\|\nabla f( y^{t-2})\|^2
              + \frac{5\theta^2\eta}{1-\beta} \frac{\sigma^2}{(1-\delta)nm}. \notag
    \end{align}
    For the term of \(\E[\|p^t-\bar p^t\|^2]\) in \eqref{lyapunov}, applying \eqref{ress}, we have
    \begin{align}
    \label{eq:aa65}
        \E[\|p^t-\bar p^t\|^2] &= \E[\|(1+\beta)\hat s^t - \beta \hat s^{t-1} - (1+\beta)\bar s^t + \beta \bar s^{t-1} \|^2]\\
        &\leq 2(1+\beta)^2 \E[\|\hat s^t - \bar s^t\|^2] + 2\beta^2 \E[\|\hat s^{t-1} - \bar s^{t-1}\|^2] \notag \\
        &\leq \frac{30\rho\delta\sigma^2}{m}\left(1+\frac{1}{(1-\delta)n}\right)\left(\chi_1+\chi_2+\frac{\alpha\beta^2m}{(1-\beta^2)m_0}\right)+30\rho\delta\zeta^2\left(\chi_3+\chi_2+1\right). \notag
    \end{align}
    Substituting \eqref{eq:aa65} into \eqref{lyapunov} and rearranging the terms yield
    \begin{align}
    \label{lyapunov-new}
    &\quad ~ \E[f( y^t)-f^*] + \frac{\eta\beta}{1-\beta}\E[\|\bar e^t\|^2] + (12\chi_6+\frac{3}{8})\eta\E[\|\nabla f( y^{t-1})\|^2]\\
    &\leq \E[f( y^{t-1})-f^*] + \frac{\eta\beta}{1-\beta}\E[\|\bar e^{t-1}\|^2] +(12\chi_6+\frac{3}{8})\eta\E[\|\nabla f( y^{t-2})\|^2] \notag \\
    &\quad ~- \left(\frac{1}{8}-12\chi_6\right)\eta\|\nabla f( y^{t-1})\|^2 + \frac{5\theta^2\eta}{1-\beta} \frac{\sigma^2}{(1-\delta)nm} \notag \\
    &\quad ~  + \frac{42\eta\rho\delta\sigma^2}{m}\left(1+\frac{1}{(1-\delta)n}\right)\left(\chi_1+\chi_2+\frac{\alpha\beta^2m}{(1-\beta^2)m_0}\right)+42\eta\rho\delta\zeta^2\left(\chi_3+\chi_2+1\right). \notag
    \end{align}
    Summing over $t$ from $2$ to $T$ and also rearranging the terms, we have
    \begin{align}
    \label{eq:aa67}
        &\quad~\sum_{t=2}^T \left(\frac{1}{8}-12\chi_6\right)\|\nabla f( y^{t-1})\|^2\\
        &\leq  \frac{1}{\eta}\E[f( y^1)-f^*] + \frac{\beta}{1-\beta}\E[\|\bar e^1\|^2] +(12\chi_6+\frac{3}{8})\E[\|\nabla f( y^0)\|^2] \notag \\
        &\quad ~ +\sum_{t=2}^T\left(\frac{5\theta^2}{(1-\beta)(1-\delta)n}+42\rho\delta\left(1+\frac{1}{(1-\delta)n}\right)\left(\chi_1+\chi_2+\frac{\alpha\beta^2m}{(1-\beta^2)m_0}\right)\right)  \frac{\sigma^2}{m} \notag \\
         &\quad ~ +\sum_{t=2}^T 42\rho\delta\zeta^2\left(\chi_3+\chi_2+1\right). \notag
    \end{align}
    For the second line in \eqref{eq:aa67}, using Lemmas \ref{lem:sgdm-byz-descent} and \ref{lem:sgdm-byz-error} with $t = 1$ yields
    \begin{align}
    \label{eq:aa68}
        &\quad~\frac{1}{\eta}\E[f( y^1)-f^*] + \frac{\beta}{1-\beta}\E[\|\bar e^1\|^2] +(12\chi_6+\frac{3}{8})\E[\|\nabla f( y^0)\|^2]\\
        &\leq  \frac{1}{\eta}\left(f( y^0)-f^*\right) + \left(1+\frac{\beta}{1-\beta}\right)\E[\|\bar e^1\|^2] +\E[\|p^1 - \bar p^1\|^2] +(12\chi_6-\frac{1}{8})\E[\|\nabla f( y^0)\|^2] \notag \\
        &\leq  \frac{1}{\eta}\Delta + \frac{(\theta+\beta^2+\theta\beta)^2}{1-\beta}\frac{\sigma^2}{(1-\delta)nm}+\E[\|p^1 - \bar p^1\|^2] \notag \\
        &\quad+(12\chi_6-\frac{1}{8}+(\theta+\beta^2+\theta\beta-1)^2)\E[\|\nabla f( y^0)\|^2]. \notag
    \end{align}
    Substituting \eqref{eq:aa68} into \eqref{eq:aa67}, with $\frac{1}{8}-12\chi_6-(\theta+\beta^2+\theta\beta-1)^2 > \chi_7$, we have
    \begin{align*}
        &\quad~\sum_{t=1}^T \chi_7\|\nabla f( y^{t-1})\|^2 \\
        &\leq \sum_{t=1}^T \left(\frac{1}{8}-12\chi_6-(\theta+\beta^2+\theta\beta-1)^2\right)\|\nabla f( y^{t-1})\|^2\\
        &\leq  \frac{1}{\eta}\Delta + \frac{(\theta+\beta^2+\theta\beta)^2}{1-\beta}\frac{\sigma^2}{(1-\delta)nm}+\sum_{t=1}^T 42\rho\delta\zeta^2\left(\chi_3+\chi_2+1\right)\\
        &\quad ~ +\sum_{t=1}^T\left(\frac{5\theta^2}{(1-\beta)(1-\delta)n}+42\rho\delta\left(1+\frac{1}{(1-\delta)n}\right)\left(\chi_1+\chi_2+\frac{\alpha\beta^2m}{(1-\beta^2)m_0}\right)\right)  \frac{\sigma^2}{m},
    \end{align*}
    where the last inequality comes from \eqref{eq:aa67} and \eqref{eq:aa68}. Consequently, it holds that
    \begin{align}
    \label{result-detail}
    &\quad~\frac{1}{T}\sum_{t=1}^T\|\nabla f( y^{t-1})\|^2 \\
    &\leq \frac{1}{\chi_7}\Bigg(\frac{1}{\eta T}\Delta+\frac{1}{T}\frac{(\theta+\beta^2+\theta\beta)^2}{1-\beta}\frac{\sigma^2}{(1-\delta)nm}+42\rho\delta\zeta^2\left(\chi_3+\chi_2+1\right) \notag \\
    &\quad~ +\left(\frac{5\theta^2}{(1-\beta)(1-\delta)n}+42\rho\delta\left(1+\frac{1}{(1-\delta)n}\right)\left(\chi_1+\chi_2+\frac{\alpha\beta^2m}{(1-\beta^2)m_0}\right)\right)  \frac{\sigma^2}{m}\Bigg). \notag
    \end{align}
    Setting $1-\beta = 12L\eta$, since $\theta^2 \leq 2(1-\theta-\beta)^2+2(1-\beta)^2 \leq (2\chi_6+2)(1-\beta)^2$, $\theta+\beta^2+\theta\beta \leq 2$, $\chi_1+\chi_2 = O(L\eta)$, $\chi_3 = O(1) \geq 0$, $m_0=m/(L^2\eta^2)$, and $\chi_7 = \Theta (1) > 0$, we have
    \begin{align*}
    \frac{1}{T}\sum_{t=1}^T\|\nabla f( y^{t-1})\|^2 &\leq O \Bigg(\frac{1}{\eta T}(\Delta+\frac{\sigma^2}{L(1-\delta)nm})\\
    &\quad~ +\eta\left(\frac{1}{(1-\delta)n}+\rho\delta\left(1+\frac{1}{(1-\delta)n}\right)\right)  \frac{L\sigma^2}{m} +\rho\delta\zeta^2 \Bigg).
    \end{align*}
    Setting the step size
    \begin{align}\label{stepsize}
    \eta = \min \left( \sqrt{\frac{\Delta+\frac{\sigma^2}{L(1-\delta)nm}}{T\left(\frac{1}{(1-\delta)n}+\rho\delta\left(1+\frac{1}{(1-\delta)n}\right)\right)  \frac{L\sigma^2}{m}}}, \frac{1}{24L} \right),
    \end{align}
    with $m=O(1)$, we obtain
    \begin{align*}
    \frac{1}{T}\sum_{t=1}^T\|\nabla f( y^{t-1})\|^2 &\leq O \Bigg( \sqrt{\frac{L\Delta+\sigma^2/n}{T}}\sqrt{\left(\frac{1}{(1-\delta)n}+\rho\delta\right)  \sigma^2}\\
    &\quad~ + \frac{L\Delta}{T}+\frac{\sigma^2}{Tn} +\rho\delta\zeta^2 \Bigg),
    \end{align*}
    which completes the proof.
\end{proof}

\subsection{Proof of Theorem \ref{thm:nc}}
\label{section:Theorem 18}
\begin{proof}
    For any $\gamma=1,\cdots,\Gamma$, define an auxiliary variable $\varkappa^{\gamma *} = \arg\min f^\gamma(\varkappa)$. In view of the strong convexity of $f^\gamma$, we have
        \begin{align}
        \label{eq:fx}
            f^\gamma(\varkappa) &\geq f^\gamma(\varkappa^{\gamma *}) + \frac{L}{2}\|\varkappa - \varkappa^{\gamma *}\|^2 \\
            &= f^\gamma(\varkappa^\gamma) + f^\gamma(\varkappa^{\gamma *}) - f^\gamma(\varkappa^\gamma) + \frac{L}{2}\|x - \varkappa^{\gamma *}\|^2 \notag \\
            &= f(\varkappa^\gamma) + L\|\varkappa^\gamma-\varkappa^{\gamma-1}\|^2 + f^\gamma(\varkappa^{\gamma *}) - f^\gamma(\varkappa^\gamma) + \frac{L}{2}\|x - \varkappa^{\gamma *}\|^2. \notag
        \end{align}

    Setting $x = \varkappa^\gamma$ in \eqref{eq:fx} and summing it from $\gamma=1,\cdots,\Gamma$, we obtain
    \begin{align}
    \label{eq:aa69}
        \frac{L}{2} \sum_{\gamma=1}^\Gamma \|\varkappa^\gamma - \varkappa^{\gamma *}\|^2 \leq \sum_{\gamma=1}^\Gamma \left( f^\gamma(\varkappa^\gamma) - f^\gamma(\varkappa^{\gamma *}) \right).
    \end{align}
    Similarly, setting $x = \varkappa^{\gamma-1}$ in \eqref{eq:fx} and using $f^\gamma(\varkappa^{\gamma-1}) = f(\varkappa^{\gamma-1})$ yield
    \[
        f(\varkappa^{\gamma-1}) \geq f(\varkappa^\gamma) + L\|\varkappa^\gamma-\varkappa^{\gamma-1}\|^2 + f^\gamma(\varkappa^{\gamma *}) - f^\gamma(\varkappa^\gamma) + \frac{L}{2}\|\varkappa^{\gamma-1} - \varkappa^{\gamma *}\|^2.
    \]
    Summing the above inequality from $\gamma=1,\cdots,\Gamma$, we obtain
    \begin{align}
    \label{eq:aa70}
        \hspace{-1em} L\sum_{\gamma=1}^\Gamma \left(\|\varkappa^\gamma - \varkappa^{\gamma-1}\|^2 + \frac{1}{2}\|\varkappa^{\gamma-1} - \varkappa^{\gamma *}\|^2\right) \leq f(\varkappa^0)-f(\varkappa^\gamma) +  \sum_{\gamma=1}^\Gamma \left( f^\gamma(\varkappa^\gamma) - f^\gamma(\varkappa^{\gamma *}) \right).
    \end{align}

    Due to the optimality of $\varkappa^{\gamma *}$, we have
    \[
        0 = \nabla f^\gamma(\varkappa^{\gamma *}) = \nabla f(\varkappa^{\gamma *}) + 2L(\varkappa^{\gamma *} - \varkappa^{\gamma-1}),
    \]
    which, together with \eqref{eq:aa69} and \eqref{eq:aa70}, further implies
        \begin{align}
        \label{eq:sumg}
            \sum_{\gamma=1}^\Gamma \|\nabla f(\varkappa^\gamma)\|^2 &= \sum_{\gamma=1}^\Gamma \|\nabla f(\varkappa^\gamma) - \nabla f(\varkappa^{\gamma *})+\nabla f(\varkappa^{\gamma *})\|^2\\
            &\leq 2\sum_{\gamma=1}^\Gamma\|\nabla f(\varkappa^\gamma) - \nabla f(\varkappa^{\gamma *})\|^2+2\sum_{\gamma=1}^\Gamma\|\nabla f(\varkappa^{\gamma *})\|^2 \notag \\
            &\leq 2L^2\sum_{\gamma=1}^\Gamma\|\varkappa^\gamma - \varkappa^{\gamma *}\|^2+8L^2\sum_{\gamma=1}^\Gamma\|\varkappa^{\gamma *} - \varkappa^{\gamma-1}\|^2 \notag \\
            &\leq 20L\sum_{\gamma=1}^\Gamma \left( f^\gamma(\varkappa^\gamma) - f^\gamma(\varkappa^{\gamma *}) \right) +16L\left(f(\varkappa^0)-f(\varkappa^\gamma)\right). \notag
        \end{align}

    Now we bound the term $f^\gamma(\varkappa^\gamma) - f^\gamma(\varkappa^{\gamma *})$ in \eqref{eq:sumg} by induction. Following the analysis from \eqref{f:redisna} to \eqref{complexity:redisna}, we know that to achieve
    \[
    \E[f^1(\varkappa^1) - f^1(\varkappa^{1*})]
    \leq \frac{\epsilon^2}{40L} + \frac{64\rho\delta\zeta^2}{L},
    \]
    the oracle query complexity of Byrd-reNester is at least
    \[
    S_1 \leq 2\sqrt{3}\log \frac{36L\Delta}{\epsilon^2} + 7680 \sqrt{3}\left(3 \rho \delta (1+\frac{1}{(1 - \delta)n})+\frac{1}{(1 - \delta)n}\right)\frac{\sigma^2}{\epsilon^2},
    \]
    since $\kappa = 3$ and Lipschitz constant of $f^1(\cdot)$ is $3L$. Suppose $\E[f^{\gamma-1}(\varkappa^{\gamma-1}) - f^{\gamma-1}(\varkappa^{(\gamma-1)*})] \leq \frac{\epsilon^2}{40L} + \frac{64\rho\delta\zeta^2}{L}$, we have
    \begin{align*}
        \E[f(\varkappa^{\gamma-1})] \leq \E [f^{\gamma - 1}(\varkappa^{\gamma-1}) ] \leq \E[f^{\gamma - 1}(\varkappa^{(\gamma-1)*})+\frac{\epsilon^2}{40L} + \frac{64\rho\delta\zeta^2}{L}] \leq \E[f(\varkappa^{\gamma - 2})+\frac{\epsilon^2}{40L} + \frac{64\rho\delta\zeta^2}{L}],
    \end{align*}
    and hence
    \[
        \E[f(\varkappa^{\gamma-1})] \leq f(\varkappa^0)+ (\gamma-1)\frac{\epsilon^2}{40L}+(\gamma-1)\frac{64\rho\delta\zeta^2}{L}.
    \]
    Using the analysis from \eqref{f:redisna} to \eqref{complexity:redisna} again, for any $\gamma=1,\cdots,\Gamma$ we obtain that
    \begin{align}
    \label{eq:aa72}
        &\E[f^\gamma(\varkappa^\gamma) - f^\gamma(\varkappa^{\gamma *})] \leq \frac{\epsilon^2}{40L} + \frac{64\rho\delta\zeta^2}{L}
    \end{align}
    with oracle query complexity
    \begin{align*}
    S_\gamma &\leq 2\sqrt{3}\log \frac{36L(f^\gamma(\varkappa^{\gamma-1})-f^\gamma(\varkappa^{\gamma *}))}{\epsilon^2} + 7680 \sqrt{3}\left(3 \rho \delta (1+\frac{1}{(1 - \delta)n})+\frac{1}{(1 - \delta)n}\right)\frac{\sigma^2}{\epsilon^2}\\
    &\leq 2\sqrt{3}\log \frac{36L(f^\gamma(\varkappa^{\gamma-1})-f^*)}{\epsilon^2} + 7680 \sqrt{3}\left(3 \rho \delta (1+\frac{1}{(1 - \delta)n})+\frac{1}{(1 - \delta)n}\right)\frac{\sigma^2}{\epsilon^2}\\
    &\leq 2\sqrt{3}\log \frac{36L(f(\varkappa^0)-f^*+ (\gamma-1)\frac{\epsilon^2}{40L}+(\gamma-1)\frac{64\rho\delta\zeta^2}{L})}{\epsilon^2} \\
    &\quad~+ 7680 \sqrt{3}\left(3 \rho \delta (1+\frac{1}{(1 - \delta)n})+\frac{1}{(1 - \delta)n}\right)\frac{\sigma^2}{\epsilon^2}.
    \end{align*}

    Combining \eqref{eq:sumg} and \eqref{eq:aa72}, we have
    \begin{align*}
        \frac{1}{\Gamma} \sum_{\gamma=1}^\Gamma \E[\|\nabla f(\varkappa^\gamma)\|^2]
        &\leq \frac{20L}{\Gamma}\sum_{\gamma=1}^\Gamma \E[ f^\gamma(\varkappa^\gamma) - f^\gamma(\varkappa^{\gamma *}) ]+\frac{16L}{\Gamma}\left(f(\varkappa^0)-f(\varkappa^\gamma)\right)\\
        &\leq \frac{\epsilon^2}{2} + 1280\rho\delta\zeta^2 +\frac{16L}{\Gamma}\left(f(\varkappa^0)-f^*\right).
    \end{align*}
    With $\Gamma = \lceil 32L\left(f(\varkappa^0)-f^*\right)\epsilon^{-2} \rceil$, the oracle query complexity is
    \[
    \sum_{\gamma=1}^\Gamma S_\gamma \leq \Gamma S_\Gamma =  O\left(\frac{L\Delta\rho\delta\sigma^2}{\epsilon^4}+\frac{L\Delta\sigma^2}{(1-\delta)n\epsilon^4}+\frac{L\Delta}{\epsilon^2}\log\frac{L\Delta(1+\rho\delta\zeta^2)}{\epsilon^2}\right),
    \]
    which completes the proof.
\end{proof}

\section{Equivalent form of Algorithm \ref{subalgo}}\label{appendix:Equivalentform}
\begin{algorithm}[ht]
    \caption{Equivalent form of Algorithm \ref{subalgo}}
    \begin{algorithmic}
    \STATE{Input: starting point $x^0$, auxiliary point $y^0=x^0$, maximum number of iterations $T$, batch size $m_0$, $m$, step size $\eta$, $\theta\in(0,1]$}, {$\beta\in[0,1)$}, {$\alpha\in[0,1]$, $\hat s^0 = s_i^0 = \frac{1}{m_0}\sum_{l = 1}^{m_0} \nabla F(y^{0};\xi^{(0,l)}_i)$.}
    \FOR{$t=1,\cdots,T$}
    \FOR{node $i \in \mathcal {H}$}
    \IF{\(\mod{(t,2)}=1\)}{
    \STATE{Independently sample $\{\xi^{(t-1,1)}_i,\cdots,\xi^{(t-1,m)}_i \}$, obtain stochastic gradients from oracle $\mathsf O \in \mathcal{O}$ and calculate
        \begin{equation*}
        \begin{aligned}
             w_i^t = g_i^{t-1} =  \frac{1}{m}\sum_{l = 1}^m \nabla F(y^{t-1};\xi^{(t-1,l)}_i),\quad
             s^t_i =  s^{t-1}_i .
        \end{aligned}
        \end{equation*}
    }}
    \ELSE{
    \STATE{\begin{equation*}
        \begin{aligned}
            w_i^t = s^t_i =  \beta s^{t-1}_i + \theta g_i^{t-2}.\quad\quad\quad\quad
        \end{aligned}
        \end{equation*}
    }    }
    \ENDIF
    \STATE{Send $w_i^t$ to server.}
    \ENDFOR
    \FOR{node $i \in \mathcal {B}$}
    \STATE{Send arbitrary vector $w_i^t \in \mathbb R^d$ to server.}
    \ENDFOR
    \FOR{server}
    \IF{\(\mod{(t,2)}=1\)}{
    \STATE{Receive \(\{w^t_i\}_{i=1}^n\) to update
    \begin{align*}
            w^t = {\mathsf A} (\{ w_i^t\}_n), \quad
            x^t = x^{t-1}.
    \end{align*}}}
    \ELSE{
    \STATE{Receive \(\{w^t_i\}_{i=1}^n\) and update
    \begin{align*}
            w^t &= \beta \hat w^{t-2} + \theta w^{t-1}, \\
            \hat w^t &= \alpha {\mathsf A} (\{w^t_i\}_n) + (1 - \alpha) w^t, \\
            x^t &= x^{t-1} - \eta \hat w^t,  \\
            y^t &= x^t + \beta (x^t - x^{t-1}).
    \end{align*} \vspace{-1em} }    }
    \ENDIF
    \ENDFOR
    \ENDFOR
    \RETURN $\tilde x^K = x^T$ for strongly convex optimization; $\tilde x^K = x^{t'}$ where $t'$ is randomly chosen from \(1,\cdots,T\) for non-convex optimization. Here $K$ is the number of oracle queries.
    \end{algorithmic}
\end{algorithm}

\end{document}